\documentclass[a4paper,10pt]{amsart}

\usepackage[colorlinks,linkcolor=blue,citecolor=blue]{hyperref}
\usepackage{latexsym, amssymb, amsmath, amsthm, mathrsfs, bbm}
\usepackage[all]{xy}
\usepackage{pstricks}
\usepackage{ifthen}
\usepackage{multirow}

\usepackage{anysize}\marginsize{22mm}{22mm}{30mm}{30mm}
\allowdisplaybreaks[4]

\newcommand{\Dchaintwo}[4]{
\rule[-3\unitlength]{0pt}{8\unitlength}
\begin{picture}(14,5)(0,3)
\put(1,2){\ifthenelse{\equal{#1}{l}}{\circle*{2}}{\circle{2}}}
\put(2,2){\line(1,0){10}}
\put(13,2){\ifthenelse{\equal{#1}{r}}{\circle*{2}}{\circle{2}}}
\put(1,5){\makebox[0pt]{\scriptsize #2}}
\put(7,4){\makebox[0pt]{\scriptsize #3}}
\put(13,5){\makebox[0pt]{\scriptsize #4}}
\end{picture}}

\def \To{\longrightarrow}
\def \dim{\operatorname{dim}}
\def \gr{\operatorname{gr}}

\def \Ker{\operatorname{Ker}}

\def \Aut{\operatorname{Aut}}
\def \id{\operatorname{id}}

\def \1{\mathbf{1}}
\def \C{\mathcal{C}}
\def \D{\Delta}

\def \e{\varepsilon}
\def \M{\mathrm{M}}
\def \MM{\mathbbm{M}}
\def \R{\mathbbm{R}}
\def \N{\mathbb{N}}
\def \S{\mathcal{S}}
\def \Z{\mathbb{Z}}
\def \unit{\mathbbm{1}}
\def \q{\mathbbm{q}}
\def \deg{\operatorname{deg}}
\def \m{\mathbbm{m}}
\def \n{\mathbbm{n}}
\def \g{\mathbbm{g}}

\numberwithin{equation}{section}
\numberwithin{table}{section}
\numberwithin{equation}{section}
\newtheorem{theorem}{Theorem}[section]
\newtheorem{lemma}[theorem]{Lemma}
\newtheorem{proposition}[theorem]{Proposition}
\newtheorem{corollary}[theorem]{Corollary}
\newtheorem{definition}[theorem]{Definition}
\newtheorem{example}[theorem]{Example}
\newtheorem{remark}[theorem]{Remark}
\newtheorem{remarks}[theorem]{Remarks}

\begin{document}
\title{Finite quasi-quantum groups of rank two}\thanks{$^\dag$Supported by SRFDP 20130131110001, NSFC 11371186 and SDNSF ZR2013AM022.}


\author{Hua-Lin Huang}
\address{School of Mathematics, Shandong University, Jinan 250100, China} \email{hualin@sdu.edu.cn}

\author{Gongxiang Liu}
\address{Department of Mathematics, Nanjing University, Nanjing 210093, China} \email{gxliu@nju.edu.cn}

\author{Yuping Yang}
\address{School of Mathematics, Shandong University, Jinan 250100, China}\email{yupingyang@mail.sdu.edu.cn}
\author{Yu Ye}
\address{School of Mathematics, University of Science and Technology of China, Hefei 230026, China} \email{yeyu@ustc.edu.cn}
\date{}
\maketitle

\begin{abstract}
This is a contribution to the structure theory of finite pointed quasi-quantum groups. We classify all finite-dimensional connected graded pointed Majid algebras of rank two which are not twist equivalent to ordinary pointed Hopf algebras.
\end{abstract}


\section{Introduction}
The theory of finite tensor categories \cite{EO} has aroused much interest in recent years. Among which, a proper classification theory is highly welcome and certainly is very challenging. As the general classification problem seems still far out of reach, it is necessary to narrow the scope and focus first on some interesting classes. In this respect, fusion and multi-fusion categories, that is, semisimple finite tensor and multi-tensor categories, are first investigated in depth, see \cite{ENO, O} and references therein. To move on, Etingof and Gelaki proposed in their pioneering work \cite{EG1} to classify finite pointed tensor categories which are nonsemisimple. By pointed it is meant that the simple objects are invertible. There are multifold reasons for this restriction: firstly, this kind of reduction is standard and powerful in representation theory; secondly, this class of tensor categories are essentially concrete, i.e., they admit quasi-fiber functors and they can be realized as the module categories of finite-dimensional elementary quasi-Hopf algebras by the Tannakian formalism \cite{EO}; thirdly, this theory is a natural generalization of the deep and beautiful theory of elementary (or equivalently, finite-dimensional pointed) Hopf algebras, see \cite{an, as2, AS1, H0, H3}.

In \cite{EG1, EG2}, Etingof and Gelaki obtained a series of classification results about graded elementary quasi-Hopf algebras over cyclic groups of prime order; in \cite{EG3, G}, they studied graded elementary quasi-Hopf algebras over general cyclic groups and their liftings. One main achievement of this series of works is a complete classification of elementary quasi-Hopf algebras of rank 1. More importantly, a novel method of constructing genuine quasi-Hopf algebras from known pointed Hopf algebras is invented. Along the same vein, Angiono classified in \cite{A} finite-dimensional elementary quasi-Hopf algebras over cyclic groups whose orders have no small prime divisors. On the other hand, our previous works \cite{qha1, qha2, qha3} introduce many useful ideas and tools from the representation theory of finite-dimensional algebras into the theory of pointed tensor categories and quasi-quantum groups (including quasi-Hopf algebras and their duals in accordance with the philosophy of Drinfeld's theory of quantum groups \cite{D, Dr}). In particular, a quiver framework is set up and a general method of constructing quasi-quantum groups and pointed tensor categories via projective representations of finite groups and quiver representation theory is provided.

However, except for some sporadic examples \cite{EG2}, so far all finite quasi-quantum groups obtained in the literature are either over cyclic groups or of rank 1 (in fact, mostly both). An obvious reason, to the authors, is that a uniform expression of 3-cocycles over non-cyclic groups was not available. This prevents us from a desired control of the associators of quasi-quantum groups. With the explicit unified formulas of 3-cocycles on general finite abelian groups recently offered in \cite{bgrc1, bgrc2}, now it seems possible to pursuit the classification of finite quasi-quantum groups and finite pointed tensor categories in a much greater scope. As a crucial step to move forward, we should first tackle the classification of finite quasi-quantum groups of rank 2 as clearly suggested by the successful development strategy of the classification theory of finite-dimensional pointed Hopf algebras, see \cite{h, h2}. This is the main aim of the present paper.

Let $H = \bigoplus_{n \ge 0} H_n$ be a graded elementary quasi-Hopf algebra. The novel idea of \cite{EG1, EG2, G} is that, if $H_0$ is the group algebra of a cyclic group $\mathbb{G}=\Z_\n,$ then $H$ can be embedded into a bigger quasi-Hopf algebra $\widetilde{H}$ which is twist equivalent to a graded elementary Hopf algebra $H'$ with $H'_0= k G$ where $G=\Z_n$ with $n=\n^2.$ The crux of this fact is essentially due to group cohomology. More precisely, if $\Phi$ is a 3-cocycle on $\mathbb{G},$ then its pull-back $\pi^* (\Phi)$ along the canonical projection $\pi \colon G \to \mathbb{G}$ vanishes, i.e., a 3-coboundary. In this situation, we say that $\Phi$ is \emph{resolvable}. The first key observation of the present paper is that this fact can be generalized to all abelian groups of form $\Z_\m \times \Z_\n$ as anticipated in \cite{EG1}. This relies heavily on our previous work of linear braided Gr-categories \cite{bgrc1}. The resolvability of any 3-cocycle on $\Z_\m \times \Z_\n$ motivates us to pursuit a similar connection between a graded quasi-Hopf algebra $H$ with $H_0=k \Z_\m \times \Z_\n$ and an appropriate Hopf algebra $H'.$ The second key observation of the present paper is that an explicit connection can be built by overcoming two difficulties, explained in below, which are relatively mild in the case of cyclic groups \cite{EG1, EG2, EG3, G, A}.

In accordance with our previous works \cite{qha1, qha2, qha3}, we always work on the dual situation. For this, let $\MM = \bigoplus_{i \ge 0} \MM_i$ be a coradically graded pointed Majid algebra, then $\MM_0$ is a group Majid algebra $(k \mathbb{G}, \Phi).$ Similar to the case of pointed Hopf algebras, we may factorize $\MM$ as $\R \# k \mathbb{G}$ by a quasi-version of the bosonization procedure. Here $\R$ is the coinvariant subalgebra of $\MM$ with respect to the natural coaction of $k \mathbb{G}.$ The first difficulty is the generation problem, that is, whether $\R$ is generated in degree 1. The second difficulty is to determine a suitable resolution $\pi \colon G \to \mathbb{G}$ such that $\R$ becomes a Nichols algebra in the twisted Yetter-Drinfeld category ${^{G}_{G}\mathcal{YD}^{\pi^{*}(\Phi)}}.$ If $\MM$ is assumed of rank 2, then we can overcome these two difficulties and thus the Majid algebra $\MM$ may be realized as an easy quotient of a pointed Majid algebra $\mathbf{M}$ which is twist equivalent to an ordinary pointed Hopf algebra $\mathbf{H}.$ This will facilitate applications to Majid algebras the theory of finite-dimensional pointed Hopf algebras. In other words, we get the following diagram:

\begin{figure}[hbt]
\begin{picture}(100,100)(0,0)
\put(0,90){\makebox(0,0){$ \MM=\R\#k\mathbbm{G}\;\;\;$ \textsf{Original Majid algebra}}}
\put(-50,85){\vector(0,-1){30}}
\put(0,50){\makebox(0,0){$ {\textbf{M}}=\R\#kG\;\;\;$ \textsf{Bigger Majid algebra}}}
\put(-50,45){\vector(0,-1){30}}\put(-40,30){\makebox(0,0){$\exists \, J$}}
\put(0,10){\makebox(0,0){$ \textbf{H}=(\R\#kG)^{J}\;\;\;$ \textsf{Ordinary Hopf algebra}}}
\end{picture}
\end{figure}

In order to take full advantage of the theory of finite-dimensional pointed Hopf algebras, especially Heckenberger's well-known classification result of finite-dimensional rank $2$ Nichols algebras \cite{h, h2} for the present purpose, we still need to answer the following question: For which finite-dimensional pointed Hopf algebra of rank $2$, can one reverse the above diagram to get a genuine pointed Majid algebra? Our third key observation is that this question can be reduced to solving some elementary congruence equations, see \eqref{eq6.8}. We show that such congruence equations have at most one solution and we give a simple criterion to determine when they do. Therefore, we complete the above one-way diagram into to a circuit in below and this finally leads to the desired classification of finite quasi-quantum groups of rank 2.

\begin{figure}[hbt]
\begin{picture}(180,105)(0,0)
\put(0,90){\makebox(0,0){$ \MM=\R\#k\mathbbm{G}\;\;\;$ \textsf{Original Majid algebra}}}     \put(150,90){\makebox(0,0){$ \MM=\mathscr{B}^{J^{-1}}\#k\mathbbm{G}$ }}
\put(-50,85){\vector(0,-1){30}}                                                              \put(150,55){\vector(0,1){30}} \put(200,70){\makebox(0,0){\small{Eq. \eqref{eq6.8} soluble}}}
\put(0,50){\makebox(0,0){$ {\textbf{M}}=\R\#kG\;\;\;$ \textsf{Bigger Majid algebra}}}           \put(150,50){\makebox(0,0){$ \textbf{H}^{J^{-1}}=\mathscr{B}^{J^{-1}}\#kG$ }}
\put(-50,45){\vector(0,-1){30}}\put(-40,30){\makebox(0,0){\small\textsf{ $\exists\, J$}}}  \put(150,15){\vector(0,1){30}}
\put(0,10){\makebox(0,0){$ \textbf{H}=(\R\#kG)^{J}\;\;\;$ \textsf{Ordinary Hopf algebra}}}      \put(150,10){\makebox(0,0){$ \textbf{H}=\mathscr{B}\#kG$ }}
\end{picture}
\end{figure}

The paper is organized as follows. Some necessary concepts, notations and facts are collected in Section 2. In particular, Nichols algebras in twisted Yetter-Drinfeld categories ${^{\mathbbm{G}}_{\mathbbm{G}}\mathcal{YD}^{\Phi}}$ are introduced. Section 3 is devoted to the resolvability of an arbitrary $3$-cocycle $\Phi$ on $\mathbb{Z}_{\m}\times \Z_{\n}$, that is, find a group epimorphism $\pi: \ G=\Z_m\times \Z_n \to \mathbb{Z}_{\m}\times \Z_{\n}$ such that $\pi^{*}(\Phi)$ is a coboundary on $G.$ The result of this section is one of several key ingredients of our classification procedure. The generation problem for finite-dimensional pointed Majid algebras of rank 2 is established in Section 4. Our classification procedure and the main result are given in Section 5. Sections 6, 7, and 8 are designed to give explicit classification results based on the previous sections. There is also an appendix of the list of full binary trees used in Sections 6-8.

Throughout of this paper, $k$ is an algebraically closed field with characteristic zero and all vector spaces, linear mappings, (co)algebras and unadorned tensor product $\otimes$ are over $k.$

\section{Preliminaries}
This section is devoted to some preliminary concepts, notations and facts.

\subsection{Majid algebras.}
The concept of Majid algebras is dual to that of quasi-Hopf algebras \cite{Dr}, and can be given as follows.
\begin{definition} A Majid algebra is a coalgebra $(H,\D,\e)$ equipped with a
compatible quasi-algebra structure and a quasi-antipode. Namely,
there exist two coalgebra homomorphisms $$\M: H \otimes H \To H, \ a
\otimes b \mapsto ab, \quad \mu: k \To H,\ \lambda \mapsto \lambda
1_H,$$ a convolution-invertible map $\Phi: H^{\otimes 3} \To k$
called associator, a coalgebra antimorphism $\S: H \To H$ and two
functionals $\alpha,\beta: H \To k$ such that for all $a,b,c,d \in
H$ the following equalities hold:
\begin{eqnarray}
&a_1(b_1c_1)\Phi(a_2,b_2,c_2)=\Phi(a_1,b_1,c_1)(a_2b_2)c_2,\\
&1_Ha=a=a1_H, \\
&\Phi(a_1,b_1,c_1d_1)\Phi(a_2b_2,c_2,d_2) =\Phi(b_1,c_1,d_1)\Phi(a_1,b_2c_2,d_2)\Phi(a_2,b_3,c_3),\\
&\Phi(a,1_H,b)=\e(a)\e(b). \\
&\S(a_1)\alpha(a_2)a_3=\alpha(a)1_H, \quad a_1\beta(a_2)\S(a_3)=\beta(a)1_H, \label{eq2.5} \\
&\Phi(a_1,\S(a_3),a_5)\beta(a_2)\alpha(a_4)
=\Phi^{-1}(\S(a_1),a_3,\S(a_5)) \alpha(a_2)\beta(a_4)=\e(a). \label{eq2.6}
 \end{eqnarray}
Throughout we use the Sweedler sigma notation $\D(a)=a_1 \otimes
a_2$ for the coproduct and $a_1 \otimes a_2 \otimes \cdots \otimes
a_{n+1}$ for the result of the $n$-iterated application of $\D$ on
$a.$\end{definition}

\begin{example} Let $G$ be a group and $\Phi$ a normalized $3$-cocycle on $G$. It is well known that
the group algebra $kG$ is a Hopf algebra with $\D(g)=g\otimes g,\; \S(g)=g^{-1}$ and
$\e(g)=1$ for any $g\in G$. By extending $\Phi$ trilinearly, then $\Phi \colon (kG)^{\otimes 3} \to k$ becomes a
 convolution-invertible map. Define two linear functions
 $\alpha,\beta \colon kG \to k$ just by $\alpha(g):=\e(g)$ and $\beta(g):=\frac{1}{\Phi(g,g^{-1},g)}$
 for any $g\in G$. It is not hard to see  that $kG$ together with these $\Phi, \ \alpha$ and $\beta$
 becomes a Majid algebra. In the following, this resulting Majid algebra is denoted by $(kG,\Phi)$.
\end{example}

 A Majid algebra $H$ is said to be \emph{pointed}, if the underlying
coalgebra is so. Given a pointed Majid algebra $(H,\D, \e,
\M, \mu, \Phi,\S,\alpha,\beta),$ let $\{H_n\}_{n \ge 0}$ be its
coradical filtration, and $$\gr H = H_0 \oplus H_1/H_0 \oplus H_2/H_1
\oplus \cdots$$ the corresponding coradically graded coalgebra. Then naturally
$\gr H$ inherits from $H$ a graded Majid algebra structure. The
corresponding graded associator $\gr\Phi$ satisfies
$\gr\Phi(\bar{a},\bar{b},\bar{c})=0$ for all homogeneous
$\bar{a},\bar{b},\bar{c} \in \gr H$ unless they all lie in $H_0.$
Similar condition holds for $\gr\alpha$ and $\gr\beta.$ In
particular, $H_0$ is a Majid subalgebra and it turns out to be the
Majid algebra $(kG, \gr\Phi)$ for $G=G(H),$ the set of group-like
elements of $H.$ We call a pointed Majid algebra $H$
\emph{coradically graded} if $H \cong \gr H$ as Majid algebras. One can also see \cite{qha1} for more details on pointed Majid algebras.

\begin{definition} Let $(H,\D, \e, \M, \mu, \Phi,\S,\alpha,\beta)$ be a Majid algebra. A convolution-invertible linear map
$$J:\;H\otimes H\to k$$
is called a twisting (or gauge transformation) on $H$ if
$$J(h,1)=\e(h)=J(1,h)$$
for all $h\in H$.
\end{definition}

Now let $H$ be a Majid algebra together with a twisting $J$. Then one can construct
a new Majid algebra $H^{J}$. By definition, $H^{J}=H$ as a coalgebra and the multiplication $``\circ"$ on $H^{J}$ is given by
\begin{equation}
 a\circ b:=J(a_1,b_1)a_2b_2J^{-1}(a_3,b_3)
\end{equation}
for all $a,b\in H$. The associator $\Phi^{J}$ and the quasi-antipode $(\S^J,\alpha^J,\beta^{J})$ are given as:
$$\Phi^J(a,b,c)=J(b_1,c_1)J(a_1,b_2c_2)\Phi(a_2,b_3,c_3)J^{-1}(a_3b_4,c_4)J^{-1}(a_4,b_5),$$
$$\S^J=\S,\;\;\;\;\alpha^J(a)=J^{-1}(\S(a_1),a_3)\alpha(a_2),\;\;\;\;\beta^J(a)=J(a_1,\S(a_3))\beta(a_2)$$
for all $a,b,c \in H$.

\begin{definition} Two Majid algebras $H_1$ and $H_2$ are called \emph{twist equivalent} if there is a twisting $J$ on $H_1$ such that we have a Majid algebra isomorphism
$$H_1^{J}\cong H_2.$$ Denote $H_1 \sim H_2$ if $H_1$ is twist equivalent to $H_2$. We call a Majid algebra $H$ \emph{genuine} if it is not twist equivalent to a Hopf algebra.
\end{definition}

\subsection{Quiver Setting for pointed Majid algebras and ranks}

A quiver is a quadruple $Q=(Q_0,Q_1,s,t),$ where $Q_0$ is the set of
vertices, $Q_1$ is the set of arrows, and $s,t:\ Q_1 \longrightarrow
Q_0$ are two maps assigning respectively the source and the target
for each arrow. A path of length $l \ge 1$ in the quiver $Q$ is a
finitely ordered sequence of $l$ arrows $a_l \cdots a_1$ such that
$s(a_{i+1})=t(a_i)$ for $1 \le i \le l-1.$ By convention a vertex is
called a trivial path of length $0.$

For a quiver $Q,$ the associated path coalgebra $k Q$ is the
$k$-space spanned by the set of paths with counit and
comultiplication maps defined by $\e(g)=1, \ \D(g)=g \otimes g$ for
each $g \in Q_0,$ and for each nontrivial path $p=a_n \cdots a_1, \
\e(p)=0,$
\begin{equation}
\D(a_n \cdots a_1)=p \otimes s(a_1) + \sum_{i=1}^{n-1}a_n \cdots
a_{i+1} \otimes a_i \cdots a_1 \nonumber + t(a_n) \otimes p \ .
\end{equation}
The length of paths provides a natural gradation to the path coalgebra.
Let $Q_n$ denote the set of paths of length $n$ in $Q,$ then $k
Q=\bigoplus_{n \ge 0} k Q_n$ and $\D(k Q_n) \subseteq
\bigoplus_{n=i+j}k Q_i \otimes k Q_j.$ Clearly $k Q$ is pointed with
the set of group-likes $G(k Q)=Q_0,$ and has the following
coradical filtration $$ k Q_0 \subseteq k Q_0 \oplus k Q_1
\subseteq k Q_0 \oplus k Q_1 \oplus k Q_2 \subseteq \cdots.$$
Thus $k Q$ is coradically graded. The path coalgebras of quivers can be
presented as cotensor coalgebras, so they are cofree in the category
of pointed coalgebras and enjoy a universal mapping property.

A quiver $Q$ is said to be a Hopf quiver if
the corresponding path coalgebra $k Q$ admits a graded Hopf algebra
structure. Hopf quivers can be determined by ramification data of
groups. Let $G$ be a group and denote its set of conjugacy classes
by $\C.$ A ramification datum $R$ of the group $G$ is a formal sum
$\sum_{C \in \C}R_CC$ of conjugacy classes with coefficients in
$\mathbb{N}=\{0,1,2,\cdots\}.$ The corresponding Hopf quiver
$Q=Q(G,R)$ is defined as follows: the set of vertices $Q_0$ is $G,$
and for each $x \in G$ and $c \in C,$ there are $R_C$ arrows going
from $x$ to $cx.$ It is clear by definition that $Q(G,R)$ is
connected if and only if the the set $\{c \in C | C \in \C \
\text{with} \ R_C \ne 0\}$ generates the group $G.$ For a given Hopf
quiver $Q,$ the set of graded Hopf structures on $k Q$ is in
one-to-one correspondence with the set of $k Q_0$-Hopf bimodule
structures on $k Q_1.$

It is shown in \cite{qha1} that the path coalgebra $k Q$ admits a
graded Majid algebra structure if and only if the quiver $Q$ is a
Hopf quiver. Moreover, for a given Hopf quiver $Q=Q(G,R),$ if we fix
a Majid algebra structure on $k Q_0=(k G,\Phi)$ with
quasi-antipode $(\S,\alpha,\beta),$ then the set of graded Majid
algebra structures on $k Q$ with $k Q_0=(k
G,\Phi,\S,\alpha,\beta)$ is in one-to-one correspondence with the
set of $(k G,\Phi)$-Majid bimodule structures on $k Q_1.$
Thanks to the Gabriel type theorem given in \cite{qha1}, for an arbitrary
pointed Majid algebra $H,$ its graded version $\gr H$ can be
realized uniquely as a large Majid subalgebra of some graded Majid
algebra structure on a Hopf quiver.  By ``large" it is meant that the
Majid subalgebra contains the set of vertices and arrows of the Hopf
quiver.  We denote this unique quiver by $Q(H)$
and call it the \emph{Gabriel
quiver} of $H$. Therefore, in principle all pointed Majid algebras are able
to be constructed on Hopf quivers.

\begin{definition} Let $H$ be a pointed Majid algebra, $Q(H)$ be its Gabriel quiver and
$R=\sum_{C \in \C}R_CC$ be the ramification datum of $Q(H)$. The \emph{rank} of $H$ is defined to be the natural
number $\sum_{C \in \C}R_C |C|$ where $|C|$ is the cardinality of $C.$ We say that $H$ is \emph{connected} if $Q(H)$ is connected as a graph.
\end{definition}

\begin{example}\label{e2.7} \emph{Pointed Majid algebras of rank one were studied in \cite{qha3}. We recall them here in detail as they shed much light on the working philosophy of the present paper. }

\emph{Consider the following Hopf quiver $Q(\mathbb{Z}_{n},g)$:}
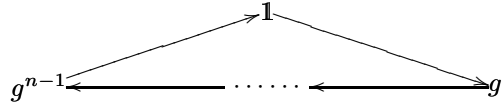
$$ \xy {\ar (0,0)*{1}; (30,-10)*{g}}; {\ar (-30,-10)*{g^{n-1}};
(0,0)*{\unit}}; {\ar (30,-10)*{g}; (3,-10)*{ \cdots  \ }}; {\ar
(-3,-10)*{\ \cdots  }; (-30,-10)*{g^{n-1}}}
\endxy $$
\emph{Now let $0\leq s\leq n-1$ be a natural number,  $q$ an $n^{2}$-th primitive root of unity and $\q:=q^{n}$.
Let $p_{i}^{l}$ denote the path in $Q(\mathbb{Z}_{n},g)$ starting from $g^{i}$ with length $l$. So $p_{i}^{0}=g^{i}$.
Let $\Phi_{s}$ be the 3-cocycle on $\mathbb{Z}_{n}$ defined by
\begin{equation}\Phi_{s}(g^{i},g^{j},g^{k})=\q^{si[\frac{j+k}{n}]},\;\;\;0\leq i,j,k\leq  n-1.
\end{equation} Here $[x]$ stands for the integral part of $x.$
 For any $h \in k$,
define $l_h=1+h+\cdots +h^{l-1}$ and $l!_h=1_h
\cdots l_h$. The Gaussian binomial coefficient is defined by
$\binom{l+m}{l}_h:=\frac{(l+m)!_h}{l!_h m!_h}$. Let $(a,b)$ denote the greatest common divisor of the two natural numbers $a,b$.}

\emph{Now we are ready to define the rank 1 pointed Majid algebra $M(n,s,q)$.
 As a coalgebra, $$M(n,s,q)=\bigoplus_{i< \frac{n^{2}}{(n^{2},s)}} k Q(\mathbb{Z}_{n},g)_{i}.$$
The associator, the multiplication, the functions $\alpha,\beta$ and the antipode are given as follows:
\begin{eqnarray}
&\Phi(p_i^{l},p_j^{m}, p_k^{t})=\delta_{l+m+t,0}\Phi_{s}(g^{i},g^{j},g^{k}),\\
&\label{eq2.10}p_i^{l}\cdot p_j^{m}=\mathbbm{q}^{-sjl}q^{-sjl}\mathbbm{q}^{s(i+l')[m+j-(m+j)']/n}
\binom{l+m}{l}_{\mathbbm{q}^{-s}q^{-s}}p_{i+j}^{l+m},\\
&\alpha(p_i^{l})=\delta_{l,0}1,\;\;\;\;\beta(p_i^{l})=\delta_{l,0}\frac{1}{\Phi_{s}(g^{i},g^{n-i},g^{i})},\\
&\S(g^{i})=g^{n-i},\;\;\;\;\S(p_{0}^{1})=\q^{-s}p_{n-1}^{1},\end{eqnarray}
for $0\leq l,m,t<\frac{n^{2}}{(n^{2},s)}$ and $0\leq i,j,k\leq n-1$, where $\delta_{a,b}$ is the Kronecker
delta, namely it is equal to $1$ if $a=b$ and $0$ otherwise, and $l'$ means the remainder of $l$ divided by $n$.}
\end{example}

We have the following basic observation.

\begin{lemma}\label{l2.8} Let $H$ be a connected pointed Majid algebra of rank $2$ and $Q(H)=Q(G,R)$ be its Gabriel
quiver, then $G$ is an abelian group which can be generated by one or two elements.
\end{lemma}
\begin{proof} Let $1$ denote the unit element of $G$. By definition, we know that in $Q(H)$ there are exactly two arrows
going out from $1$. We denote the ending vertices of these two arrows by $g$ and $h$ respectively. As the graph is connected,
it follows that $G$ can be generated by $g$ and $h$. If $g$ and $h$ live in different conjugacy classes, then the definition
of Hopf quivers implies that the conjugacy class containing $g$ (resp. $h$) is just $g$ (resp. $h$). So both $g$ and $h$ lie in the
center of $G$ and thus $G$ is abelian. If $g,h$ live in the same conjugacy class, then again by the definition of Hopf quivers we have
$$ghg^{-1}=g,\;\;\;\;\;\textrm{or}\;\;\;\;\;\;ghg^{-1}=h.$$
This implies that $g=h$ or $gh=hg$. In either case $G$ is abelian.
\end{proof}

\subsection{Yetter-Drinfeld modules over $(kG,\Phi)$.}
The definition of Yetter-Drinfeld modules over a general Majid algebra seems cumbersome. However,
its formulation becomes much simpler when this Majid algebra is $(kG,\Phi)$ with $G$ an \emph{abelian} group. This special case already suffices for our purpose.

Assume that $V$ is a left $kG$-comudule with comodule structure map $\delta_{L}:\; V\to kG\otimes V$.
Define $^{g}V:=\{v\in V|\delta_{L}(v)=g\otimes v\}$ and thus $V=\bigoplus_{g\in G}\;^{g}V.$
For the $3$-cocycle $\Phi$ on $G$ and any $g\in G$, define
\begin{equation} \widetilde{\Phi}_g:\;G\times G\to k^{\ast}, \quad (e,f)\mapsto \frac{\Phi(g,e,f)\Phi(e,f,g)}{\Phi(e,g,f)}.
\end{equation}
Direct computation shows that
$$\widetilde{\Phi}_g\in \mathbb{Z}^{2}(G,k^{\ast}).$$

\begin{definition}\label{d2.9} A left $kG$-comudule $V$ is a \emph{left-left Yetter-Drinfeld module}
over $(kG,\Phi)$ if
each $^{g}V$ is a projective
$G$-representation with respect to the $2$-cocycle $\widetilde{\Phi}_g,$ namely the $G$-action $\triangleright$ on ${^{g}V}$ satisfies
\begin{equation}
e\triangleright(f\triangleright v)=\widetilde{\Phi}_g(e,f) (ef)\triangleright v,\;\;\;\; \forall e,f\in G,\;v\in \;^{g}V.
\end{equation}
\end{definition}

The category of all left-left Yetter-Drinfeld modules is denoted by $_{ G}^{ G}\mathcal{YD}^{\Phi}$.
Similarly, one can define left-right, right-left and right-right Yetter-Drinfeld modules over $(kG,\Phi)$.
As the familiar Hopf case, $_{ G}^{ G}\mathcal{YD}^{\Phi}$ is a braided tensor category. More precisely, for any  $M, N\in \;_{ G}^{ G}\mathcal{YD}^{\Phi},$ the structure maps of $M\otimes N$ as a left-left Yetter-Drinfeld module are given by
\begin{equation}\delta_{L}(m_{g}\otimes n_{h}):=gh \otimes m_{g} \otimes n_{h},\;\;x\triangleright (m_{g}\otimes n_{h}):=\widetilde{\Phi}_x(g,h)x\triangleright m_{g}
\otimes x\triangleright n_{h}\end{equation}
for all $x,g,h\in G$ and $m_{g} \in {^{g}M},\;n_{h}\in {^{h}N}$.
The associativity constraint $a$ and the braiding $c$ of $_{ G}^{ G}\mathcal{YD}^{\Phi}$ are given respectively by
\begin{eqnarray}
&\ a((u_{e}\otimes v_{f})\otimes w_{g}) =\Phi(e,f,g)^{-1} u_{e}\otimes (v_{f}\otimes w_{g})\\
&c(u_{e}\otimes v_{f})=e \triangleright v_{f}\otimes u_{e}
\end{eqnarray}
for all $e,f,g\in G$, $u_{e}\in {^{e}U},\  v_{f}\in {^{f}V},\  w_{g} \in {^{g}W}$ and $U,V,W\in\; _{ G}^{ G}\mathcal{YD}^{\Phi}$.

\begin{remark} A left-left Yetter-Drinfeld module $V$ over $(kG,\Phi)$ is called \emph{diagonal} if every projective $G$-representation
$^{g}V$ is a direct sum of $1$-dimensional projective representations. We point out that not like the Hopf case, here the condition of $G$ being abelian can \emph{NOT} guarantee
that every $V$ is diagonal. It turns out that all $V \in {_{ G}^{ G}\mathcal{YD}^{\Phi}}$ are diagonal if and only if $\Phi$ is an abelian cocycle, see \cite{EM} and \cite{MN}. We will show in Section 3 that all $3$-cocycles
on $\mathbb{Z}_{m}\times \mathbb{Z}_{n}$ are abelian.
\end{remark}

\subsection{Hopf algebras in braided tensor categories}
Let $\mathcal{C}=(\mathcal{C},\otimes, \1, a, l, r, c)$ be a braided tensor category, where $\1$ is the unit object, $a$ ($l,$ or $r$) is the associativity (left, or right unit) constraint and $c$ is the braiding. An associative algebra in $\C$ is an object $A$ of $\mathcal{C}$ endowed with a multiplication morphism $m: A\otimes A\To A$ and a unit morphism $u: \1 \To A$ such that \[ m\circ(m\otimes \id)=m\circ(\id\otimes m)\circ a_{A,A,A},  \quad m\circ(\id \otimes u)=r_A,  \quad m\circ(u \otimes \id)=l_A. \]
If $(A,m_A,u_A)$ and $(B,m_B,u_B)$ are two algebras in $\mathcal{C},$ then one can define a natural morphism $m_{A\otimes B}: (A\otimes B)\otimes (A\otimes B)\To A\otimes B$ by
\begin{equation*}
m_{A\otimes B}=(m_A\otimes m_B)\circ a_{A\otimes A,B,B}\circ(a_{A,A,B}^{-1}\otimes \id)
\circ(\id\otimes c_{B,A}\otimes \id)\circ(a_{A,B,A}\otimes \id)\circ a_{A\otimes B,A,B}^{-1}
\end{equation*}
such that $(A\otimes B,m_{A\otimes B},u_A\otimes e_B)$ is again an algebra in $\mathcal{C}.$ The resulting algebra is called the braided tensor product of $A$ and $B.$

Dually, a coassociative coalgebra in $\C$ is an object $C$ of $\C$ endowed with a comultiplication  morphism $\Delta: C\To C\otimes C$ and a counit morphism $\varepsilon: C \To \1$ such that \[ a_{C,C,C}\circ(\Delta\otimes \id) \circ \Delta=(\id\otimes \Delta)\circ \Delta, \quad  r_C^{-1}= (\id \otimes \varepsilon)\circ\Delta, \quad l_C^{-1}=(\varepsilon\otimes \id)\circ\Delta. \] If $(C, \Delta_C, \varepsilon_C)$ and $(D, \Delta_D, \varepsilon_D)$ are coalgebras in $\C,$ then one can define a suitable morphism $\Delta_{C \otimes D}: C \otimes D \To (C \otimes D) \otimes (C \otimes D)$ to get the braided tensor product coalgebra $(C \otimes D, \Delta_{C \otimes D}, \varepsilon_C \otimes \varepsilon_D)$ in $\C.$

Endowed with braided tensor products, one may naturally define Hopf algebras in braided tensor categories.  A sextuplet $(H,m,u,\Delta,\varepsilon,S)$ is a  Hopf algebra in $\mathcal{C}$ if $(H,m,u)$ is an algebra in $\mathcal{C},$  $(H,\Delta,\varepsilon)$ is a coalgebra in $\mathcal{C}$ and $\Delta:H \To H\otimes H$ and $\varepsilon: H \To \1$ are algebra maps in $\C,$ and $\S: H \To H$ is a morphism, to be called the antipode, subject to
\begin{equation}
m \circ (\S \otimes \id) \circ \Delta = u \circ \varepsilon = m \circ (\id \otimes \S) \circ \Delta.
\end{equation}
As the usual case, we can define ideals of algebras, coideals of coalgebras, Hopf ideals of Hopf algebras, and the corresponding quotient structures in braided tensor categories. We do not include further details. For our purpose, we only record some more facts on the duals of finite-dimensional Hopf algebras in the braided tensor category ${^G_G\mathcal{YD}^{\Phi}}$. Let $H$ be such a Hopf algebra. Then its left dual $^*H$ and right dual $H^*$ are again Hopf algebras in ${^G_G\mathcal{YD}^{\Phi}}$ in a natural manner, and in addition, $^*(H^*)\cong (^*H)^*\cong H.$ For more details on algebras and Hopf algebras in braided tensor categories, the reader is referred to \cite{majid}.

\subsection{Nichols algebras in $^{ G}_{ G}\mathcal{Y}\mathcal{D}^\Phi$}
In simple terms, Nichols algebras are the analogue of the usual symmetric algebras in more general braided tensor categories. They can be defined by various equivalent ways, see \cite{as1, as2}. Here we adopt the defining method in terms of the universal property.

Let $V$ be a nonzero object in $^{ G}_{ G}\mathcal{Y}\mathcal{D}^\Phi.$ By $T_{\Phi}(V)$ we denote the tensor algebra in $^{ G}_{ G}\mathcal{Y}\mathcal{D}^\Phi$ generated freely by $V.$ It is clear that $T_{\Phi}(V)$ is isomorphic to $\bigoplus_{n \geq 0}V^{\otimes \overrightarrow{n}}$ as a linear space, where $V^{\otimes \overrightarrow{n}}$ means
$$\underbrace{(\cdots((}_{n-1}V\otimes V)\otimes V)\cdots \otimes V).$$ This induces a natural $\mathbb{N}$-graded structure on $T_{\Phi}(V).$ Define a comultiplication on $T_{\Phi}(V)$ by $\Delta(X)=X\otimes 1+1\otimes X, \ \forall X \in V,$ a counit by $\varepsilon(X)=0,$ and an antipode by $S(X)=-X.$ It is routine to verify that these provide an $\N$-graded Hopf algebra structure on $T_{\Phi}(V)$ in the braided tensor category $^{ G}_{ G}\mathcal{Y}\mathcal{D}^\Phi.$

\begin{definition}
The Nichols algebra $\mathscr{B}(V)$ of $V$ is defined to be the quotient Hopf algebra $T_{\Phi}(V)/I$ in $^{ G}_{ G}\mathcal{Y}\mathcal{D}^\Phi,$ where $I$ is the unique maximal graded Hopf ideal generated by homogeneous elements of degree greater than or equal to $2.$
\end{definition}

Let $J$ be a $2$-cochain on $G,$ i.e., a function $J \colon G \times G \To k^*$ such that $J(g,1)=1=J(1,g)$ for all $g \in G.$ Then clearly the group Majid algebras $(kG, \Phi)$ and $(kG, \Phi \partial(J))$ are twist equivalent with twisting offered by extending $J$ bilinearly. It is well known that their associated Yetter-Drinfeld categories $^{ G}_{ G}\mathcal{Y}\mathcal{D}^\Phi$ and $^{ G}_{ G}\mathcal{Y}\mathcal{D}^{\Phi \partial(J)}$ are tensor equivalent \cite{DPR}. More precisely, the tensor functor $(\mathcal{F}^J, \varphi_0, \varphi_2) \colon {^{ G}_{ G}\mathcal{Y}\mathcal{D}^\Phi} \To {^{ G}_{ G}\mathcal{Y}\mathcal{D}^{\Phi \partial(J)}}$ is given by \[ \mathcal{F}^J(U)=U^J, \quad \varphi_0=\id, \quad \varphi_2 \colon V^J \otimes W^J \To (V \otimes W)^J, \ X \otimes Y \mapsto J(x, y) X \otimes Y \] where $U^J$ as a $G$-comodule is the same as $U,$ but with a new $G$-action obtained by twisting that on $U$ via $J$ as follows
\begin{equation}\label{ta}
g \triangleright^J X=\frac{J(g,x)}{J(x,g)}g \triangleright X, \quad \forall g \in G, \ X \in {^xU}.
\end{equation}
Naturally, this tensor equivalence maps algebras in $^{ G}_{ G}\mathcal{Y}\mathcal{D}^\Phi$ to algebras in $^{ G}_{ G}\mathcal{Y}\mathcal{D}^{\Phi \partial(J)}.$ In particular, the Nichols algebra $\mathscr{B}(V)$ is mapped to $\mathscr{B}(V)^J$ which is again a Nichols algebra in $^{ G}_{ G}\mathcal{Y}\mathcal{D}^{\Phi \partial(J)}.$ Note that the multiplication of the latter, denoted by ``$\circ$", is given by
\begin{equation}\label{tm}
X \circ Y = J(x,y)XY, \quad \forall X \in {^xV}, \ Y \in {^yV}.
\end{equation}
In addition, we have the following obvious but useful observation:

\begin{lemma}\label{tn}
$\mathscr{B}(V)^J \cong \mathscr{B}(V^J)$ as Nichols algebras in $^{ G}_{ G}\mathcal{YD}^{\Phi \partial(J)}.$
\end{lemma}

\subsection{Bosonization for pointed Majid algebras}
The theory of bosonization in a broad context can be found in \cite{majid} in terms of braided diagrams. For our purpose, it is enough to focus on the situation of graded pointed Majid algebras. For the sake of completeness and later applications, we record in the following some explicit concepts, notations and results without proof.

In the rest of the paper, we always assume that
$$\MM=\bigoplus_{i \in \N} \MM_{i}$$ is a connected coradically graded pointed Majid algebra with unit $1.$ So $\MM_0=(kG,\Phi)$ for some
group $G$ together with a $3$-cocycle $\Phi$ on $G$.
 Let $\pi:\; \MM\to \MM_{0}$ be
the canonical projection. Then $\MM$ is a $kG$-bicomodule naturally through
$$\delta_{L}:=(\pi\otimes \id)\D,\;\;\;\;\delta_{R}:=(\id\otimes \pi)\D.$$
Thus there is a $G$-bigrading on $\MM$, that is,
$$\MM=\bigoplus_{g,h\in G}\;^{g}\MM^{h}$$
where $^{g}\MM^{h}=\{m\in \MM \mid \delta_{L}(m)=g\otimes m,\;\delta_{R}(m)=m\otimes h\}$. We only deal with
homogeneous elements with respect to this $G$-bigrading unless stated otherwise.
For example, whenever we write $\D(X)=X_1\otimes X_2,$ all $X,X_1,X_2$ are assumed homogeneous. For the convenience of the exposition, we make a convention: given any capital $X \in {^{g}\MM^{h}}$, use its lowercase $x$ to denote $gh^{-1}.$

Define the subalgebra of $\MM$ consisting of coinvariants as
$$\R:=\{m\in \MM \mid \delta_R(m)=m\otimes 1\}.$$ Clearly $1 \in \R$ and $\R$ inherits from $\MM$ a left $G$-coaction, i.e., $\R=\bigoplus_{g \in G} {^g\R}.$
There is also a $(kG,\Phi)$-action on $\R$ given by
\begin{equation}\label{eq4.1} f\triangleright X:=\frac{\Phi(fx,f^{-1},f)}{\Phi(f,f^{-1},f)}(f\cdot X)\cdot f^{-1} \end{equation}
for all $f,x\in G$ and $X\in {^{x}\R}$. Here $\cdot$ is the multiplication in $\MM$. Then $(\R,\delta_{L},\rhd)$ is a left-left Yetter-Drinfeld module over $(kG,\Phi)$.

Moreover, there are several natural operations on $\R$ inherited from $\MM$ as follows:
\begin{eqnarray*}&\M:\;\R\otimes \R\to \R, \;\;\;\;\;(X,Y)\mapsto XY:=X\cdot Y;\\
&u \colon k \to \R, \quad \lambda \mapsto \lambda 1; \\
&\D_{\R}:\;\R\to \R\otimes \R,\;\;\;\;X\mapsto \Phi(x_1,x_2,x_2^{-1})X_{1}\cdot x_{2}^{-1}\otimes X_{2};\\
&\e_{\R}:\;\R\to k,\;\;\;\;\e_{\R}:=\e|_{\R};\\
&\S_{\R}:\;\R\to \R,\;\;\;\;X\mapsto \frac{1}{\Phi(x,x^{-1},x)}x\cdot \S(X).\end{eqnarray*}
Then it is routine to verify that $(\R, \M, u, \D_{\R}, \e_{\R}, \S_{\R})$ is a Hopf algebra in $^{G}_{G}\mathcal{YD}^{\Phi}$.

Conversely, let $H$ be a Hopf algebra in $_{G}^{G}\mathcal{YD}^{\Phi}.$ Since $H$ is a left $G$-comodule, there is a $G$-grading on $H$:
 $$H=\bigoplus_{x\in G} {^{x}H}$$
 where $^{x}H=\{X\in H|\delta_{L}(X)=x\otimes X\}.$ As before, we only need to deal with $G$-homogeneous elements.
 As a convention, homogeneous elements in $H$ are denoted by capital letters, say $X, Y, Z, \dots,$ and the associated degrees are denoted by their lower cases, say $x, y, z, \dots.$

 For our purpose, we also assume  that $H$ is $\mathbb{N}$-graded with $H_0=k$. If $X\in H_{n}$, then we say that $X$ has length $n$. Moreover, we assume that both gradings are compatible in the sense that \[ H=\bigoplus_{g\in G} {^gH} =\bigoplus_{g\in G} \bigoplus_{n\in \mathbb{N}} {^{g}H_{n}}. \] For example, the Hopf algebra $\R$ in $_{G}^{G}\mathcal{YD}^{\Phi}$ considered above satisfies these assumptions as $\R=\bigoplus_{i\in \mathbb{N}}\R_{i}$ is coradically graded. For any $X\in H$, we write its comultiplication as
 $$\D_{H}(X)=X_{(1)}\otimes X_{(2)}.$$
\begin{proposition}
Keep the assumptions on $H$ as above. Define on $H\otimes kG$ a product by
\begin{equation}
(X\otimes g)(Y\otimes h)=\frac{\Phi(xg,y,h)\Phi(x,y,g)}{\Phi(x,g,y)\Phi(xy,g,h)}X(g\triangleright Y)\otimes gh,
\end{equation}
and a coproduct  by
\begin{equation}
\D(X\otimes g)=\Phi(x_{(1)},x_{(2)},g)^{-1}(X_{(1)}\otimes x_{(2)}g)\otimes (X_{(2)}\otimes g).
\end{equation}
Then $H\otimes k G$ becomes a graded Majid algebra with a quasi-antipode $(\S,\alpha,\beta)$ given by
\begin{eqnarray}
&\S(X\otimes g)=\frac{\Phi(g^{-1},g,g^{-1})}{\Phi(x^{-1}g^{-1},xg,g^{-1})\Phi(x,g,g^{-1})}(1\otimes x^{-1}g^{-1})(\S_H(X)\otimes 1), \\
&\alpha(1\otimes g)=1,\ \ \ \alpha(X\otimes g)=0,  \\
&\beta(1\otimes g)=\Phi(g,g^{-1},g)^{-1},\ \ \beta(X\otimes g)=0,
\end{eqnarray}
here $g,h\in G$ and $X,Y$ are homogeneous elements of length  $\geq 1.$
\end{proposition}

In the following, by $H\# k G$ we denote the resulting Majid algebra defined on $H\otimes kG.$

\begin{proposition}\label{p4.5} Let $\MM$ and $\R$ be as before, and $\R\#kG$ be the Majid algebra as defined in the previous proposition. Then the map
 $$F:\;\R\#kG \to \MM,\;\;\;\;X\otimes g\mapsto Xg$$
 is an isomorphism of Majid algebras.
\end{proposition}

\subsection{Generators of abelian groups}
For applications in Section 5, we also need to recall some elementary results given in \cite{qha5} about generators of abelian groups. Given two generators $g,h$ of $\mathbbm{Z}_{m}\times \mathbbm{Z}_{n}=\langle g_1,g_2|g_{1}^{m}=g_{2}^{n}=1, g_{1}g_{2}=g_{2}g_{1}\rangle$ with $m|n$,
we know that there are integers $a,b,c,d$ such that $g=g_{1}^{a}g_{2}^{b},h=g_{1}^{c}g_{2}^{d}$ and $g,h$ generate
$\mathbbm{Z}_{m}\times \mathbbm{Z}_{n}$. The question is that can we simplify the expression of $g,h$? That is, up to an automorphism of $\mathbbm{Z}_{m}\times \mathbbm{Z}_{n}$, deduce the integers $a,b,c,d$ as simple as possible.
To this end, we call two generators $h_{1},h_{2}$ of $\mathbbm{Z}_{m}\times \mathbbm{Z}_{n}$ are \emph{standard} if there
is an automorphism $\sigma\in \Aut(\mathbbm{Z}_{m}\times \mathbbm{Z}_{n})$ satisfying $\sigma(g_1)=h_1, \sigma(g_2)=h_2$.

The following two lemmas, which are \cite[Corollary 4.3]{qha5} and \cite[Proposition 4.1]{qha5} respectively, answer the above question.

\begin{lemma}\label{l4} Let $g,h$ be two generators of $\mathbbm{Z}_{m}\times \mathbbm{Z}_{n}$ with $m|n$. Assume the order of $h$ is $n$, then there are
standard generators $g_{1},g_{2}$ of $\mathbbm{Z}_{m}\times \mathbbm{Z}_{n}$ such that
$$g=g_{1}g_{2}^{a},\;\;h=g_{2}$$
for some $0\leq a<m$.
\end{lemma}

\begin{lemma}\label{l3} Assume that $g$ and $h$ generate the abelian group $\mathbbm{Z}_{m}\times \mathbbm{Z}_{n}$ with $m|n$, then there are integers $m_{1},m_{2},n_{1},n_{2},a,b$ such that

\emph{(i)} $m=m_{1}n_{1},\;n=m_{2}n_{2},\;\;\;\;m_{1}|m_{2},\; n_{1}|n_{2},\;\;\;\;(m_{2},n_{2})=1$;

\emph{(ii)} $0\leq a< n_{2},\;0\leq b<m_{2}$ and $$g=g_{2}h_{1}h_{2}^{a},\;\;h=g_{1}g_{2}^{b}h_{2}$$
where $g_1,g_{2}$ $(\mathrm{resp.} \ h_{1},h_{2}) $ are standard generators of $\mathbbm{Z}_{m_1}\times \mathbbm{Z}_{m_2}$ $(\mathrm{resp.} \  \mathbbm{Z}_{n_1}\times \mathbbm{Z}_{n_2}).$
\end{lemma}

\section{$3$-cocycles on $\mathbb{Z}_{\mathbbm{m}}\times \mathbb{Z}_{\mathbbm{n}}$ and their resolutions}
The aim of this section is to show that every $3$-cocycle $\Phi$ on $\mathbb{Z}_{\mathbbm{m}}\times \mathbb{Z}_{\mathbbm{n}}$ is abelian in the sense of \cite{EM} and can be ``resolved" in a bigger abelian group $G,$ namely there exists a group epimorphism $\pi \colon G \to\mathbb{Z}_{\mathbbm{m}}\times \mathbb{Z}_{\mathbbm{n}}$ such that the pull-back $\pi^{*}(\Phi)$ is a coboundary on $G$.

\subsection{Abelian cocycles}\label{s3.2}
The original definition of abelian cocycles was given in \cite{EM}. For our purpose, we prefer the following equivalent definition via twisted quantum doubles appeared in \cite{MN}. So first we need to recall the definition of twisted quantum doubles \cite{DPR}. The {\em twisted quantum double} $D^{\omega}(G)$ of $G$ with respect to the $3$-cocycle $\omega$ on $G$ is the semisimple quasi-Hopf algebra with underlying vector space $(kG)^{\ast} \otimes kG$  in which multiplication,
comultiplication $\Delta$, associator $\phi$, counit $\varepsilon$, antipode
$\S$, $\alpha$ and $\beta$ are given by
\begin{eqnarray*}
&&(e(g) \otimes x)(e(h) \otimes y) =\theta_g(x,y) \delta_{g^x,h}
e(g)\otimes x y,\\
&&\Delta(e(g)\otimes x)  = \sum_{hk=g} \gamma_x(h,k) e(h)\otimes x
\otimes e(k) \otimes x,\\
&&\phi = \sum_{g,h,k \in G} \omega(g,h,k)^{-1} e(g) \otimes 1 \otimes
e(h) \otimes 1\otimes e(k) \otimes 1,\\
&&\S(e(g)\otimes x) =
\theta_{g^{-1}}(x,x^{-1})^{-1}\gamma_x(g,g^{-1})^{-1}e(x^{-1}g^{-1}x)\otimes
x^{-1},\\
&&\varepsilon(e(g)\otimes x) = \delta_{g,1}, \quad \alpha=1, \quad \beta=\sum_{g \in
G} \omega(g,g^{-1},g)e(g)\otimes 1,
\end{eqnarray*}
where $\{e(g)|g\in G\}$ is the dual basis of $\{g|g\in G\}$,  $\delta_{g,1}$ is the Kronecker delta,  $g^x=x^{-1}g x$, and
\begin{eqnarray*}
\theta_g(x,y) &=&\frac{\omega(g,x,y)\omega(x,y,(x y)^{-1}g x y)}{\omega(x,x^{-1}g x,y)}, \\
\gamma_g(x,y) & = & \frac{\omega(x,y,g)\omega(g, g^{-1}x g, g^{-1}yg)}{\omega(x,g,
g^{-1}y g)}
\end{eqnarray*}
for any $x, y, g \in G.$

It is well known that $M$ is a
left $D^{\omega}(G)$-module if and only if $M$ is a left-left Yetter-Drinfeld module over $(kG,\omega)$ as defined in Subsection 2.3.

\begin{definition} A $3$-cocycle $\omega$ on $G$ is called \emph{abelian} if $D^{\omega}(G)$ is a commutative algebra.
\end{definition}

\subsection{$3$-cocycles} Let $\mathbb{G}:=\mathbb{Z}_{\mathbbm{m}}\times \mathbb{Z}_{\mathbbm{n}}$ and
 $\mathbbm{g}_{1}$ (resp. $\mathbbm{g}_{2}$) be a generator of $\mathbb{Z}_{\mathbbm{m}}$
(resp. $\mathbb{Z}_{\mathbbm{n}}$).
For any natural numbers $0\leq a<\mathbbm{m}, \ 0\leq b<(\mathbbm{m},
\mathbbm{n}), \ 0\leq d<\mathbbm{n}$, define a map
$$\Phi_{a,b,d}: \mathbb{G}\times \mathbb{G}\times \mathbb{G}\to k^{\ast}$$ by setting
\begin{equation} \label{eq3.1}
\Phi_{a,b,d}(\mathbbm{g}_{1}^{i}\mathbbm{g}_{2}^{j},\mathbbm{g}_{1}^{s}\mathbbm{g}_{2}^{t},\mathbbm{g}_{1}^{k}\mathbbm{g}_{2}^{l})
=\zeta_{\mathbbm{m}}^{a[\frac{k+s}{\mathbbm{m}}]i}
\zeta_{\mathbbm{n}}^{b[\frac{k+s}{\mathbbm{m}}]j}\zeta_{\mathbbm{n}}^{d[\frac{t+l}{\mathbbm{n}}]j}.\end{equation}

\begin{lemma}\label{l3.6} \emph{(\cite[Proposition 3.9]{bgrc1})} The set $\{\Phi_{a,b,d}|0\leq a<\mathbbm{m}, 0\leq b<(\mathbbm{m},
\mathbbm{n}),0\leq d<\mathbbm{n}\}$ is a complete set of representatives of the normalized $3$-cocycles
on $\mathbb{Z}_{\mathbbm{m}}\times \mathbb{Z}_{\mathbbm{n}}$.
\end{lemma}

With this, it is not hard to find that
\begin{proposition} The $3$-cocycles $\Phi_{a,b,d}$ are abelian.
\end{proposition}
\begin{proof} By \eqref{eq3.1} it is clear that
$$\Phi_{a,b,d}{(x,y,z)}=\Phi_{a,b,d}(x,z,y)$$
for all $x,y,z\in \mathbb{G}$. It follows that $$\theta_g(x,y)=\theta_g(y,x)$$
for all $g,x,y\in \mathbb{G}$, which implies that $D^{\Phi_{a,b,d}}(\mathbb{G})$ is commutative.
\end{proof}

\begin{corollary} \label{c3.7}
All Yetter-Drinfeld modules over $(k\mathbb{G},\Phi_{a,b,d})$ are diagonal.
\end{corollary}

\subsection{Resolutions}
One of our key observations is that any $3$-cocycle $\Phi$ on $\mathbb{G}$ can be ``resolved" in
 a suitable bigger abelian group $G$. More precisely, we may take $G=\mathbb{Z}_{m} \times \mathbb{Z}_n =\langle g_1 \rangle \times \langle g_2 \rangle$
with $m=\mathbbm{m}^2, \ n=\mathbbm{n}^2$ and the canonical epimorphism
$$\pi:\;G\to \mathbbm{G},\;\;\;\;g_{1}\mapsto \mathbbm{g}_{1},\;g_{2}\mapsto \mathbbm{g}_{2}.$$
By pulling back the $3$-cocycles on $\mathbb{G}$ along $\pi$ one gets $3$-cocycles on $G$. Therefore, for any $a,b,d$, the map
$$\pi^{\ast}(\Phi_{a,b,d}):\;G\times G\times G\to k^{\ast},\;\;(g,h,z)\mapsto \Phi_{a,b,d}(\pi(g),\pi(h),\pi(z)),\;\;\;\forall g,h,z\in G$$
becomes a $3$-cocycle on $G$. The observation is that $\pi^{\ast}(\Phi_{a,b,d})$ is indeed a coboundary. In fact, consider the following map
\begin{equation} \label {eq3.2} J_{a,b,d}:\; G\times G\to k^{\ast};\;\;\;\;(g_{1}^{x_{1}}g_{2}^{x_{2}},g_{1}^{y_{1}}g_{2}^{y_{2}})\mapsto
\zeta_{m}^{ax_{1}(y_{1}-y_{1}')}\zeta_{\mathbbm{m}\mathbbm{n}}^{bx_{2}(y_{1}-y_{1}')}\zeta_{n}^{dx_{2}(y_{2}-y_{2}'')}
\end{equation}
where $y_1'$ is the remainder of $y_1$ divided by $\mathbbm{m}$ (resp. $y_2''$ is the remainder of $y_2$ divided by $\mathbbm{n}$).
Here we require that $\zeta_{m}^{\mathbbm{m}}=\zeta_{\mathbbm{m}\mathbbm{n}}^{\mathbbm{n}}=\zeta_{\mathbbm{m}}$ and
$\zeta_{n}^{\mathbbm{n}}=\zeta_{\mathbbm{m}\mathbbm{n}}^{\mathbbm{m}}=\zeta_{\mathbbm{n}}$. Of course, this requirement can
be easily satisfied. For example, just take $\zeta_{t}=e^{\frac{2\pi i}{t}}$ for $t\in \mathbb{N}$. Thus, we have

\begin{proposition}\label{l3.8} The differential of $J_{a,b,d}$ is equal to $\pi^{\ast}(\Phi_{a,b,d})$, that is, $\partial(J_{a,b,d})=\pi^{\ast}(\Phi_{a,b,d}).$
\end{proposition}
\begin{proof} Indeed, by direct computation
\begin{eqnarray*}
& &\partial(J_{a,b,d})(g_{1}^{i_1}g_{2}^{i_2},g_{1}^{j_1}g_{2}^{j_2},g_{1}^{k_1}g_{2}^{k_2}) \\
&=&\frac{J(g_{1}^{j_1}g_{2}^{j_2},g_{1}^{k_1}g_{2}^{k_2})
J(g_{1}^{i_1}g_{2}^{i_2},g_{1}^{j_1+k_1}g_{2}^{j_2+k_2})}{J(g_{1}^{i_1+j_1}g_{2}^{i_2+j_2},g_{1}^{k_1}g_{2}^{k_2})J(g_{1}^{i_1}g_{2}^{i_2},g_{1}^{j_1}g_{2}^{j_2})}\\
&=&\frac{\zeta_{m}^{aj_{1}(k_{1}-k_{1}')}\zeta_{\mathbbm{m}\mathbbm{n}}^{bj_{2}(k_{1}-k_{1}')}\zeta_{n}^{dj_{2}(k_{2}-k_{2}'')}
\zeta_{m}^{ai_{1}((j_1+k_{1})-(j_1+k_{1})')}\zeta_{\mathbbm{m}\mathbbm{n}}^{bi_{2}((j_1+k_{1})-(j_1+k_{1})')}\zeta_{n}^{di_{2}((j_2+k_{2})-(j_2+k_{2})'')}}
{\zeta_{m}^{a(i_1+j_{1})(k_{1}-k_{1}')}\zeta_{\mathbbm{m}\mathbbm{n}}^{b(i_2+j_{2})(k_{1}-k_{1}')}\zeta_{n}^{d(i_2+j_{2})(k_{2}-k_{2}'')}
\zeta_{m}^{ai_{1}(j_1-j_1')}\zeta_{\mathbbm{m}\mathbbm{n}}^{bi_{2}(j_1-j_1')}\zeta_{n}^{di_{2}(j_2-j_2'')}}\\
&=&\frac{\zeta_{m}^{ai_{1}((j_1+k_{1})-(j_1+k_{1})')}\zeta_{\mathbbm{m}\mathbbm{n}}^{bi_{2}((j_1+k_{1})-(j_1+k_{1})')}\zeta_{n}^{di_{2}((j_2+k_{2})-(j_2+k_{2})'')}}
{\zeta_{m}^{ai_{1}(j_1+k_{1}-j_1'-k_{1}')}\zeta_{\mathbbm{m}\mathbbm{n}}^{bi_{2}(j_1+k_{1}-j_1'-k_{1}')}\zeta_{n}^{di_{2}(j_2+k_{2}-j_2''-k_{2}'')}}\\
&=&\zeta_{\mathbbm{m}}^{ai_1'[\frac{j_1'+k_1'}{\mathbbm{m}}]}
\zeta_{\mathbbm{n}}^{bi_2''[\frac{j_1'+k_1'}{\mathbbm{m}}]}\zeta_{\mathbbm{n}}^{di_2''[\frac{j_2''+k_2''}{\mathbbm{n}}]}\\
&=&\pi^{\ast}(\Phi_{a,b,d})(g_{1}^{i_1}g_{2}^{i_2},g_{1}^{j_1}g_{2}^{j_2},g_{1}^{k_1}g_{2}^{k_2}).
\end{eqnarray*}
\end{proof}

\section{Generation in degree one}
Throughout this section, $\MM$ is a finite-dimensional connected coradically graded pointed Majid algebra of rank $2$. The aim of this section is to prove that $\MM$ is generated by $\MM_0$ and $\MM_1$. Recall that, as before, $$\MM_{0}=(k\mathbb{G},\Phi)$$
and we can assume that
$$\mathbb{G}=\mathbb{Z}_{\mathbbm{m}}\times \mathbb{Z}_{\mathbbm{n}} = \langle \mathbbm{g}_1 \rangle \times  \langle \mathbbm{g}_2 \rangle $$
with $\mathbbm{m}|\mathbbm{n}$ by Lemma \ref{l2.8}, and
$$\Phi=\Phi_{a,b,d}$$
for some $0\leq a,b\leq \mathbbm{m}-1,\;0\leq d\leq \mathbbm{n}-1$ by Lemma \ref{l3.6}. Thanks to Proposition \ref{p4.5}, we have
$$\MM=\R\# k\mathbb{G}.$$
Note that $\R=\bigoplus_{i\in \mathbb{N}}\R_i$ is also coradically graded. The main result of this section can be stated as follows:

 \begin{proposition}\label{np1} In ${^{\mathbb{G}}_{\mathbb{G}}\mathcal{YD}^{\Phi}}$, we have $$\R\cong \mathscr{B}(\R_1).$$
 \end{proposition}

The proof of this proposition is divided into three steps and each of them appears as a subsection.

 \subsection{The $\mathbbm{Z}^{l}$-grading of $\mathscr{B}(V)$}
In this subsection, $G$ stands for an arbitrary finite abelian group and $\Phi$ a $3$-cocycle on $G.$ Let $V$ be a diagonal Yetter-Drinfeld module in ${_G^G \mathcal{YD}^\Phi}$ with dimension $l$. Let $V=\bigoplus_{i=1}^{l}kX_i$ be a decomposition of $V$ into the direct sum of $1$-dimensional Yetter-Drinfeld modules. Let $\Z^l$ be the free abelian group of rank $l$ and $e_i (1\leq i\leq l)$ be the canonical generators of $\Z^l$.

\begin{lemma}\label{nl1}
There is a $\Z^{l}$-grading on the Nichols algebra $\mathscr{B}(V)\in {_G^G \mathcal{YD}^\Phi}$ by setting $\deg X_i=e_i$.
\end{lemma}

\begin{proof}
Obviously, there is a $\Z^{l}$-grading on the tensor algebra $T_{\Phi}(V)\in {_G^G \mathcal{YD}^\Phi}$ by assigning $\deg X_i=e_i$. Let $I=\oplus_{i\geq i_0} I_i$ be the maximal graded Hopf ideal generated by homogeneous elements of degree greater than or equal to $2.$ To prove that $\mathscr{B}(V)$ is $\Z^l$-graded, it amounts to prove that $I$ is $\Z^l$-graded. This will be done by induction on the $\N$-degree.

First let $X\in I$ be a homogenous element with minimal degree $i_0.$ Since $\Delta(X)\in T(V)\otimes I+ I\otimes T(V),$ $X$ must be a primitive element, i.e., $\Delta(X)=X\otimes 1+1\otimes X.$
Suppose $X=X^1+X^2+\cdots +X^n,$ where $X^i$ is $\Z^l$-homogenous, and $X^i$ and $X^j$ have different $\Z^l$-degrees if $i\neq j.$ Suppose $\Delta(X^i)=X^i\otimes 1+1\otimes X^i+(X^i)_1\otimes (X^i)_2.$ Then we have  $\sum_{i=1}^n(X^i)_1\otimes(X^i)_2=0$. This forces $(X^i)_1\otimes (X^i)_2 = 0$ for each $1\leq i\leq l$ since $\Delta$ preserves the $\Z^l$-degrees. So, each $X^i$ is a primitive element and hence must be contained in $I$ by the maximality of $I.$ Therefore, $I_{i_0}$ is $\Z^l$-graded.

Then suppose that $I^k:=\oplus_{i_0\leq i\leq k} I_i$ is $\Z^l$-graded. We shall prove that $I^{k+1}=\oplus_{i_0\leq i\leq k+1} I_i$ is also $\Z^l$-graded. Let $X\in I_{k+1}$ and $X=X^1+X^2+\cdots +X^n,$ with each $X^i$ being $\Z^l$-homogenous and $X^i$ and $X^j$ having different $\Z^l$-degrees if $i\neq j.$ Write $\Delta(X^i)=X^i\otimes 1+1\otimes X^i+(X^i)_1\otimes (X^i)_2.$ Since $\Delta(X)=X\otimes 1+1\otimes X+(X)_1\otimes (X)_2,$ where $(X)_1\otimes (X)_2\in T(V)\otimes I^k+I^k\otimes T(V),$ i.e., $\sum (X^i)_1\otimes (X^i)_2\in T(V)\otimes I^k+I^k\otimes T(V).$ According to the inductive assumption, $T(V)\otimes I^k+I^k\otimes T(V)$ is a $\Z^l$-graded space. So each $(X^i)_1\otimes (X^i)_2\in T(V)\otimes I^k+I^k\otimes T(V)$ as $\Delta$ preserves $\Z^l$-degrees. If there was an $X^i\notin I_{k+1},$ then $I+\langle X^i\rangle$ is a Hopf ideal properly containing $I,$ which contradicts to the maximality of $I.$ It follows that $X^i\in I_{k+1}$ for all $1\leq i\leq n$ and hence $I^{k+1}$ is also $\Z^l$-graded by the assumption on $X.$ We complete the proof of the lemma.
\end{proof}

Now return to our Majid algebra $\MM$. Since it is assumed of
rank $2$, $\dim\R_1=2$. By Corollary \ref{c3.7}, $\R_1$ is a diagonal Yetter-Drinfeld module over $(k\mathbb{G},\Phi).$
Therefore, we may write \[ \R_1=V_1 \oplus V_2 = k X_1 \oplus k X_2 \] as the direct sum of two $1$-dimensional Yetter-Drinfeld modules. As in Section 3, consider a bigger abelian group $G=\mathbb{Z}_{m}\times \mathbb{Z}_n = \langle g_1 \rangle \times  \langle g_2 \rangle$
with $m=\mathbbm{m}^2, n=\mathbbm{n}^2$ and the canonical epimorphism:
$$\pi:\;kG\to k\mathbbm{G},\;\;\;\;g_{1}\mapsto \mathbbm{g}_{1},\;g_{2}\mapsto \mathbbm{g}_{2}.$$
Observe that $\pi$ has a section
$$\iota:\;k\mathbbm{G}\to kG,\;\;\;\;\mathbbm{g}^{i}_{1}\mathbbm{g}^{j}_{2}\mapsto {g}^{i}_{1}g_{2}^j$$
which is not a group morphism.
Let $\delta_L$ and $\triangleright$ be the comodule and module structure maps of $\R\in {_\mathbbm{G}^\mathbbm{G} \mathcal{YD}^\Phi}$. Define
\begin{eqnarray*}
&&\rho_{L}:\;\R_1 \to kG\otimes \R_1,\;\;\;\;\rho_{L}=(\iota\otimes \id)\delta_{L}\\
&&\blacktriangleright:\;kG\otimes \R_1\to \R_1,\;\;\;\;g\blacktriangleright Z=\pi(g)\triangleright Z
\end{eqnarray*}
for all $g\in G$ and $Z\in \R_1$. Through this way, $\R_1\in {_G^G \mathcal{YD}^{\pi^*(\Phi)}}$ and this can be verified by direct computation:
\begin{eqnarray*}
e\blacktriangleright(f\blacktriangleright Z)&=&\pi(e)\triangleright(\pi(f)\triangleright Z)\\
&=&\tilde{\Phi}_{z}(\pi(e),\pi(f))(\pi(e)\pi(f))\triangleright Z\\
&=&\widetilde{\pi^{\ast}(\Phi)}_{\iota(z)}(e,f)ef\blacktriangleright Z
\end{eqnarray*}
for all $e,f\in G$ and $\delta_{L}(Z)=z\otimes Z$ for $Z\in \R_1$. We denote this new Yetter-Drinfeld module by $\widetilde{\R_1}$ in order to distinguish from the original one $\R_1 \in {_\mathbbm{G}^\mathbbm{G} \mathcal{YD}^\Phi}.$ From these, we have two essentially identical Nichols algebras $\mathscr{B}(\R_1)\in {_\mathbbm{G}^\mathbbm{G} \mathcal{YD}^\Phi}$ and $\mathscr{B}(\widetilde{\R_1})\in {_G^{G} \mathcal{YD}^{\pi^{*}(\Phi)}}$ which however live in different environment.

\begin{lemma}\label{nl2}
There is a linear isomorphism $F:\mathscr{B}(\R_1)\To \mathscr{B}(\widetilde{\R_1})$ which preserves the product and coproduct of these two algebras.
\end{lemma}
\begin{proof}
Let $F: T_{\Phi}(\R_1)\to T_{\pi^{*}(\Phi)}(\widetilde{\R_1})$ be the multiplicative linear map which preserves $\R_1,$ i.e., $F|_{\R_1}=\id_{\R_1}.$ It is easy to show that $F$ also preserves the comultiplication of $T_{\Phi}(\R_1)$ and $T_{\pi^{*}(\Phi)}(\widetilde{\R_1}).$
Note that $F$ induces a one to one correspondence between the set of $\Z^2$-graded Hopf ideals of $T_{\Phi}(\R_1)$ and that of $T_{\pi^{*}(\Phi)}(\widetilde{\R_1}).$ By Lemma \ref{nl1}, we know that the maximal Hopf ideals generated by homogeneous elements of degree $\geq 2$ in $T_{\Phi}(\R_1)$ and in $T_{\pi^{*}(\Phi)}(\widetilde{\R_1})$ are $\Z^2$-graded. It is obvious that $F$ maps the maximal Hopf ideal of $T_{\Phi}(\R_1)$ to that of $T_{\pi^{*}(\Phi)}(\widetilde{\R_1}).$ Therefore, $F$ induces a linear isomorphism from $\mathscr{B}(\R_1)$ to $\mathscr{B}(\widetilde{\R_1}),$ which preserves multiplication and comultiplication.
\end{proof}

\begin{remark}
Via this isomorphism, we also view $\mathscr{B}(\R_1)$ as a Nichols algebra in ${_G^{G} \mathcal{YD}^{\pi^{*}(\Phi)}}$. In particular, $\mathscr{B}(\R_1)$  is a Yetter-Drinfeld module over $(kG, \pi^{*}(\Phi))$.
\end{remark}

\subsection{Twisted version of ordinary Nichols algebras}
Again, first suppose $G$ is an arbitrary finite abelian group. Let $V$ be a diagonal Yetter-Drinfeld module in ${^G_G\mathcal{YD}}$ and let $V=\bigoplus_{i=1}^{N}kX_i$ be a decomposition of $V$ into $1$-dimensional Yetter-Drinfeld modules. We use $\delta_L$ and $\triangleright$ to denote the comodule and module structure maps of $V$. Then there are $g_i\in G$ and $q_{ij}\in k^*$ such that $\delta_L(X_i)=g_{i} \otimes X_i$ and $g_{i}\triangleright X_j=q_{ij}X_j.$ Let $J$ be a $2$-cochain on $G$ and let $\Phi$ denote its differential $\partial(J).$ Recall that $\mathscr{B}(V)^{J}$ is defined in Subsection 2.5 and we have $\mathscr{B}(V)^J \cong \mathscr{B}(V^J)$ as Nichols algebras in ${_G^G \mathcal{YD}^{\partial(J)}}.$

Now assume that $\mathscr{B}(V)\in \ _G^G \mathcal{YD}$ is  finite-dimensional,  then $\mathscr{B}(V)=T(V)/I$ where $I$ is the Hopf ideal of $T(V)$ generated by the polynomials listed in \cite[Theorem 3.1]{Ang}. In the following, let $\mathbf{S}$ denote the set of these polynomials. Preserve the notations of Subsection 2.5. Define a map $\Psi: T_{\Phi}(V^J)\to T(V)$ by
\begin{equation} \label{ne1}
\Psi((\cdots ((Y_1\circ Y_2)\circ Y_3)\cdots Y_n))=\prod_{i=1}^{n-1}J(Y_1\cdots Y_i,Y_{i+1})Y_1Y_2\cdots Y_n
\end{equation} for all $Y_i\in \{X_1,X_2,\ldots,X_N\}.$ It is easy to see that $\Psi$ is an isomorphism of linear spaces.

\begin{lemma}\label{nl5}
The set $\Psi^{-1}(\mathbf{S})$ is a minimal set of defining relations of $\mathscr{B}(V)^J.$
\end{lemma}
\begin{proof}
We claim that
\begin{equation} \label{ne2}
\Psi(E\circ F)=J(e,f)\Psi(E)\Psi(F)
\end{equation} for any homogeneous $E,F\in T_{\Phi}(V^J).$
We prove this by induction on the length of $F.$

If the length of $F$ is $1,$ i.e., $F\in\{X_1,X_2,\ldots,X_N\},$ then \eqref{ne2} is just a special case of \eqref{ne1}. Next assume \eqref{ne2} holds for all $F$ of length $<n.$ If $F$ is of length $n,$ then we can write $F=F'X_i$ for some $1\le i\le N.$ And we have
\begin{eqnarray*}
\Psi(E\circ F)&=&\Psi((E\circ (F'\circ X_i))\\
             &=&\Phi^{-1}(e,f',x_i)\Psi((E\circ F')\circ X_i)\\
             &=&\Phi^{-1}(e,f',x_i)J(ef',x_i)\Psi(E\circ F')\Psi(X_i)\\
             &=&\Phi^{-1}(e,f',x_i)J(ef',x_i)J(e,f')\Psi(E)\Psi(F')\Psi(X_i)\\
             &=&\Phi^{-1}(e,f',x_i)J(ef',x_i)J(e,f')J^{-1}(f',x_i)\Psi(E)\Psi(F'\circ X_i)\\
             &=&J(e,f)\Psi(E)\Psi(F).
\end{eqnarray*}

Now with \eqref{ne2}, it is clear that $T_{\Phi}(V^J)$ is exactly $T(V)^J.$
Let $I'$ be the ideal generated by $\Psi^{-1}(\mathbf{S}),$
then one can easily show that $\Psi$ induces an isomorphism $\overline{\Psi}:T_{\Phi}(V^J)/I'\to \mathscr{B}(V)$ such that
\begin{equation}
\overline{\Psi}(E\circ F)=J(e,f)\overline{\Psi}(E)\overline{\Psi}(F)
\end{equation} for any homogeneous $E,F\in T_{\Phi}(V^J)/I.$ So $T_{\Phi}(V^J)/I'$ is actually $\mathscr{B}(V)^J.$ By Lemma \ref{tn}, we see that $T_{\Phi}(V^J)/I'$ is a Nichols algebra in $_G^G \mathcal{YD}^{\Phi}.$
The minimality of $\Psi^{-1}(\mathbf{S})$ follows from that of $\mathbf{S}.$
\end{proof}

\begin{lemma}\label{nl6}
Let $Z$ be a polynomial in $\mathbf{S},$ then we have $\mathscr{B}(V^J\oplus k\Psi^{-1}(Z))\cong \mathscr{B}(V\oplus kZ)^J.$
\end{lemma}
\begin{proof}
It suffices to prove that $(V\oplus kZ)^J\cong V^J\oplus k\Psi^{-1}(Z)$ as objects in ${_G^G \mathcal{YD}^{\Phi}}.$ But this is clear.
\end{proof}

The following lemma is a summary of the important results in \cite[Section 4]{Ang}, which is also crucial for the present paper.
\begin{lemma}\label{nl7}
Let $\mathscr{B}(V)\in \ _G^G\mathcal{YD}$ be a finite-dimensional Nichols algebra of diagonal type. Suppose that $Z$ is a polynomial in $\mathbf{S}$ and $U=V\oplus kZ,$ then $\mathscr{B}(U)$ is infinite-dimensional.
\end{lemma}
\begin{proof}
If $Z$ is of the form (3.2) or (3.5-3.29) listed in \cite[Theorem 3.1]{Ang}, then from the proofs of \cite[Proposition 4.1]{Ang} and \cite[Lemmas 4.2-4.12]{Ang}, we know that the root systems of $\mathscr{B}(U)$ is not in the Heckenberger's list of classification of finite arithmetic root systems. Hence $\mathscr{B}(U)$ must be infinite-dimensional.

If $Z$ is of the form (3.1) or (3.3-3.4) listed in \cite[Theorem 3.1]{Ang}, then from the proof of \cite[Theorem 4.13]{Ang} we know that the subalgebra of $\mathscr{B}(U)$ generated by $Z$ is infinite-dimensional. Therefore $\mathscr{B}(U)$ is also infinite-dimensional.
\end{proof}

In the rest of this subsection, we return to the case of $\MM$, i.e., $\mathbbm{G}=\Z_{\m}\times \Z_{\n}$ and $\Phi=\Phi_{a,b,d},J=J_{a,b,d}$. The following result is a generalization of \cite[Theorem 4.13]{Ang} to pointed Majid algebras.
\begin{proposition}\label{np2}
Let $R=\oplus_{i\geq 0} R_i$ be a finite-dimensional graded (not necessarily coradically graded) Hopf algebra in ${_\mathbbm{G}^\mathbbm{G}\mathcal{YD}^\Phi}$ such that $R_0=k1$ and $\dim_{k}R_1=2$. If $R$ is generated by $R_1,$ then $R=\mathscr{B}(R_1).$
\end{proposition}
\begin{proof}
Let $I$ be an ideal of $T_{\Phi}(R_1)$ such that $R=T_{\Phi}(R_1)/I.$ Clearly, we have a surjective Hopf map $$\theta:\; R\twoheadrightarrow \mathscr{B}(R_1).$$  By Lemma \ref{nl2}, $\mathscr{B}(R_1)$ is also a Nichols algebra in ${^G_G\mathcal{YD}^{\pi^{*}(\Phi)}}$ for $G=\Z_m\times \Z_n$. By Proposition \ref{l3.8}, $\pi^{*}(\Phi)=\partial{J}$. Therefore, $\mathscr{B}{(R_1)}^{J^{-1}} \in {^G_G\mathcal{YD}}$ is a usual Nichols algebra. Now assume that $\theta$ is not an isomorphism, then there should be some polynomials in $\Psi^{-1}(\mathbf{S}),$ which are not contained in $I$ by Lemma \ref{nl5}. Suppose that $\Psi^{-1}(Z)$ is one of those with minimal length. From the proof of Lemma \ref{nl1}, we know that $\Psi^{-1}(Z)$ must be a primitive element in $R.$ Let $U=R_1\oplus k\Psi^{-1}(Z),$ then by the preceding assumption there is an embedding of linear spaces $\mathscr{B}(U)\to R.$

 We already know that $\mathscr{B}(R_1)^{J^{-1}}$ is a finite-dimensional Nichols algebra in $_G^G\mathcal{YD}$. By Lemma \ref{tn}, there exists $R_1'\in {_G^G\mathcal{YD}}$ such that $R_1=R_1'^J$.  By Lemma \ref{nl6}, we have $\mathscr{B}(R_1'\oplus Z)^J=\mathscr{B}(R_1\oplus k\Psi^{-1}(Z))=\mathscr{B}(U).$ Note that $\mathscr{B}(R_1'\oplus Z)^J$ is infinite-dimensional due to Lemma \ref{nl7}. Hence $\mathscr{B}(U)$ is infinite-dimensional, which contradicts to the assumption that $R$ is finite-dimensional. Thus $\theta$ is an isomorphism and $R$ is the Nichols algebra $\mathscr{B}(R_1).$
\end{proof}

\subsection{Dual Hopf algebras and the proof of Proposition \ref{np1}}

For any Hopf algebra $H$ in the category of (twisted) Yetter-Drinfeld modules, we use $P(H)$ to denote the set of primitive elements of $H$, that is, $P(H)=\{X\in H| \Delta(X)=1\otimes X+X\otimes 1\}.$  The following lemma is a generalization of \cite[Lemma 5.5]{as}.

\begin{lemma}\label{nl8}
Let $R=\oplus_{i\geq 0}R_i$ be a graded Hopf algebra in $_\mathbbm{G}^\mathbbm{G} \mathcal{YD}^\Phi$ with $R_0=k1$ and $P(R)=R_{1}.$ Then $R^*=\oplus_{i\geq 0}R_i^*$ $(resp. \ ^*R=\oplus_{i\geq 0}{^*R_i})$ is generated by $R_{1}^*$ $(resp. \ {^*R_{1}}).$
\end{lemma}
\begin{proof}
Let $$m_{R^*}^{\overrightarrow{s}}=m_{R^*}\circ(m_{R^*}\otimes \id_{R^*})\circ\cdots \circ(m_{R^*}\otimes \underbrace{\id_{R^*}\cdots \otimes \id_{R^*}}_{s-1})$$ and $$\D^{\overrightarrow{s}}=(\cdots(\D\otimes \underbrace{\id_{R})\cdots \id_{R}}_{s-1})\circ \cdots \circ (\D_{R}\otimes \id_{R})\circ \D_{R}.$$
For $f^1,f^2,\cdots, f^s\in R_{1}^*$ and $X\in R_s,$ we have

\begin{eqnarray*}
m_{R^*}^{\overrightarrow{s-1}}(\cdots(f^1\otimes f^2)\cdots f^s)(X)&=&(\cdots(f^1\otimes f^2)\cdots f^s)(\D_{R}^{\overrightarrow{s-1}}(X))\\
&=&(\cdots(f^1\otimes f^2)\cdots f^s)(\pi^{\overrightarrow{s}}\circ \D_{R}^{\overrightarrow{s-1}}(X))
\end{eqnarray*}
where $\pi^{\overrightarrow{s}}=(\cdots(\underbrace{\pi\otimes \pi)\cdots \pi}_{s})$ and $\pi:R\to R_{1}$ is the canonical projection. Hence as linear maps, $m_{R^*}^{\overrightarrow{s-1}}$ are dual to $\pi^{\overrightarrow{s}}\circ\D^{\overrightarrow{s-1}}.$ So for all $s\geq 2,$ $\Gamma_s=\pi^{\overrightarrow{s}}\circ \D^{\overrightarrow{s-1}}:R_s\to R_{1}^{\overrightarrow{s}}$ is injective if and only if $m_{R^*}^{\overrightarrow{s-1}}:R^*(1)^{\overrightarrow{s}}\to R^*_s$ is surjective.

Therefore, to prove the claim it is enough to show that  $P(R)=R_{1}$ if and only if $\Gamma_s$ is injective for all $s\geq 2.$ On the one hand, assume that the $\Gamma_s$ are injective. If $X\in R_s$ is a primitive element for some $s \geq 2,$ then $\Gamma_s(X)=0$ by definition, and hence $X=0.$ So it follows that $P(R)=R_{1}.$

On the other hand, assume that $P(R)=R_{1}.$ We will use induction to prove that each $\Gamma_s$ is injective for all $s\geq 2.$ If $s=2,$ for an $X\in R_2$ we have $\D(X)=X\otimes 1+1\otimes X+ X_1\otimes X_2$ where $X_1, X_2\in R_{1}$ since $\D(R_N)\subset \bigoplus_{i=0}^NR_{N-i}\otimes R_i.$ So  $\Gamma_2(X)=0$ implies $X_1\otimes X_2=0,$ hence $\D(X)=1\otimes X+X\otimes 1$ is a primitive element. But $P(R)=R_{1},$ so $X=0.$ Denote $$A_{s,t}(V)=(\cdots(a^{-1}_{V^{\overrightarrow{s}},V,V}\otimes \underbrace{\id_V)\cdots \id_V}_{t-2})\circ (\cdots(a^{-1}_{V^{\overrightarrow{s}},V^{\overrightarrow{2}},V}\otimes \underbrace{\id_V)\cdots \id_V}_{t-3})\cdots \circ a^{-1}_{V^{\overrightarrow{s}},V^{\overrightarrow{t-1}},V}.$$ Thus we have
$$\Gamma_{s+1}=A_{l,s-1}(R_{1})\circ(\Gamma_l\otimes \Gamma_{s-l})\circ \D$$ for $1\leq l\leq s.$
If $\Gamma_{s+1}(X)=0$ for $X\in R_{s+1},$ then we have $(\Gamma_l\otimes \Gamma_{s-l})\circ \D(X)=0$ since $A_{l,s-1}(R_{1})$ is an isomorphism. Write $\D(X)=X\otimes 1+1\otimes X+X_1\otimes X_2+Y_1\otimes Y_2,$ here $X_1\in R_l,X_2\in R_{s-l},$ $Y_1\otimes Y_2\notin R_l\otimes R_{s-l}.$ Note that $(\Gamma_l\otimes \Gamma_{s-l})(Y_1\otimes Y_2)=0$ by definition. So we have $(\Gamma_l\otimes \Gamma_{s-l})(X_1\otimes X_2)=0$ since $\Gamma_{s+1}(X)=0,$ which implies $X_1\otimes X_2=0$ since $\Gamma_l$ are injective for $l<s+1$ by induction. Since $l$ is arbitrary, we have  $\D(X)\in R_{s+1}\otimes R_{0}\oplus R_{0}\otimes R_{s+1}.$ Hence $X$ must be primitive, which implies $X=0.$
\end{proof}

\noindent \textbf{Proof of Proposition \ref{np1}.} By assumption, $\R_0=k1$ and $P(\R)=\R_1$ has dimension $2.$ According to Lemma \ref{nl8}, $\R^*=\oplus_{i\geq 0}\R_i^*$ is generated by $\R_{1}^*.$ By Proposition \ref{np2}, $\R^*=\mathscr{B}(\R_{1}^*).$ So we have $P(\R^*)=\R_{1}^*,$ and  $^*(\R^*)=\R$ is generated by $\R_{1}$ according to Lemma \ref{nl8} again. Hence $\R$ is also a Nichols algebra again by Proposition \ref{np2}. Thus, $\R=\mathscr{B}(\R_{1})$.

\begin{corollary}\label{nc1}
Let $H$ be a finite-dimensional pointed Majid algebras of rank $2$, then $H$ is generated by group-like and skew-primitive elements.
\end{corollary}
\begin{proof}
It is clear that $H$ is generated by group-like and skew-primitive elements if and only if $\gr(H)$ is.  So we may assume that $H$ is coradically graded. Let $H_0=k \mathbbm{G},$ then the coinvariant subalgebra $\R$ is a graded Hopf algebra in $_\mathbbm{G}^\mathbbm{G} \mathcal{YD}^\Phi$. By Proposition \ref{np1}, $\R$ is generated by primitives. This clearly leads to the claim.
\end{proof}

\section{Classification procedure and the main result}
In this section, a theoretical procedure to classify graded connected pointed Majid algebras of rank $2$ is provided. The procedure is applied to get our main classification result, that is, Theorem \ref{t5.17}.

\subsection{General setup.}
In this section, we always assume that $\MM$ is a finite-dimensional connected coradically graded pointed Majid algebra of rank $2$. Keep the notations of Section 4. Recall that, $\MM_{0}=(k\mathbb{G},\Phi)$ where
$\mathbb{G}=\mathbb{Z}_{\mathbbm{m}}\times \mathbb{Z}_{\mathbbm{n}} = \langle \mathbbm{g}_1 \rangle \times  \langle \mathbbm{g}_2 \rangle $
with $\mathbbm{m}|\mathbbm{n}$ and $\Phi=\Phi_{a,b,d}$ for some $0\leq a,b\leq \mathbbm{m}-1,\;0\leq c\leq \mathbbm{n}-1$. By the bosonization procedure, we have
$\MM=\R\# k\mathbb{G}$ and $\dim\R_1=2$ since $\MM$ is assumed of rank $2.$
By Corollary \ref{c3.7}, $\R_1$ is a diagonal Yetter-Drinfeld module over $(k\mathbb{G},\Phi)$ .
Therefore, we may write \[ \R_1=V_1 \oplus V_2 = k X_1 \oplus k X_2 \] as a direct sum of two $1$-dimensional Yetter-Drinfeld modules. As in Subsection 3.3, we consider the bigger abelian group $G=\mathbb{Z}_{m}\times \mathbb{Z}_n = \langle g_1 \rangle \times  \langle g_2 \rangle$
with $m=\mathbbm{m}^2, n=\mathbbm{n}^2$ and the canonical epimorphism
$$\pi:\;kG\to k\mathbbm{G},\;\;\;\;g_{1}\mapsto \mathbbm{g}_{1},\;g_{2}\mapsto \mathbbm{g}_{2}$$
with a typical section
$$\iota:\;k\mathbbm{G}\to kG,\;\;\;\;\mathbbm{g}^{i}_{1}\mathbbm{g}^{j}_{2} \mapsto {g}^{i}_{1}g_{2}^j.$$

\subsection{From Majid algebras to Hopf algebras.}
It was shown in Subsection 2.6 that the coinvariant subalgebra $\R$ of $\MM$ is a Hopf algebra in $^{\mathbbm{G}}_{\mathbbm{G}}\mathcal{YD}^{\Phi}$. Moreover, $\R=\mathscr{B}(\R_{1})\in {^{\mathbbm{G}}_{\mathbbm{G}}\mathcal{YD}^{\Phi}}$ by Proposition \ref{np1}. Thanks to Lemma \ref{nl2}, $\R$ is also a Hopf algebra in $^{G}_{G}\mathcal{YD}^{\pi^{*}(\Phi)}$ by extending the following Yetter-Drinfeld module structure on $\R_1$ to $\R$:
\begin{eqnarray}
&&\rho_{L}:\;\R_1 \to kG\otimes \R_1,\;\;\;\;\rho_{L}=(\iota\otimes \id)\delta_{L}\\
\label{eq5.2} &&\blacktriangleright:\;kG\otimes \R_1\to \R_1,\;\;\;\;g\blacktriangleright X=\pi(g)\triangleright X
\end{eqnarray}
for all $g\in G$ and $X\in \R_1$.

The following observation is contained implicitly in the proof of Proposition \ref{np1}. As it is the crux of our classification procedure, we include an explicit proof here.

\begin{proposition}\label{p5.6} The Majid bosonization
$$\emph{\textbf{M}}:=\R\# kG$$
as defined in Subsection 2.6 is twist equivalent to an ordinary connected pointed Hopf algebra. \end{proposition}
\begin{proof}
Note that the associator of $\mathbf M$ is provided by $\pi^{\ast}(\Phi_{a,b,d})$ for some $0\leq a,b<\mathbbm{m},\ 0\leq d<\mathbbm{n}$. Then by Proposition \ref{l3.8}, there exists a 2-cochain $J$ on $G$ such that $\pi^{\ast}(\Phi_{a,b,d})=\partial J$. Clearly, ${\bf{M}}^{J^{-1}}$ is an ordinary Hopf algebra, see Subsection 2.1. Finally, the connectedness of $\mathbf M$ implies that of ${\bf{M}}^{J^{-1}}.$
\end{proof}

The following example, though of rank 1, provides an explanation of the previous proposition.
\begin{example}\label{ex5.4} \emph{ Take $M(\mathbbm{n},1,q)$ as given in Example \ref{e2.7}. Recall that \[
\Phi (\mathbbm{g}^{i},\mathbbm{g}^{j},\mathbbm{g}^{k})=\mathbbm{q}^{i[\frac{j+k}{\mathbbm{n}}]}.\]
In this case, the coinvariant subalgebra $\R=k[X]/(X^{n})$ with $n=\mathbbm{n}^2,$ and $\R_1=kX.$ Then by equations \eqref{eq2.10} and \eqref{eq4.1}, we have
\begin{equation}\mathbbm{g}^{i}\triangleright X=(\mathbbm{q}q)^{i}X.
\end{equation}
Let $\Z_{n}=\langle g \rangle$ and $\pi\colon \Z_n\to \Z_{\mathbbm{n}},\; g\mapsto \mathbbm{g}$. This $\Z_{\mathbbm{n}}$-action extends to a $\mathbb{Z}_{n}$-action $\blacktriangleright$ according to \eqref{eq5.2}. More precisely,
$$g^{i}\blacktriangleright X=\left \{
\begin{array}{ll} \mathbbm{q}(\mathbbm{q}q)^{i'}X & \;\;\;\;\mathbbm{n}\nmid i,\\
 X& \;\;\;\;\mathbbm{n}\mid i.
\end{array}\right. $$
Now $\R$ becomes a Hopf algebra in $^{\Z_n}_{\Z_n}\mathcal{YD}^{\pi^{\ast}(\Phi)}$ and we get the Majid algebra $\R\# k\Z_{n}$ by the bosonization. Note that the associator of $\R\# k\Z_{n}$ is $\pi^{\ast}(\Phi)$, which is exactly the differential of the following $2$-cochain
$$J:\; k\mathbb{Z}_n\otimes k\mathbb{Z}_{n}\to k,\;\;\;\;(g^i,g^j)\mapsto q^{i(j-j')}.$$
Therefore, the associator of $(\R\# k\mathbb{Z}_{n})^{J^{-1}}$ is trivial by Subsection 2.1, and thus  $(\R\# k\mathbb{Z}_{n})^{J^{-1}}$ is an ordinary Hopf algebra. }\end{example}

\begin{remark}
Let $\R$ be a Hopf algebra in $^{G}_{G}\mathcal{YD}^{\Phi}$. Note that the bosonization $\R\# kG$ is an ordinary Hopf algebra if and only if each ${^{g}\R}$ is a linear representation of $G$.   As a matter of fact,  in the above example
the $\Z_{n}$-action on $\R$ in $(\R\# k\mathbb{Z}_{n})^{J^{-1}}$  is given by$$g^{i}\blacktriangleright^J X=\left \{
\begin{array}{ll} \mathbbm{q}^{-1}q^{i+(-i)'}(\mathbbm{q}q)^{i'}X & \;\;\;\;\mathbbm{n}\nmid i\\
 q^{i}X& \;\;\;\;\mathbbm{n}\mid i.
\end{array}\right. $$
From this, we always have
$$g^{i}\blacktriangleright^J (g^j\blacktriangleright^J X)=g^{i+j}\blacktriangleright^{J}X.$$
That is, we turn the projective representation into the usual linear representation. That is, $(\R\# k\mathbb{Z}_{n})^{J^{-1}}$ is a usual Hopf algebra, which in fact is the familiar Taft algebra $T_n$.
\end{remark}

So far, we have achieved the following one-way road map:
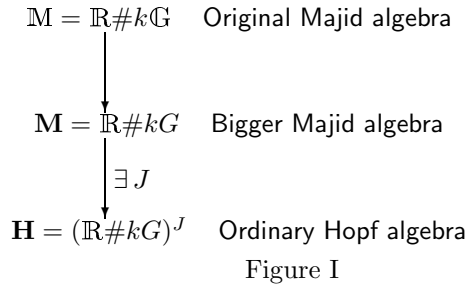
\begin{figure}[hbt]
\begin{picture}(100,100)(0,0)
\put(0,90){\makebox(0,0){$ \MM=\R\#k\mathbbm{G}\;\;\;$ \textsf{Original Majid algebra}}}
\put(-50,85){\vector(0,-1){30}}
\put(0,50){\makebox(0,0){$ {\textbf{M}}=\R\#kG\;\;\;$ \textsf{Bigger Majid algebra}}}
\put(-50,45){\vector(0,-1){30}}\put(-40,30){\makebox(0,0){$\exists \, J$}}
\put(0,10){\makebox(0,0){$ \textbf{H}=(\R\#kG)^{J}\;\;\;$ \textsf{Ordinary Hopf algebra}}}
\put(20,-5){\makebox(0,0){\textrm{Figure I}}}
\end{picture}
\end{figure}

This offers us a possible chance to take advantage of the successful theory of finite-dimensional pointed Hopf algebras to Majid algebras. Of course, our next task is to get a ``return ticket" from ordinary pointed Hopf algebras to genuine pointed Majid algebras.

\subsection{From Hopf algebras to Majid algebras.}
Although $\R$ is a ``twisted" version of a Nichols algebra, it is completely not clear yet when a ``twisted" version of a Nichols algebra is indeed the coinvariant subalgebra of a genuine Majid algebra. The aim of this subsection is to find an answer to this question.

According to Figure I, take $\mathscr{B}=\mathscr{B}(V)$ a finite-dimensional diagonal Nichols algebra of rank $2$ as classified in \cite{h}. Fix a decomposition $$V=kX_1\oplus kX_2$$
as a direct sum of $1$-dimensional Yetter-Drinfeld modules. Find two square integers $m=\mathbbm{m}^2,\;n=\mathbbm{n}^2$ with
$m\mid n$ and the abelian group $G=\mathbb{Z}_{m}\times \mathbb{Z}_{n}=\langle g_1\rangle \times \langle g_2\rangle$ such that $\mathscr{B}$ is a Nichols algebra in $^{G}_{G}\mathcal{YD}$.
Take $J=J_{a,b,d}$ as given in \eqref{eq3.2} and consider
$$(\mathscr{B}\# kG)^{J}=\mathscr{B}^{J}\# kG$$
where $(\mathscr{B}^{J}, \rho_{L}, \blacktriangleright)$ is regarded as a Nichols algebra in $^{G}_{G}\mathcal{YD}^{\partial(J)}$.
As before, let $\mathbbm{G}=\mathbb{Z}_{\mathbbm{m}}\times \mathbb{Z}_{\mathbbm{n}}=\langle \mathbbm{g}_1\rangle \times \langle \mathbbm{g}_2\rangle$.
Let $\Phi=\Phi_{a,b,d}$ as given in \eqref{eq3.1} which is a $3$-cocycle on $\mathbbm{G}$.
Keep the notations of Subsection 5.2. What we need to do is to reverse the first step of the procedure in Subsection 5.2, that is, find \begin{equation}
\delta_{L}:\;\mathscr{B}^{J}\to k\mathbbm{G}\otimes \mathscr{B}^{J},
\quad \triangleright:\;k\mathbbm{G}\otimes \mathscr{B}^J\to \mathscr{B}^J,
\end{equation}
such that $(\mathscr{B}^{J}, \delta_{L}, \triangleright)$ is a Nichols algebra in $^{\mathbbm{G}}_{\mathbbm{G}}\mathcal{YD}^{\Phi}$ and the restrictions of $\rho_{L},\ \blacktriangleright$ to $V$ are just $(\iota\otimes \id)\delta_{L},\ \pi(-)\triangleright$ respectively, i.e., $\rho_{L}|_{V}=(\iota\otimes \id)\delta_{L}|_{V},\ \blacktriangleright|_{V}=\pi(-)\triangleright|_{V}$. If this is the case, we call $(\mathscr{B}^{J}, \rho_{L}, \blacktriangleright)$  the {\it induced Nichols algebra} of $(\mathscr{B}^{J}, \delta_{L}, \triangleright).$

\begin{lemma} The aforementioned
$\delta_{L}$ and $\triangleright$ exist if and only if both $g_1^{\mathbbm{m}}$
and $g_2^{\mathbbm{n}}$ lie in the center of  $\mathscr{B}^{J}\# kG$.
\end{lemma}
\begin{proof} For brevity and for the consistence of the notations, denote the Majid algebra $(\mathscr{B}\# kG)^{J}$ by $\textbf{H}^{J}.$ Let $I$ be the right ideal of $\textbf{H}^{J}$ generated by $g_1^{\mathbbm{m}}-1$ and $g_2^{\mathbbm{n}}-1.$ That is,
$$I=(g_1^{\mathbbm{m}}-1)\textbf{H}^{J}+(g_2^{\mathbbm{n}}-1)\textbf{H}^{J}.$$
It is not hard to find that
\begin{equation} \label{eq5.6} \dim \textbf{H}^{J}-\dim I=\dim (\mathscr{B}^{J}\otimes k\mathbbm{G}).\end{equation}

\noindent Claim: \emph{$\mathscr{B}^{J}$ is a Nichols algebra in $^{\mathbbm{G}}_{\mathbbm{G}}\mathcal{YD}^{\Phi}$ such that its
induced Nichols algebra is $\mathscr{B}^{J}$ in $^{G}_{G}\mathcal{YD}^{\partial(J)}$ if and only if $I$ is a Majid ideal.}

\emph{Proof of the claim: }  ``$\Leftarrow$" Assume that $I$ is a Majid ideal, so $\textbf{H}^{J}/I=\mathscr{B}^{J}\otimes k\mathbbm{G}$ is a Majid algebra. This implies $\mathscr{B}^{J}$ is a Nichols algebra in $^{\mathbbm{G}}_{\mathbbm{G}}\mathcal{YD}^{\Phi}$. Simple computation shows that the induced Nichols algebra of this $\mathscr{B}^{J}$ is exactly $(\mathscr{B}^{J}, \rho_{L}, \blacktriangleright)$.

``$\Rightarrow$" By assumption, $\mathscr{B}^{J}\# k{G}$ is a Majid algebra and we have a canonical epimorphism  $$F:\ \textbf{H}^{J}=\mathscr{B}^{J}\otimes k{G}\to \mathscr{B}^{J}\# k\mathbbm{G},\quad x\otimes g\mapsto x\otimes \pi(g).$$
It is easy to see that $I\subseteq \Ker(F)$. By \eqref{eq5.6}, $I=\Ker(F)$, hence a Majid ideal.

Now get back to the proof of the lemma. Obviously $I$ is a coideal. Therefore, $I$ is a Majid ideal if and only if it is an ideal. Hence to prove the lemma, it amounts to show that $I$ is an ideal if and only if both $g_1^{\mathbbm{m}}$ and $g_2^{\mathbbm{n}}$ lie in the center of $\mathscr{B}^{J}\# kG$.

Clearly, if both $g_1^{\mathbbm{m}}$
and $g_2^{\mathbbm{n}}$ lie in the center then $I$ is an ideal. Conversely, assume that $I$ is an ideal.
Then $X_1\circ (g_1^{\mathbbm{m}}-1)\in I$, where $\circ$ denotes the multiplication
of $\textbf{H}^{J}$.  Since we always have
$$X_1\circ g_1^{\mathbbm{m}}=\alpha g_1^{\mathbbm{m}}\circ X_1$$
for some $0\neq \alpha\in k$. Therefore,
$$X_1\circ (g_1^{\mathbbm{m}}-1)=\alpha g_1^{\mathbbm{m}}\circ X_1 -X_1=\alpha (g_1^{\mathbbm{m}}-1)X_1+(\alpha -1)X_1.$$
So  $(\alpha -1)X_1\in I$, this forces that $\alpha=1$, that is, $g_1^{\mathbbm{m}}$
commutes with $X_1$. Similarly,  $g_1^{\mathbbm{m}}$ commutes with $X_2$. Thus
$g_1^{\mathbbm{m}}$ lies in the center. By the same method one can show that $g_2^{\mathbbm{n}}$
also belongs to the center.
\end{proof}

\begin{example}
\emph{ Let $T_{n}=T_{\mathbbm{n}^2}$ be the Taft algebra of dimension $\mathbbm{n}^4$. By definition,
$T_{n}$ is generated by two elements $x,g$ subject to
$$x^{\mathbbm{n}^2}=0,\;\;g^{\mathbbm{n}^2}=1,\;\;gx=\zeta_{n}xg.$$
Its comultiplication is given as
$$\D(x)=g\otimes x+x\otimes 1,\;\;\;\;\D(g)=g\otimes g.$$
Consider the right ideal $I$ generated by $g^{\mathbbm{n}}-1$. That is,
$$I=(g^{\mathbbm{n}}-1)T_{n}.$$
It is not hard to see that $I$ is a coideal but is NOT an ideal. Define  a twisting
$J$ in the following way}

$$J(g^{i_1}x^{j_1},g^{i_2}x^{j_2})=\left \{
\begin{array}{ll} 0 & \;\;\;\;\emph{if}\; j_1>0 \;\emph{or} \;j_2>0,\\
\zeta_{n}^{i_1(i_2-i_2')}&
\;\;\;\;\emph{if}\; j_1=j_2=0.
\end{array}\right. $$

\emph{Consider the twisted Majid algebra $T_{n}^{J}$ and we denote the new product by $\circ$. In this algebra, we have}
\begin{eqnarray*}
(g^{\mathbbm{n}}-1)\circ x&=& J((g^{\mathbbm{n}}-1)',x')(g^{\mathbbm{n}}-1)''x''J^{-1}((g^{\mathbbm{n}}-1)''',x''')\\
&=&J((g^{\mathbbm{n}}-1),g)x+J(g^{\mathbbm{n}},g)(g^{\mathbbm{n}}-1)x+J(g^{\mathbbm{n}},g)xJ^{-1}((g^{\mathbbm{n}}-1),1)\\
&=&(g^{\mathbbm{n}}-1)x;
\end{eqnarray*}
\begin{eqnarray*}
x\circ(g^{\mathbbm{n}}-1)&=& J(x',(g^{\mathbbm{n}}-1)')x''(g^{\mathbbm{n}}-1)''J^{-1}(x''',(g^{\mathbbm{n}}-1)''')\\
&=&J(g,(g^{\mathbbm{n}}-1))x+J(g,g^{\mathbbm{n}})x(g^{\mathbbm{n}}-1)+J(g,g^{\mathbbm{n}})xJ^{-1}(1,(g^{\mathbbm{n}}-1))\\
&=&(\zeta_{n}^{\mathbbm{n}}-1)x+\zeta_{n}^{\mathbbm{n}}x(g^{\mathbbm{n}}-1)x\\
&=&-x+\zeta_{n}^{\mathbbm{n}}xg^{\mathbbm{n}}\\
&=&(g^{\mathbbm{n}}-1)x.
\end{eqnarray*}

\emph{So we have $(g^{\mathbbm{n}}-1)\circ x=x\circ(g^{\mathbbm{n}}-1)$ in $T_{n}^{J}$. This implies the right ideal $I$ generated by
$g^{\mathbbm{n}}-1$ is a Majid ideal in $T_{n}^{J}$, hence the quotient
$T_{n}^{J}/I$ is a Majid algebra.
It is not hard to verify that
$$T_{n}^{J}/I\cong M(n,1,q),$$
where the latter was considered in Example \ref{ex5.4}.  Using the same method, the readers can realize all $M(n,s,q)$ constructed
in Example \ref{e2.7} as the quotients of twisted Taft algebras.}\qed
\end{example}

Now we give a criteria to determine when both $g_1^{\mathbbm{m}}$
and $g_2^{\mathbbm{n}}$ lie in the center of  $\mathscr{B}^{J}\# kG$.
At first, we set several parameters. Assume that
$$\sigma_{L}(X_{i})=h_{i}\otimes X_{i}$$
$$h_{i}\diamond X_j=q_{ij}X_{j}$$
for $1\leq i,j\leq 2$, $h_{i}\in G$ and $q_{ij}\in k^*$, where $\sigma_L$ (resp. $\diamond$) is the comodule (reap. module)
structure map of $\mathscr{B}(V)\in {^G_G\mathcal{YD}}$. So there are $\alpha_{i},\beta_{i}\in \mathbb{N}$ such that
$$h_{1}=g_1^{\alpha_{1}}g_2^{\alpha_{2}},\;\;h_{2}=g_1^{\beta_{1}}g_2^{\beta_{2}}.$$
Since we always assume that $h_{1},h_{2}$ generate the group $G$, there are $s_{i},t_{i}\in \mathbb{N}$ such that
$$g_{1}=h_1^{s_{1}}h_2^{s_{2}},\;\;g_{2}=h_1^{t_{1}}h_2^{t_{2}}.$$

With these preparations, we have
\begin{proposition}
Both $g_1^{\mathbbm{m}}$ and $g_2^{\mathbbm{n}}$ lie in the center of  $\mathscr{B}^{J}\# kG$ if and only if $a,b,d$ satisfy the equations
\begin{equation}\label{eq5.5}
\begin{split}
\zeta_{m}^{a\alpha_1 \mathbbm{m}}\zeta_{\mathbbm{m}\mathbbm{n}}^{b\alpha_2 \mathbbm{m}}&=q_{11}^{\mathbbm{m}s_1}q_{21}^{\mathbbm{m}s_2}\\
\zeta_{m}^{a\beta_1 \mathbbm{m}}\zeta_{\mathbbm{m}\mathbbm{n}}^{b\beta_2 \mathbbm{m}}&=q_{12}^{\mathbbm{m}s_1}q_{22}^{\mathbbm{m}s_2}\\
\zeta_{n}^{d\alpha_2 \mathbbm{n}}&=q_{11}^{\mathbbm{n}t_1}q_{21}^{\mathbbm{n}t_2}\\
\zeta_{n}^{d\beta_2 \mathbbm{n}}&=q_{12}^{\mathbbm{n}t_1}q_{22}^{\mathbbm{n}t_2}
\end{split}
\end{equation}
\end{proposition}

\begin{proof}
This is a consequence of direct computations. Indeed, $$g_1^{\mathbbm{m}}\circ X_1=
J(g_1^{\mathbbm{m}},g_1^{\alpha_{1}}g_2^{\alpha_{2}})g_1^{\mathbbm{m}}X_1=\zeta_{m}^{a\mathbbm{m}(\alpha_1- \alpha'_1)}g_1^{\mathbbm{m}}X_1
=q_{11}^{\mathbbm{m}s_1}q_{21}^{\mathbbm{m}s_2}X_1g_1^{\mathbbm{m}}$$
and
$$X_1\circ g_1^{\mathbbm{m}}=
J(g_1^{\alpha_{1}}g_2^{\alpha_{2}},g_1^{\mathbbm{m}})X_1g_1^{\mathbbm{m}}=\zeta_{m}^{a\alpha_1 \mathbbm{m}}\zeta_{\mathbbm{m}\mathbbm{n}}^{b\alpha_2 \mathbbm{m}}X_1g_1^{\mathbbm{m}}.$$
So $g_1^{\mathbbm{m}}\circ X_1=X_1\circ g_1^{\mathbbm{m}}$ if and only if $$\zeta_{m}^{a\alpha_1 \mathbbm{m}}\zeta_{\mathbbm{m}\mathbbm{n}}^{b\alpha_2 \mathbbm{m}}=q_{11}^{\mathbbm{m}s_1}q_{21}^{\mathbbm{m}s_2}.$$
Similarly, $g_1^{\mathbbm{m}}\circ X_2=X_2\circ g_1^{\mathbbm{m}}$ if and only if $$\zeta_{m}^{a\beta_1 \mathbbm{m}}\zeta_{\mathbbm{m}\mathbbm{n}}^{b\beta_2 \mathbbm{m}}=q_{12}^{\mathbbm{m}s_1}q_{22}^{\mathbbm{m}s_2};$$ and
$g_2^{\mathbbm{n}}\circ X_1=X_1\circ g_2^{\mathbbm{n}}$ if and only if $$\zeta_{n}^{d\alpha_2 \mathbbm{n}}=q_{11}^{\mathbbm{n}t_1}q_{21}^{\mathbbm{n}t_2};$$ and
$g_2^{\mathbbm{n}}\circ X_2=X_2\circ g_2^{\mathbbm{n}}$ if and only if $$\zeta_{n}^{d\beta_2 \mathbbm{n}}=q_{12}^{\mathbbm{n}t_1}q_{22}^{\mathbbm{n}t_2}.$$
\end{proof}

\begin{definition}
Equations \eqref{eq5.5} are called soluble if there exist integers $0\leq a,b<\m, \ 0\leq d<\n$ such that \eqref{eq5.5} hold.
\end{definition}

Now we can add a return road map to Figure I and complete the circuit as follows:

\begin{figure}[hbt]
\begin{picture}(180,115)(0,0)
\put(0,90){\makebox(0,0){$ \MM=\R\#k\mathbbm{G}\;\;\;$ \textsf{Original Majid algebra}}}     \put(150,90){\makebox(0,0){$ \MM=\mathscr{B}^{J^{-1}}\#k\mathbbm{G}$ }}
\put(-50,85){\vector(0,-1){30}}                                                              \put(150,55){\vector(0,1){30}} \put(200,70){\makebox(0,0){\small{Eq. \eqref{eq5.5} soluble}}}
\put(0,50){\makebox(0,0){$ {\textbf{M}}=\R\#kG\;\;\;$ \textsf{Bigger Majid algebra}}}           \put(150,50){\makebox(0,0){$ \textbf{H}^{J^{-1}}=\mathscr{B}^{J^{-1}}\#kG$ }}
\put(-50,45){\vector(0,-1){30}}\put(-40,30){\makebox(0,0){\small\textsf{ $\exists\, J$}}}  \put(150,15){\vector(0,1){30}}
\put(0,10){\makebox(0,0){$ \textbf{H}=(\R\#kG)^{J}\;\;\;$ \textsf{Ordinary Hopf algebra}}}      \put(150,10){\makebox(0,0){$ \textbf{H}=\mathscr{B}\#kG$ }}
\put(90,-10){\makebox(0,0){\textrm{Figure II}}}
\end{picture}\\[4mm]
\end{figure}

\subsection{Solutions of equations \eqref{eq5.5}}
In this subsection, we will show that equations \eqref{eq5.5} have at most one solution and give a criterion to determine when equations \eqref{eq5.5} do have a solution.

Keep the notations of Subsection 5.3. Call $\mathscr{B}(V)$ a \emph{connected} rank $2$ Nichols algebra in $^{{G}}_{{G}}\mathcal{YD}$ if the ordinary pointed Hopf algebra $\mathscr{B}(V)\# kG$ is connected. As before, let $V=kX_1\oplus kX_2$ be the direct sum of $1$-dimensional Yetter-Drinfeld modules.  Therefore, there exist $ h_i\in G$ and $q_{ij}\in k^{\ast}$ such that
$$\sigma_L(X_i)=h_i \otimes X_i,\ \  h_{i}\diamond X_j=q_{ij}X_j$$
for $1\leq i,j\leq 2$, where $\sigma_{L}$ and $\diamond$ are structure maps of $V\in{^{{G}}_{{G}}\mathcal{YD}}$.

Assume that $h_1=g_1^{\alpha_1}g_2^{\alpha_2}, \ h_2=g_1^{\beta_1}g_2^{\beta_2}$ for some
$\alpha_{i},\beta_{i}\in \N$. As $h_1$ and $h_2$ generate $G$ (due to the connectedness of  $\mathscr{B}(V)\# kG$), there exist $0\leq s_1,t_1< m, 0\leq s_2, t_2 < n$ such that $g_1=h_1^{s_1}h_2^{s_2},\ g_2=h_1^{t_1}h_2^{t_2}.$  Therefore, we have

\begin{equation} \label{eq6.4} \left(\begin{array}{cc}\alpha_1 & \beta_1 \\ \alpha_2 & \beta_2  \end{array}\right) \left(\begin{array}{cc}s_1 & t_1 \\ s_2 & t_2  \end{array}\right)\equiv \left(\begin{array}{cc}1\ (\text{mod} \ m) & 0 \ (\text{mod}\ m) \\ 0\ (\text{mod}\ n)& 1\ (\text{mod}\ n)  \end{array}\right)\end{equation}

Note that \[g_1\diamond X_1=\zeta_{m}^{x_{11}}X_1,\quad g_1\diamond X_2=\zeta_{m}^{x_{12}}X_2,\quad g_2\diamond X_1=\zeta_{n}^{x_{21}}X_1,\quad g_2\diamond X_2=\zeta_{n}^{x_{22}}X_2\] for some $0\leq x_{11},x_{12}<{m}, \ 0\leq x_{21},x_{22}<n.$ Hence we have
\begin{equation}\label{eq6.2}
\begin{split}
q_{11}&=\zeta_{m}^{x_{11}\alpha_1}\zeta_{n}^{x_{21}\alpha_2},\\
q_{12}&=\zeta_{m}^{x_{12}\alpha_1}\zeta_{n}^{x_{22}\alpha_2},\\
q_{21}&=\zeta_{m}^{x_{11}\beta_1}\zeta_{n}^{x_{21}\beta_2},\\
q_{22}&=\zeta_{m}^{x_{12}\beta_1}\zeta_{n}^{x_{22}\beta_2}.
\end{split}
\end{equation}

Now equations \eqref{eq5.5} read as

\begin{equation}\label{eq6.6}
\begin{split}
\zeta_{m}^{a\alpha_1 \mathbbm{m}}\zeta_{\mathbbm{m}\mathbbm{n}}^{b\alpha_2 \mathbbm{m}}&=\zeta_{m}^{\mathbbm{m}x_{11}(\alpha_{1}s_1+\beta_{1}s_2)}
\zeta_{n}^{\mathbbm{m}x_{21}(\alpha_{2}s_1+\beta_{2}s_2)}\\
\zeta_{m}^{a\beta_1 \mathbbm{m}}\zeta_{\mathbbm{m}\mathbbm{n}}^{b\beta_2 \mathbbm{m}}&=\zeta_{m}^{\mathbbm{m}x_{12}(\alpha_{1}s_1+\beta_{1}s_2)}
\zeta_{n}^{\mathbbm{m}x_{22}(\alpha_{2}s_1+\beta_{2}s_2)}\\
\zeta_{n}^{d\alpha_2 \mathbbm{n}}&=\zeta_{m}^{\mathbbm{n}x_{11}(\alpha_{1}t_1+\beta_{1}t_2)}
\zeta_{n}^{\mathbbm{n}x_{21}(\alpha_{2}t_1+\beta_{2}t_2)}\\
\zeta_{n}^{d\beta_2 \mathbbm{n}}&=\zeta_{m}^{\mathbbm{n}x_{12}(\alpha_{1}t_1+\beta_{1}t_2)}
\zeta_{n}^{\mathbbm{n}x_{22}(\alpha_{2}t_1+\beta_{2}t_2)}
\end{split}
\end{equation}

By \eqref{eq6.4}, equations \eqref{eq6.6} can be simplified as
\begin{equation}\label{eq6.7}
\begin{split}
\zeta_{m}^{a\alpha_1 \mathbbm{m}}\zeta_{\mathbbm{m}\mathbbm{n}}^{b\alpha_2 \mathbbm{m}}&=\zeta_{m}^{\mathbbm{m}x_{11}}\\
\zeta_{m}^{a\beta_1 \mathbbm{m}}\zeta_{\mathbbm{m}\mathbbm{n}}^{b\beta_2 \mathbbm{m}}&=\zeta_{m}^{\mathbbm{m}x_{12}}\\
\zeta_{n}^{d\alpha_2 \mathbbm{n}}&=\zeta_{n}^{\mathbbm{n}x_{21}}\\
\zeta_{n}^{d\beta_2 \mathbbm{n}}&=\zeta_{n}^{\mathbbm{n}x_{22}}
\end{split}
\end{equation}
which clearly are equivalent to following congruence equations
\begin{equation}\label{eq6.8}
\begin{split}
a\alpha_1\frac{\mathbbm{n}}{\mathbbm{m}}+b\alpha_2 &\equiv x_{11}\frac{\mathbbm{n}}{\mathbbm{m}}\ (\text{mod} \ \mathbbm{n})\\
a\beta_1\frac{\mathbbm{n}}{\mathbbm{m}}+b\beta_2&\equiv x_{12}\frac{\mathbbm{n}}{\mathbbm{m}}\ (\text{mod} \ \mathbbm{n})\\
d\alpha_2&\equiv x_{21} \ (\text{mod} \ \mathbbm{n})\\
d\beta_2&\equiv x_{22} \ (\text{mod} \ \mathbbm{n})
\end{split}
\end{equation}

Using \eqref{eq6.4} again, by multiplying the first (resp. second) congruence equation by $s_1$ (resp. $s_2$) and take the
sum of them, we have
\begin{equation} a=(x_{11}s_1+x_{12}s_2)'.
\end{equation}
Similarly, we have
$$b=(\frac{\mathbbm{n}}{\mathbbm{m}}(x_{11}t_1+x_{12}t_2))',\;\; d=(t_1x_{21}+t_2x_{22})'',$$
here as before, $x'$ is the remainder of $x$ divided by $\mathbbm{m}$ (resp. $x''$ is the remainder of $x$ divided by $\mathbbm{n}$).

In summary, we have

\begin{proposition} \label{p5.12} Equations \eqref{eq5.5} have at most one solution in the following form:
\begin{equation}\label{eq5.14} a=(x_{11}s_1+x_{12}s_2)',\;\;b=(\frac{\mathbbm{n}}{\mathbbm{m}}(x_{11}t_1+x_{12}t_2))',\;\; d=(t_1x_{21}+t_2x_{22})''.
\end{equation}
\end{proposition}

\begin{remark}
It is possible that equations \eqref{eq5.5} have no solution. In fact, take $\alpha_1=1,\ \alpha_2 =1,\ \beta_1=0,\ \beta_2=1,\ x_{21}=0,\ x_{22}=1$ for example, then the last two equations in \eqref{eq6.8} clearly have no solution.\end{remark}

The last problem is to determine when equations \eqref{eq5.5} do have a solution. This is settled as follows.
\begin{proposition} \label{p5.14} Keep the above notations. Equations \eqref{eq5.5} have a solution if and only if
\begin{equation}\label{eq5.15} x_{22}\alpha_{2}-x_{21}\beta_{2}\equiv 0\ (\emph{mod}\ \mathbbm{n}).
\end{equation}
\end{proposition}
\begin{proof} ``$\Leftarrow$" It is enough to show that \eqref{eq5.14} is exactly a solution of \eqref{eq6.8}.  Consider the last two equations in \eqref{eq6.8}, we have
\begin{eqnarray*} d\alpha_2-x_{21} &\equiv& (t_1x_{21}+t_2x_{22}) \alpha_2-x_{21}\ (\text{mod}\ \mathbbm{n})\\
&\equiv& (1-\beta_2t_2)x_{21}+t_2x_{22}\alpha_2-x_{21}\ (\text{mod}\ \mathbbm{n})\\
&\equiv& t_2(x_{22}\alpha_{2}-x_{21}\beta_{2})\ (\text{mod}\ \mathbbm{n})\\
&\equiv& 0 \quad (\text{mod}\ \mathbbm{n}),
\end{eqnarray*}
and
\begin{eqnarray*} d\beta_2-x_{22} &\equiv& (t_1x_{21}+t_2x_{22}) \beta_2-x_{22}\ (\text{mod}\ \mathbbm{n})\\
&\equiv& t_1x_{21}\beta_2+(1-\alpha_2t_1)x_{22}-x_{22}\ (\text{mod}\ \mathbbm{n})\\
&\equiv& t_1(-x_{22}\alpha_{2}+x_{21}\beta_{2})\ (\text{mod}\ \mathbbm{n})\\
&\equiv& 0 \quad (\text{mod}\ \mathbbm{n}).
\end{eqnarray*}
Therefore, the last two equations of \eqref{eq6.8} are satisfied. Next we consider the first two equations of \eqref{eq6.8}.
To show this, we need  the following claim.\\[2mm]
\emph{Claim: We always have }
\begin{equation}\label{eq5.16} \left(\begin{array}{cc}s_1 & t_1 \\ s_2 & t_2  \end{array}\right) \left(\begin{array}{cc}\alpha_1 & \beta_1 \\ \alpha_2 & \beta_2  \end{array}\right) \equiv \left(\begin{array}{cc}1\ (\text{mod} \ m) & 0 \ (\text{mod}\ m) \\ 0\ (\text{mod}\ m)& 1\ (\text{mod}\ m)  \end{array}\right).\end{equation}
\emph{Proof of this claim: } Note that the reason for equation \eqref{eq6.4} fulfilled is that if $g_1=g_{1}^{s}g_{2}^{t},\
g_2=g_{1}^{x}g_{2}^{y}$ then $s\equiv 1\ (\text{mod}\ m), t\equiv 0\ (\text{mod}\ n), x\equiv 0\ (\text{mod}\ m), y\equiv 1\ (\text{mod}\ n)$. So, to show \eqref{eq5.16} it is enough to show that
$$h_1=h_{1}^{s}h_{2}^{t},\
h_2=h_{1}^{x}h_{2}^{y}$$
will imply that $s\equiv 1\ (\text{mod}\ m), t\equiv 0\ (\text{mod}\ m), x\equiv 0\ (\text{mod}\ m), y\equiv 1\ (\text{mod}\ m)$.
We only consider the case $h_1=h_{1}^{s}h_{2}^{t}$ since the other one is similar. Clearly, it is equivalent to showing that if
$h_{1}^{s}h_{2}^{t}=1$ then $m|s$ and $m|t$. We use Lemma \ref{l3} to prove this. Here we use notations of Lemma \ref{l3} freely. In order to not cause confusion, we point out that our $h_1$ (resp. $h_2$) is just the element  $g$ (resp. $h$) of Lemma
\ref{l3}. Therefore by Lemma \ref{l3}, $g^{s}h^{t}=1$ implies that
$$m_1|t,\ \ m_2| (s+tb),\ \ n_1|s,\ \ n_2|(sa+t).$$
From $m_1|t,\ m_2| (s+tb)$ and $m_1|m_2$, it follows that $m_1|s$. Using $n_1|s$ and $(m_1,n_1)=1$, we have $m=m_1n_1|s$. Similarly, one can show that $m_1n_1|t$.

We return to the proof of this proposition. Now, we have
\begin{eqnarray*} &&a\alpha_1\frac{\mathbbm{n}}{\mathbbm{m}}+b\alpha_2-x_{11}\frac{\mathbbm{n}}{\mathbbm{m}}\\
&\equiv& (x_{11}s_1+x_{12}s_2)\alpha_1\frac{\mathbbm{n}}{\mathbbm{m}}+(x_{11}t_1+x_{12}t_2)\alpha_2\frac{\mathbbm{n}}{\mathbbm{m}}-x_{11}\frac{\mathbbm{n}}{\mathbbm{m}}  \ (\text{mod}\ \mathbbm{n})\\
&\equiv& \frac{\mathbbm{n}}{\mathbbm{m}}[x_{11}(s_1\alpha_1+t_1\alpha_2)+x_{12}(s_2\alpha_1+t_2\alpha_2)-x_{11}]\ (\text{mod}\ \mathbbm{n})\\
&\equiv&\frac{\mathbbm{n}}{\mathbbm{m}}[x_{11}-x_{11}]\ (\text{mod}\ \mathbbm{n})\\
&\equiv& 0 \quad (\text{mod}\ \mathbbm{n}),
\end{eqnarray*}
where in the third ``$\equiv$" we use equation \eqref{eq5.16}.  The second equation of \eqref{eq6.8} can be verified similarly:
\begin{eqnarray*} &&a\beta_1\frac{\mathbbm{n}}{\mathbbm{m}}+b\beta_2-x_{12}\frac{\mathbbm{n}}{\mathbbm{m}}\\
&\equiv& (x_{11}s_1+x_{12}s_2)\beta_1\frac{\mathbbm{n}}{\mathbbm{m}}+(x_{11}t_1+x_{12}t_2)\beta_2\frac{\mathbbm{n}}{\mathbbm{m}}-x_{12}\frac{\mathbbm{n}}{\mathbbm{m}}  \ (\text{mod}\ \mathbbm{n})\\
&\equiv& \frac{\mathbbm{n}}{\mathbbm{m}}[x_{11}(s_1\beta_1+t_1\beta_2)+x_{12}(s_2\beta_1+t_2\beta_2)-x_{12}]\ (\text{mod}\ \mathbbm{n})\\
&\equiv&\frac{\mathbbm{n}}{\mathbbm{m}}[x_{12}-x_{12}]\ (\text{mod}\ \mathbbm{n})\\
&\equiv& 0 \quad (\text{mod}\ \mathbbm{n}).
\end{eqnarray*}

``$\Rightarrow$"  By the proof of the sufficiency, we know that
$$t_1(-x_{22}\alpha_{2}+x_{21}\beta_{2})\equiv 0 \equiv t_2(x_{22}\alpha_{2}-x_{21}\beta_{2}) \quad (\text{mod}\ \mathbbm{n}).$$
Therefore,
\begin{eqnarray*} x_{22}\alpha_{2}-x_{21}\beta_{2} &\equiv& (\alpha_2t_1+\beta_2t_2)(  x_{22}\alpha_{2}-x_{21}\beta_{2})\ (\text{mod}\ \mathbbm{n})\\
&\equiv& \alpha_2 t_1(-x_{22}\alpha_{2}+x_{21}\beta_{2})+\beta_2 t_2(x_{22}\alpha_{2}-x_{21}\beta_{2}) \ (\text{mod}\ \mathbbm{n})\\
&\equiv& 0 \quad (\text{mod}\ \mathbbm{n}).
\end{eqnarray*}
\end{proof}

\subsection{Main result}
Keep the notations of Subsections 5.3 and 5.4.

\begin{definition}
A rank $2$ connected Nichols algebra $\mathscr{B}(V)\in {^{G}_{G}\mathcal{YD}}$ is said to be of \emph{Majid type} if equation \eqref{eq5.15} is fulfilled.
\end{definition}

Now let $\mathscr{B}(V) \in {^{G}_{G}\mathcal{YD}}$ be a Nichols algebra of Majid type and we may choose $3$ natural numbers $a,b,d$ in the form of \eqref{eq5.14}. By Figure II, we get a Majid algebra
$$\MM(V,G):=\mathscr{B}(V)^{J_{a,b,d}}\# k{G}/(\mathbbm{g}_1^{\m}-1,\mathbbm{g}_2^{\n}-1)=\mathscr{B}(V)^{J_{a,b,d}}\# k\mathbb{G}.$$

 \begin{definition} The Majid algebra $\MM(V,G)$ is called the \emph{associated Majid algebra} of the Nichols algebra $\mathscr{B}(V)\in {^{G}_{G}\mathcal{YD}}$ of Majid type. In addition, if one of $a,b,d$ is not zero, then we call $\MM(V,G)$  the \emph{associated genuine Majid algebra} of  $\mathscr{B}(V)$.
 \end{definition}

 Now, our main result can be stated as follows.

 \begin{theorem}\label{t5.17}
 Keep all the previous assumptions and notations.
 \begin{itemize}
 \item[(1)] $\MM(V,G)$ is a finite-dimensional connected coradically graded pointed Majid algebra of rank $2$;
 \item[(2)] Let $\MM$ be a finite-dimensional connected coradically graded pointed Majid algebra of rank $2$. Then $\MM$ is isomorphic to some $\MM(V,G)$ associated to an appropriate Nichols algebra $\mathscr{B}(V)\in {^{G}_{G}\mathcal{YD}}$ of Majid type. Moreover, $\MM$ is genuine if and only if $\MM(V,G)$ is genuine.
\end{itemize}
 \end{theorem}
 \begin{proof}
 The statement (1) is clear since $\mathscr{B}(V)$ is finite-dimensional. The statement (2) is a direct consequence of Figure II and Proposition \ref{p5.14}.
 \end{proof}

\section{Presentation of $\MM(V,G)$}
The main aim of this section is to provide an explicit presentation of $\MM(V,G)$, or equivalently, of  $\mathscr{B}(V)^{J_{a,b,d}}$. In principle, the idea is simple: first we recall the classification results about rank $2$ Nichols algebras given by Heckenberger \cite{h,h2}, then we carry out the twisting process.

\subsection{Generalized Dynkin diagrams and full binary trees}
Let $E=\{e_i|1\leq i\leq d\}$ be a canonical basis of $\Z^d,$ and $\chi$ be a bicharacter of $\Z^d.$ The numbers $q_{ij}=\chi(e_i,e_j)$ are called the structure constants of $\chi$ with respect to $E.$
\begin{definition}
The generalized Dynkin diagram of the pair $(\chi, E)$ is a nondirected graph $\mathcal{D}_{\chi,E}$ with the following properties:
\begin{itemize}
\item[(1)]There is a bijective map $\phi$ from $I=\{ 1, 2, \dots, d \}$ to the set of vertices of $\mathcal{D}_{\chi,E}.$
\item[(2)]For all $1\leq i\leq d,$ the vertex $\phi(i)$ is labelled  by $q_{ii}.$
\item[(3)]For all $1\leq i,j\leq d,$ the number  $n_{ij}$ of edges between $\phi(i)$ and $\phi(j)$ is either $0$ or $1.$ If $i=j$ or $q_{ij}q_{ji}=1$ then $n_{ij}=0,$ otherwise $n_{ij}=1$ and the edge is labelled by $q_{ij}q_{ji}.$
\end{itemize}
\end{definition}
Obviously a generalized Dynkin diagram is completely determined by the matrix $(q_{ij})_{d \times d},$ which also arises naturally in the context of Yetter-Drinfeld modules and Nichols algebras. Recall that for a diagonal Yetter-Drinfeld module $V \in {_G^G \mathcal{YD}^\Phi},$ there exists a basis $\{X_i\}$ of $V$ such that
\begin{equation}\label{eq6.1}
\delta_L(X_i)=g_i \otimes X_i,\ \  g_{i}\triangleright X_j=q_{ij}X_j
\end{equation}
for some $g_i\in G$ and $q_{ij}\in k^{\ast}$. Hence we obtain a matrix $(q_{ij})$ and the corresponding generalized Dynkin diagram $\mathcal{D}_{(q_{ij})}.$ We also call $\mathcal{D}_{(q_{ij})}$ the generalized Dynkin diagram of $V$ or $\mathscr{B}(V).$

Suppose $V \in {_G^G \mathcal{YD}}$ and let $J$ be a 2-cochain on $G.$ Recall that $\mathscr{B}(V)^J$ is a Nichols algebra in $_G^G \mathcal{YD}^{\partial(J)}$ by Lemma \ref{tn}. It is natural to observe that

\begin{lemma}
 $\mathscr{B}(V)$ and $\mathscr{B}(V)^J$ share the same generalized Dynkin diagram.
\end{lemma}

\begin{proof} By easy computation using \eqref{ta}. \end{proof}

To describe the structure of rank $2$ Nichols algebras of diagonal type, we need the notion of full binary trees (see \cite{h} and Appendix A).  A binary tree $T$ is a tree such that each node has at most two children. $T$ is called full if each node of $T$ has exactly zero or two children. Let $r(T)$ denote the root of the binary tree $T$. We use $N_0(T)$ and $N_2(T)$ to present the set of nodes of $T$ which have zero and two children respectively, and $N(T)=N_0(T)\cup N_2(T).$ Let $\{L,R\}$ be a set of two abstract elements and define $\overline{N}_2(T)=N_2(T)\cup\{L,R\},$ $\overline{N}(T)=N(T)\cup\{L,R\}.$ We use $a_L$ and $a_R$ to denote the left and right child of $a\in N_2(T).$ For $b\in N(T) \setminus \{r(T)\}$, let $b^{P}$ denote its parent.
\begin{definition}
Let $T$ be a binary tree and $a\in N(T).$ A node $b\in \overline{N}_2(T)$ is called the left godfather of $a,$ denoted by $b=a^L,$ if one of the following conditions holds:
\begin{itemize}
\item[(1)] $a=r(T)$ and $b=L;$
\item[(2)] $a$ is the right children of $b;$
\item[(3)] $a$ is the left child of $c$ and $b$ is the left godfather of $c.$
\end{itemize}
\end{definition}
In a similar manner one defines the right godfather $a^R$ of $a.$ Moreover, we can define a function $\mathfrak{l}:N(T)\to \N$ inductively by

$$\mathfrak{l}(a)=\left \{
\begin{array}{ll} \mathfrak{l}(a^R)+1 & \mathrm{if} \ \ {a^R}_L=a, \\
1 & \text{otherwise}.
\end{array}\right. $$

\subsection{Classification of rank $2$ Nichols algebras}
Let $V$ be a diagonal Yetter-Drinfeld module of dimension $2$ over an abelian group $G$. Let $\{X_1,X_2\}$ be a basis
of $V$ such that \eqref{eq6.1} holds for certain $g_i\in G$ and $q_{ij}\in k^{\ast}$ for  $1\leq i,j\leq 2.$ Choose a basis $E=\{e_1,e_2\}$ of $\Z^2$. Then there exist a unique group homomorphism $\varphi \colon \Z^2\to G$ and a unique bicharacter $\chi \colon \Z^2\times \Z^2 \to k^{\ast}$ satisfying
\begin{equation}
\varphi(e_i)=g_i,\quad \chi(e_i,e_j)=q_{ij}  \;\;(1\leq i,j\leq 2).
\end{equation}
Clearly, the generalized Dynkin diagram of $V$ is the same as $\mathcal{D}_{\chi, E}$.

Let $T(V)$ be the tensor algebra of $V$, then $T(V)$ admits a $\Z^{2}$-grading by extending
\begin{equation} \deg X_i=e_i, \;\;\;\;1\leq i\leq 2.
\end{equation}
For brevity, here and below we write $|X|=\deg(X)$ for the $\Z^{2}$-homogenous
$X \in T(V)$.

Define a map $\tau \colon \overline{N}(T)\to T(V)$ as follows:
\begin{itemize}
\item[(1)] $\tau(R)=X_1,\tau(L)=X_2;$
\item[(2)] if $a\in N(T),$ then $\tau(a)=\tau(a^R)\tau(a^L)-\chi(|\tau(a^R)|,|\tau(a^L)|)\tau(a^L)\tau(a^R).$
\end{itemize}

For $a\in \overline{N}(T),$ let $$P_a=\chi^{-1}(|\tau(a)|,|\tau(a)|).$$ From now on, for convenience we will write $\chi(a,b)$ instead of $\chi(|\tau(a)|,|\tau(a)|)$ for all $a,b\in \overline{N}(T).$ For $a\in \overline{N}(T),$ define
$$\lambda(a)=\left \{
\begin{array}{ll}0& \mathrm{if}\ \  a \notin N(T),\\
\chi^{-1}(L,R)-\chi(R,L)& \mathrm{if}\ \  a=r(T),\\
\chi^{-1}(a^L,a^R)-\chi(a^R,a^L)+\lambda(a^{P})& \text{otherwise}.
\end{array}\right. $$
Furthermore, for any $b\in N(T)$ with $b^L\in N(T),$ set
$$\mu(b)=\left \{
\begin{array}{ll} \lambda(b)& \mathrm{if} \ \ b={b^L}_R, \\
\lambda(b)\mu(b^R) & \text{otherwise}.
\end{array}\right. $$
In the following, by $R_n$ we denote the set of primitive $n$-th roots of unity and set $\mathscr{R}=\bigcup_{n\geq 2} R_{n}.$
With these notations, we can state Heckenberger's classification result as follows.

\begin{theorem} \emph{(\cite[Theorem 7.1]{h} and \cite[Theorem 7]{h2})}
The set of finite-dimensional rank $2$ Nichols algebras $\mathscr{B}(V)$ are in one-to-one correspondence with the set of pairs $(\mathcal{D},T)$ appeared in the same horizontal line of Table 6.1, where $\mathcal{D}$ is a generalized Dynkin diagram with all parameters
taken in $\mathscr{R}$  and $T$ is a full binary tree. Moreover, the Nichols algebra $\mathscr{B}(V)$ associated with $(\mathcal{D},T)$ is the quotient of the tensor algebra $T(V)$ by the ideal generated by the following set
\begin{equation} \label{6.3}
\begin{split}
&\{\tau(a) \mid a\in N_0(T)\}\cup \{\tau(a)^{\operatorname{ord}P_a} \mid a\in \overline{N}_2(T), 2\leq \operatorname{ord} P_a\leq \infty \}\cup \\
&\{\tau(b)\tau(c^L)-\chi(b,c^L)\tau(c^L)\tau(b)-\frac{\mu(b)}{(\mathfrak{l}(b)+1)!_{P_c}} \tau(c)^{\mathfrak{l}(b)+1} \mid b\in N_2(T),c=b^L\in N_2(T)\}.
\end{split}
\end{equation}

{\setlength{\unitlength}{1mm}
\begin{table}[t]\centering
\caption{Finite-dimensional Nichols algebras of rank $2$.}
\vspace{1mm}
\begin{tabular}{r|l|l|l}
\hline
& \text{Dynkin diagrams} & \text{Fixed parameters}&\text{Binary trees} \\
\hline
\hline
1 & \rule[-3\unitlength]{0pt}{8\unitlength}
\begin{picture}
(14,5)(0,3) \put(1,2){\circle{2}} \put(13,2){\circle{2}} \put(1,5){\makebox[0pt]{\scriptsize~$q$}}
\put(13,5){\makebox[0pt]{\scriptsize~$r$}}
\end{picture}
& $q,r\in \Bbbk^\ast $& $T_1$
\\
\hline
2 & \Dchaintwo{}{$q$}{$q^{-1}$}{$q$} & $q\in \Bbbk^\ast \setminus \{1\}$& $T_1$
\\
\hline
3 & \Dchaintwo{}{$q$}{$q^{-1}$}{$-1$}\  \Dchaintwo{}{$-1$}{$q$}{$-1$}& $q\in \Bbbk^\ast \setminus \{-1,1\}$& $T_2,T_2$
\\
\hline
4 & \Dchaintwo{}{$q$}{$q^{-2}$}{$q^2$} & $q\in \Bbbk^\ast \setminus \{-1,1\}$& $T_3$
\\
\hline
\rule[-3\unitlength]{0pt}{10\unitlength} 5 & \Dchaintwo{}{$q$}{$q^{-2}$}{$-1$}&
$q\in \Bbbk^\ast \setminus \{-1,1\}$, $q\notin R_4$& $T_3$
\\
\hline
\rule[-3\unitlength]{0pt}{10\unitlength} 6 &
\Dchaintwo{}{$\zeta $}{$q^{-1}$}{$q$} &
$\zeta \in R_3,$ $q\in \Bbbk^\ast \setminus \{1,\zeta,\zeta^2\}$
& $T_3$
\\
\hline
7 & \Dchaintwo{}{$\zeta $}{$-\zeta $}{$-1$} & $\zeta \in R_3$& $T_3$
\\
\hline
8 &\Dchaintwo{}{$-\zeta^{-2}$}{$-\zeta^3$}{$-\zeta^2$}\ \Dchaintwo{}{$-\zeta^{-2}$}{$\zeta^{-1}$}{$-1$}\
\Dchaintwo{}{$-\zeta^3$}{$\zeta$}{$-1$} & $\zeta \in R_{12}$& $T_4,T_5,T_7$
\\
\hline
9 &\Dchaintwo{}{$-\zeta^2$}{$\zeta$}{$-\zeta^2$}\ \Dchaintwo{}{$-\zeta^2$}{$\zeta^3$}{$-1$}\
\Dchaintwo{}{$-\zeta^{-1}$}{$-\zeta^3 $}{$-1$} & $\zeta \in R_{12}$& $T_4,T_5,T_7$
\\
\hline
10 &\Dchaintwo{}{$-\zeta$}{$\zeta^{-2}$}{$\zeta^3$}\ \Dchaintwo{}{$\zeta^3$}{$\zeta^{-1}$}{$-1$}\
 \Dchaintwo{}{$-\zeta^2$}{$\zeta $}{$-1$}& $\zeta \in R_9$& $T_6,T_9,T_{14}$
\\
\hline
\rule[-3\unitlength]{0pt}{10\unitlength} 11 & \Dchaintwo{}{$q$}{$q^{-3}$}{$q^3$} & $q\in \Bbbk^\ast \setminus \{-1,1\}$,
$q\notin R_3$& $T_8$
\\
\hline
12 &\Dchaintwo{}{$\zeta^2$}{$\zeta $}{$\zeta^{-1}$}\ \Dchaintwo{}{$\zeta^2$}{$-\zeta^{-1}$}{$-1$}\
\Dchaintwo{}{$\zeta $}{$-\zeta $}{$-1$}& $\zeta \in R_8$& $T_8,T_8,T_8$
\\
\hline
13 & \Dchaintwo{}{$\zeta^6 $}{$-\zeta^{-1}$}{$\zeta^{-4}$}\ \Dchaintwo{}{$\zeta^6$}{$\zeta$}{$\zeta^{-1}$}\
\Dchaintwo{}{$-\zeta^{-4}$}{$\zeta^5$}{$-1$}\ \Dchaintwo{}{$\zeta$}{$\zeta^{-5} $}{$-1$} & $\zeta \in R_{24}$& $T_{10},T_{13},T_{17},T_{21}$
\\
\hline
14 & \Dchaintwo{}{$\zeta $}{$\zeta^2$}{$-1$}\ \Dchaintwo{}{$-\zeta^{-2}$}{$\zeta^{-2}$}{$-1$} & $\zeta \in R_5$& $T_{11},T_{16}$
\\
\hline
15 & \Dchaintwo{}{$\zeta $}{$\zeta^{-3}$}{$-1$}\ \Dchaintwo{}{$-\zeta^{-2}$}{$\zeta^3$}{$-1$} & $\zeta \in R_{20}$& $T_{11},T_{16}$
\\
\hline
16 & \Dchaintwo{}{$-\zeta $}{$-\zeta^{-3}$}{$\zeta^5$}\ \Dchaintwo{}{$\zeta^3$}{$-\zeta^4$}{$-\zeta^{-4}$}\
\Dchaintwo{}{$\zeta^5$}{$-\zeta^{-2}$}{$-1$}\ \Dchaintwo{}{$\zeta^3$}{$-\zeta^2$}{$-1$} & $\zeta \in R_{15}$& $T_{12},T_{15},T_{18},T_{20}$
\\
\hline
17 & \Dchaintwo{}{$-\zeta- $}{$-\zeta^{-3}$}{$-1$}\ \Dchaintwo{}{$-\zeta^{-2}$}{$-\zeta^3$}{$-1$} &$\zeta\in R_7 $&$T_{19},T_{22}$
\\
\hline
\end{tabular}
\end{table}}
\end{theorem}

\subsection{The presentation of $\mathscr{B}(V)^J$}
In this subsection, the notations of Section 5 are used freely. Especially, we take $$\mathbbm{G}=\Z_\mathbbm{m}\times \Z_\mathbbm{n}=\langle \mathbbm{g}_1\rangle\times \langle \mathbbm{g}_2\rangle, \quad G=\Z_{m}\times \Z_{n}=\langle g_1\rangle\times \langle g_2\rangle $$ with $m=\mathbbm{m}^2, n=\mathbbm{n}^2$. Moreover, let  $V=kX_1\oplus kX_2$ be the direct sum of $1$-dimensional Yetter-Drinfeld modules of $(kG, \pi^{\ast}(\Phi))$.  Therefore, there exist $ h_i\in G$ and $q_{ij}\in k^{\ast}$ such that
$$\delta_L(X_i)=h_i \otimes X_i,\ \  h_{i}\triangleright X_j=q_{ij}X_j$$
for all $1\leq i,j\leq 2$.

Consider the tensor algebra $T_{\pi^{\ast}(\Phi)}(V)$ in $_G^G \mathcal{YD}^{\pi^*(\Phi)}$. The $G$-comodule structure of $V$ induces a $G$-grading on $T_{\pi^{\ast}(\Phi)}(V)$. For any homogenous element $X \in T_{\pi^{\ast}(\Phi)}(V)$, let $\delta(X)$ denote its $G$-degree, that is $\delta_L(X) = \delta(X) \otimes X$. Note that $\delta(X)$ is also denoted simply by $x$ in Sections 4 and 5. We introduce this new notation mainly for the awkward situation when there is no sense of lowercase, for example, a long expression.

According to Figure II of Section 5, any finite-dimensional rank $2$ Nichols algebras in $_G^G \mathcal{YD}^{\pi^*(\Phi)}$ can be realized as the twist $\mathscr{B}(V)^{J}$ of a rank $2$ Nichols algebra $\mathscr{B}(V)$ in $_G^G \mathcal{YD}.$ Recall that, the product of $\mathscr{B}(V)^{J}$ is denoted by $\circ$ and by definition
\begin{equation*}
X\circ Y =J(x,y)XY
\end{equation*} for all homogeneous elements $X,Y\in \mathscr{B}(V).$

Define a map $\tau^*:\overline{N}(T)\to T_{\pi^{\ast}(\Phi)}(V)$, a twisted version of $\tau$, as follows:
\begin{itemize}
\item[(1)]$\tau^*(R)=X_1, \tau^*(L)=X_2;$
\item[(2)]for $a\in N(T),$ define $$\tau^*(a)=\frac{\tau^*(a^R)\circ\tau^*(a^L)}{J_{a,b,d}(a^R,a^L)}
    -\chi(a^R,a^L)\frac{\tau^*(a^L) \circ \tau^*(a^R)}{J_{a,b,d}(a^L,a^R)},$$
\end{itemize}
where for conciseness  $J_{a,b,d}(a^R,a^L)$ stands for $J_{a,b,d}(\delta(\tau^{\ast}(a^R)),\delta(\tau^{\ast}(a^L)))$.

Now the structure of $\mathscr{B}(V)^{J}$ can be described explicitly in the following proposition.

\begin{proposition}\label{np7.5}
Suppose that $\mathscr{B}(V)$ is a finite-dimensional rank $2$ Nichols algebra in $_G^G \mathcal{YD}$ associated with the pair $(\mathcal{D}, T).$ Then $\mathscr{B}(V)^{J}\in\ _G^G \mathcal{YD}^{\Phi}$ with $\Phi=\partial(J)$ is isomorphic to the algebra $T_{\Phi}(V)/I \in\ _G^G \mathcal{YD}^{\Phi}$ where the ideal $I$ is generated by
\begin{equation}\label{6.4}
\begin{split}
& \{\tau^*(a) \mid a\in N_0(T)\} \cup \{\tau^*(a)^{\overrightarrow{\emph{ord}\ P_a}} \mid a\in \overline{N}_2(T), 2\leq \emph{ord}\ P_a\leq \infty\}\cup \\
&\{\frac{\tau^*(b)\tau^*(c^L)}{J_{a,b,d}(b,c^L)}-\chi(b,c^L)\frac{\tau^*(c^L)\tau^*(b)}{J_{a,b,d}(c^L ,b)}-\frac{\mu(b)}{(\mathfrak{l}(b)+1)!_{P_c}\prod_{i=1}^{\mathfrak{l}(b)}J_{a,b,d}(c^i,c)} \tau^*(c)^{\overrightarrow{\mathfrak{l}(b)+1}}\\
&\mid b\in N_2(T),c=b^L\in N_2(T)\}.
\end{split}
\end{equation}
\end{proposition}
\begin{proof}
Recall that $\mathscr{B}(V)^J$ is identical to $\mathscr{B}(V)$ as a linear space and its multiplication is obtained by a twisting from that of the latter, i.e., $E\circ F=J(e,f)EF$ for any $E,F\in \mathscr{B}(V).$

Define a map $\Psi: T_{\Phi}(V)\to T(V)$ by
\begin{equation} \label{e6.6}
\Psi((\cdots ((Y_1\circ Y_2)\circ Y_3)\cdots Y_n))=\prod_{i=1}^{n-1}J(Y_1\cdots Y_i,Y_{i+1})Y_1Y_2\cdots Y_n
\end{equation} for all $Y_i\in \{X_1,X_2\}.$ Evidently $\Psi$ is an isomorphism of linear spaces.
By the proof of Lemma \ref{nl5}, we have
\begin{equation} \label{e6.7}
\Psi(E\circ F)=J(e,f)\Psi(E)\Psi(F)
\end{equation} for any homogeneous $E,F\in T_{\Phi}(V).$

 By the definition of $\tau^*,$ one can easily show that $\Psi(\tau^*(a))=\tau(a)$ for any $a\in N(T).$ With this fact and \eqref{e6.7}, it is immediately that $\Psi$ maps relations \eqref{6.4} to relations \eqref{6.3}, hence  $\Psi$ induces an isomorphism $\overline{\Psi}:T_{\Phi}(V)/I\to \mathscr{B}(V)$ such that
\begin{equation}
\overline{\Psi}(E\circ F)=J(e,f)\overline{\Psi}(E)\overline{\Psi}(F)
\end{equation} for any $E,F\in T_{\Phi}(V)/I.$
This completes the proof of the proposition.
\end{proof}

\begin{remark} It is clear that to give the structure of $\mathscr{B}(V)^{J}$ we need to know the group $G$ and the bicharacter
$$\chi:\;G\times G\to k^{\ast},\;\;(h_{i},h_{j})\mapsto q_{ij}$$
at first. Therefore it is more convenient for us to translate Table 6.1 to the following Table 6.2.
\end{remark}

{\setlength{\unitlength}{1mm}
\begin{table}[htd]\centering
\caption{Finite-dimensional rank $2$ Nichols algebras over $\Z_{m}\times \Z_{n}$.} \label{tab.5}
\vspace{1mm}
\begin{tabular}{r|l|c}
\hline
& \text{Structure constants of Dynkin diagrams} & \text{Binary tree} \\
\hline
\hline
1. & $q_{12}q_{21}=1$ & $ T_1 $ \\
\hline
2. & $q_{12}q_{21}=q_{11}^{-1},\ q_{11}=q_{22}$ & $T_1$\\
\hline
\multirow{2}{*}{3.} & $q_{12}q_{21}=q_{11}^{-1},\ q_{11}\neq -1,\ q_{22}=-1$& $T_2$\\
\cline{2-3}
                    & $q_{12}q_{21}\neq -1,\ q_{11}=q_{22}=-1$& $T_2$\\
\hline
4. & $q_{12}q_{21}=q_{11}^{-2},\ q_{22}=q_{11}^2,\ q_{11}\neq -1$& $T_3$\\
\hline
5. & $q_{12}q_{21}=q_{11}^{-2},\ q_{11}\notin R_4,\ q_{22}=-1$& $T_3$\\
\hline
6. & $q_{12}q_{21}=q_{22}^{-1},\ q_{11}\in R_3,\ q_{22}\notin R_3$& $T_3$\\
\hline
7. & $q_{12}q_{21}=-q_{11},\ q_{11}\in R_3,\ q_{22}=-1$& $T_3$\\
\hline
\multirow{3}{*}{8.} & $q_{12}q_{21}=-\zeta^3,\ q_{11}=-\zeta^{-2},\ q_{22}=-\zeta^2,\ \zeta\in R_{12}$&$T_4$\\
\cline{2-3}
                    & $q_{12}q_{21}=\zeta^{-1},\ q_{11}=-\zeta^{-2},\ q_{22}=-1,\ \zeta\in R_{12}$&$T_5$\\
\cline{2-3}
                    & $q_{12}q_{21}=\zeta,\ q_{11}=-\zeta^{3},\ q_{22}=-1,\ \zeta\in R_{12}$& $T_7$\\
\hline
\multirow{3}{*}{9.} & $q_{12}q_{21}=\zeta,\ q_{11}=-\zeta^{2},\ q_{22}=-\zeta^{2},\ \zeta\in R_{12}$&$T_4$\\
\cline{2-3}
                    & $q_{12}q_{21}=\zeta^3,\ q_{11}=-\zeta^{2},\ q_{22}=-1,\ \zeta\in R_{12}$&$T_5$\\
\cline{2-3}
                    & $q_{12}q_{21}=-\zeta^{3}, \ q_{11}=-\zeta^{-1},\ q_{22}=-1,\ \zeta\in R_{12}$&$T_7$\\
\hline
\multirow{3}{*}{10.} & $q_{12}q_{21}=\zeta^{-2},\ q_{11}=-\zeta,\ q_{22}=\zeta^3,\ \zeta\in R_9$&$T_6$\\
\cline{2-3}
                     & $q_{12}q_{21}=\zeta^{-1},\ q_{11}=\zeta^3,\ q_{22}=-1,\ \zeta\in R_9$&$T_9$\\
\cline{2-3}
                     & $q_{12}q_{21}=\zeta,\ q_{11}=-\zeta^2,\ q_{22}=-1,\ \zeta\in R_9$&$T_{14}$\\
\hline
11.& $q_{12}q_{21}=q_{11}^{-3},\ q_{22}=q_{11}^3,\ q_{11}\neq -1,\ q_{11}\notin R_3$ & $T_8$\\
\hline
\multirow{3}{*}{12.} & $q_{12}q_{21}=\zeta,\ q_{11}=\zeta^2,\ q_{22}=\zeta^{-1},\ \zeta\in R_8$&$T_8$\\
\cline{2-3}
                     & $q_{12}q_{21}=-\zeta^{-1},\ q_{11}=\zeta^2,\ q_{22}=-1,\ \zeta\in R_8$&$T_8$\\
\cline{2-3}
                     & $q_{12}q_{21}=\zeta^{-1},\ q_{11}=\zeta,\ q_{22}=-1,\ \zeta\in R_8$&$T_8$\\
\hline
\multirow{4}{*}{13.} & $q_{12}q_{21}=-\zeta^{-1},\ q_{11}=\zeta^6,\ q_{22}=\zeta^{-4},\ \zeta\in R_{24}$&$T_{10}$\\
\cline{2-3}
                     &  $q_{12}q_{21}=\zeta,\ q_{11}=\zeta^6,\ q_{22}=\zeta^{-1},\ \zeta\in R_{24}$&$T_{13}$\\
\cline{2-3}
                     &  $q_{12}q_{21}=\zeta^{5},\ q_{11}=-\zeta^{-4},\ q_{22}=-1,\ \zeta\in R_{24}$&$T_{17}$\\
\cline{2-3}
                     & $q_{12}q_{21}=\zeta^{-5},\ q_{11}=\zeta,\ q_{22}=-1,\ \zeta\in R_{24}$&$T_{21}$\\
\hline
\multirow{2}{*}{14.} & $q_{12}q_{21}=\zeta^{2},\ q_{11}=\zeta,\ q_{22}=-1,\ \zeta\in R_5$&$T_{11}$\\
\cline{2-3}
                     & $q_{12}q_{21}=\zeta^{-2},\ q_{11}=-\zeta^{-2},\ q_{22}=-1,\ \zeta\in R_{5}$&$T_{16}$\\
\hline
\multirow{2}{*}{15.} & $q_{12}q_{21}=\zeta^{-3},\ q_{11}=\zeta,\ q_{22}=-1,\ \zeta\in R_{20}$&$T_{11}$\\
\cline{2-3}
                     & $q_{12}q_{21}=\zeta^{3},\ q_{11}=-\zeta^{-2},\ q_{22}=-1,\ \zeta\in R_{20}$&$T_{16}$\\
\hline
\multirow{4}{*}{16.} & $q_{12}q_{21}=-\zeta^{-3},\ q_{11}=-\zeta,\ q_{22}=\zeta^{5},\ \zeta\in R_{15}$&$T_{12}$\\
\cline{2-3}
                     & $q_{12}q_{21}=-\zeta^{4},\ q_{11}=\zeta^3,\ q_{22}=-\zeta^{-4},\ \zeta\in R_{15}$&$T_{15}$\\
\cline{2-3}
                     & $q_{12}q_{21}=-\zeta^{-2},\ q_{11}=\zeta^{5},\ q_{22}=-1,\ \zeta\in R_{15}$&$T_{18}$\\
\cline{2-3}
                     & $q_{12}q_{21}=-\zeta^{2},\ q_{11}=\zeta^{3},\ q_{22}=-1,\ \zeta\in R_{15}$&$T_{20}$\\
\hline
\multirow{2}{*}{17.} & $q_{12}q_{21}=-\zeta^{-3},\ q_{11}=-\zeta,\ q_{22}=-1,\ \zeta\in R_{7}$&$T_{19}$\\
\cline{2-3}
                     & $q_{12}q_{21}=-\zeta^{3},\ q_{11}=-\zeta^{-2},\ q_{22}=-1,\ \zeta\in R_{7}$&$T_{22}$\\
\hline
\end{tabular}
\end{table}}

\section{Examples: the standard case}
Keep the notations of Sections 5 and 6.
We call a rank $2$ graded pointed Majid algebra $\MM=\MM(V,G)$ over the group $\mathbb{G}=\Z_{\m}\times \Z_{\n}=\langle \g_1\rangle \times
\langle \g_2\rangle$ \emph{standard} if
\begin{equation} \delta_{L}(X_1)=\g_1\otimes X_1,\;\;\;\;\delta_{L}(X_2)=\g_2\otimes X_2.
\end{equation}
By Figure I, we have the corresponding Nichols algebra $\mathscr{B}(V)\in {^{G}_{G}\mathcal{YD}}$ in which $G=\Z_m\times \Z_n$ and
\begin{equation}
\sigma_{L}(X_1)=g_1\otimes X_1,\;\;\;\;\sigma_{L}(X_2)=g_2\otimes X_2.
\end{equation}
In this case, we also call the Nichols algebra $\mathscr{B}(V)\in {^{G}_{G}\mathcal{YD}}$ standard. As an explicit example to explain our theoretical results obtained so far, the aim of the present section is to classify standard rank $2$ graded pointed Majid algebras. Moreover, we will show that if $\m=\n$ then $\MM$ is always standard.

\subsection{General case} A nice property of standard Nichols algebras $\mathscr{B}(V)\in {^{G}_{G}\mathcal{YD}}$ is that
\begin{equation}\label{eq7.3}  \left(\begin{array}{cc}\alpha_1 & \beta_1 \\ \alpha_2 & \beta_2  \end{array}\right) =
\left(\begin{array}{cc}1 & 0 \\ 0 & 1  \end{array}\right)\;\; \text{and}\;\;\left(\begin{array}{cc}s_1 & t_1 \\ s_2 & t_2  \end{array}\right)= \left(\begin{array}{cc}1& 0 \\ 0 & 1 \end{array}\right).
\end{equation}

\begin{lemma} \label{l7.1} A standard Nichols algebra $\mathscr{B}(V)\in {^{G}_{G}\mathcal{YD}}$ is of Majid type if and only if
$$q_{21}^{\n}=1.$$
\end{lemma}
\begin{proof} It is a direct consequence of Proposition \ref{p5.14} and \eqref{eq7.3}.
\end{proof}
Therefore, Table \ref{tab7.1} gives all the possible standard rank $2$ Nichols algebras of Majid type in ${^{G}_{G}\mathcal{YD}}$.

{\setlength{\unitlength}{1mm}
\begin{table}[htd]\centering
\caption{Standard f.d. rank $2$ Nichols algebras of Majid type over $\Z_{m}\times \Z_{n}$.} \label{tab7.1}
\vspace{1mm}
\renewcommand{\arraystretch}{1.15}
\begin{tabular}{r|l|c}
\hline
& \text{Structure constants of Dynkin diagrams} & \text{Binary tree}
\\
\hline
\hline
1. & $q_{12}q_{21}=1,\ q_{21}^{\n}=1$ & $ T_1 $ \\
\hline
2. & $q_{12}q_{21}=q_{11}^{-1},\ q_{11}=q_{22},\ q_{21}^{\n}=1$ & $T_1$\\
\hline
\multirow{2}{*}{3.} & $q_{12}q_{21}=q_{11}^{-1},\ q_{11}\neq -1,\ q_{22}=-1,\ q_{21}^{\n}=1$& $T_2$\\
\cline{2-3}
                    & $q_{12}q_{21}\neq -1,\ q_{11}=q_{22}=-1,\ q_{21}^{\n}=1$& $T_2$\\
\hline
4. & $q_{12}q_{21}=q_{11}^{-2},\ q_{22}=q_{11}^2,\ q_{11}\neq -1,\ q_{21}^{\n}=1$& $T_3$\\
\hline
5. & $q_{12}q_{21}=q_{11}^{-2},\ q_{11}\notin R_4,\ q_{11}\neq -1, \ q_{22}=-1,\ q_{21}^{\n}=1$& $T_3$\\
\hline
6. & $q_{12}q_{21}=q_{22}^{-1},\ q_{11}\in R_3,\ q_{22}\notin R_3,\ q_{21}^{\n}=1$& $T_3$\\
\hline
7. & $q_{12}q_{21}=-q_{11},\ q_{11}\in R_3,\ q_{22}=-1,\ q_{21}^{\n}=1$& $T_3$\\
\hline
\multirow{3}{*}{8.} & $q_{12}q_{21}=-\zeta^3,\ q_{11}=-\zeta^{-2},\ q_{22}=-\zeta^2,\ \zeta\in R_{12},\ q_{21}^{\n}=1$&$T_4$\\
\cline{2-3}
                    & $q_{12}q_{21}=\zeta^{-1},\ q_{11}=-\zeta^{-2},\ q_{22}=-1,\ \zeta\in R_{12},\ q_{21}^{\n}=1$&$T_5$\\
\cline{2-3}
                    & $q_{12}q_{21}=\zeta,\ q_{11}=-\zeta^{3},\ q_{22}=-1,\ \zeta\in R_{12},\ q_{21}^{\n}=1$& $T_7$\\
\hline
\multirow{3}{*}{9.} & $q_{12}q_{21}=\zeta,\ q_{11}=-\zeta^{2},\ q_{22}=-\zeta^{2},\ \zeta\in R_{12},\ q_{21}^{\n}=1$&$T_4$\\
\cline{2-3}
                    & $q_{12}q_{21}=\zeta^3,\ q_{11}=-\zeta^{2},\ q_{22}=-1,\ \zeta\in R_{12},\ q_{21}^{\n}=1$&$T_5$\\
\cline{2-3}
                    & $q_{12}q_{21}=-\zeta^{3}, \ q_{11}=-\zeta^{-1},\ q_{22}=-1,\ \zeta\in R_{12},\ q_{21}^{\n}=1$&$T_7$\\
\hline
\multirow{3}{*}{10.} & $q_{12}q_{21}=\zeta^{-2},\ q_{11}=-\zeta,\ q_{22}=\zeta^3,\ \zeta\in R_9,\ q_{21}^{\n}=1$&$T_6$\\
\cline{2-3}
                     & $q_{12}q_{21}=\zeta^{-1},\ q_{11}=\zeta^3,\ q_{22}=-1,\ \zeta\in R_9,\ q_{21}^{\n}=1$&$T_9$\\
\cline{2-3}
                     & $q_{12}q_{21}=\zeta,\ q_{11}=-\zeta^2,\ q_{22}=-1,\ \zeta\in R_9,\ q_{21}^{\n}=1$&$T_{14}$\\
\hline
11.& $q_{12}q_{21}=q_{11}^{-3},\ q_{22}=q_{11}^3,\ q_{11}\neq -1,\ q_{11}\notin R_3,\ q_{21}^{\n}=1$ & $T_8$\\
\hline
\multirow{3}{*}{12.} & $q_{12}q_{21}=\zeta,\ q_{11}=\zeta^2,\ q_{22}=\zeta^{-1},\ \zeta\in R_8,\ q_{21}^{\n}=1$&$T_8$\\
\cline{2-3}
                     & $q_{12}q_{21}=-\zeta^{-1},\ q_{11}=\zeta^2,\ q_{22}=-1,\ \zeta\in R_8,\ q_{21}^{\n}=1$&$T_8$\\
\cline{2-3}
                     & $q_{12}q_{21}=\zeta^{-1},\ q_{11}=\zeta,\ q_{22}=-1,\ \zeta\in R_8,\ q_{21}^{\n}=1$&$T_8$\\
\hline
\multirow{4}{*}{13.} & $q_{12}q_{21}=-\zeta^{-1},\ q_{11}=\zeta^6,\ q_{22}=\zeta^{-4},\ \zeta\in R_{24},\ q_{21}^{\n}=1$&$T_{10}$\\
\cline{2-3}
                     & $q_{12}q_{21}=\zeta,\ q_{11}=\zeta^6,\ q_{22}=\zeta^{-1},\ \zeta\in R_{24},\ q_{21}^{\n}=1$&$T_{13}$\\
\cline{2-3}
                     & $q_{12}q_{21}=\zeta^{5},\ q_{11}=-\zeta^{-4},\ q_{22}=-1,\ \zeta\in R_{24},\ q_{21}^{\n}=1$&$T_{17}$\\
\cline{2-3}
                     & $q_{12}q_{21}=\zeta^{-5},\ q_{11}=\zeta,\ q_{22}=-1,\ \zeta\in R_{24},\ q_{21}^{\n}=1$&$T_{21}$\\
\hline
\multirow{2}{*}{14.} & $q_{12}q_{21}=\zeta^{2},\ q_{11}=\zeta,\ q_{22}=-1,\ \zeta\in R_5,\ q_{21}^{\n}=1$&$T_{11}$\\
\cline{2-3}
                     & $q_{12}q_{21}=\zeta^{-2},\ q_{11}=-\zeta^{-2},\ q_{22}=-1,\ \zeta\in R_{5},\ q_{21}^{\n}=1$&$T_{16}$\\
\hline
\multirow{2}{*}{15.} & $q_{12}q_{21}=\zeta^{-3},\ q_{11}=\zeta,\ q_{22}=-1,\ \zeta\in R_{20},\ q_{21}^{\n}=1$&$T_{11}$\\
\cline{2-3}
                     & $q_{12}q_{21}=\zeta^{3},\ q_{11}=-\zeta^{-2},\ q_{22}=-1,\ \zeta\in R_{20},\ q_{21}^{\n}=1$&$T_{16}$\\
\hline
\multirow{4}{*}{16.} & $q_{12}q_{21}=-\zeta^{-3},\ q_{11}=-\zeta,\ q_{22}=\zeta^{5},\ \zeta\in R_{15},\ q_{21}^{\n}=1$&$T_{12}$\\
\cline{2-3}
                     & $q_{12}q_{21}=-\zeta^{4},\ q_{11}=\zeta^3,\ q_{22}=-\zeta^{-4},\ \zeta\in R_{15},\ q_{21}^{\n}=1$&$T_{15}$\\
\cline{2-3}
                     & $q_{12}q_{21}=-\zeta^{-2},\ q_{11}=\zeta^{5},\ q_{22}=-1,\ \zeta\in R_{15},\ q_{21}^{\n}=1$&$T_{18}$\\
\cline{2-3}
                     & $q_{12}q_{21}=-\zeta^{2},\ q_{11}=\zeta^{3},\ q_{22}=-1,\ \zeta\in R_{15},\ q_{21}^{\n}=1$&$T_{20}$\\
\hline
\multirow{2}{*}{17.} & $q_{12}q_{21}=-\zeta^{-3},\ q_{11}=-\zeta,\ q_{22}=-1,\ \zeta\in R_{7},\ q_{21}^{\n}=1$&$T_{19}$\\
\cline{2-3}
                     & $q_{12}q_{21}=-\zeta^{3},\ q_{11}=-\zeta^{-2},\ q_{22}=-1,\ \zeta\in R_{7},\ q_{21}^{\n}=1$&$T_{22}$\\
\hline
\end{tabular}
\end{table}}

On the other hand, take a Nichols algebra $\mathscr{B}(V)$ in Table \ref{tab7.1} and from which we get $a,b,d$ uniquely due to Proposition \ref{p5.12}. As a matter of fact, in this case we have
\begin{equation} a=x_{11}',\;\;b=(\frac{\n}{\m}x_{12})',\;\;d=x_{22}''
\end{equation} for $q_{11}=\zeta_{m}^{x_{11}},\ q_{12}=\zeta_{m}^{x_{12}},\ q_{22}=\zeta_{n}^{x_{22}}$.
Then we get a Majid algebra
$$\MM(V,G):=\mathscr{B}(V)^{J_{a,b,d}}\# kG/(g^{\m}_1-1,g^{\n}_2-1).$$ So, our result in this section can be stated as follows.

\begin{theorem} \label{t7.2} Any connected graded rank $2$ pointed Majid algebra of standard type is isomorphic to $$\mathscr{B}(V)^{J_{a,b,d}}\# kG/(g^{\m}_1-1,g^{\n}_2-1)$$ for some $\mathscr{B}(V)$ given in Table \ref{tab7.1}.
\end{theorem}

\subsection{$\m=\n$ implies standard}
It turns out that the condition $\m=\n$ will always put us in the standard situation. More precisely, we have
\begin{lemma} \label{l7.3} Let $\MM$ be a rank $2$ Majid algebra over $\Z_{\m}\times \Z_{\n}$. If $\m=\n$, then $\MM$ is standard. \end{lemma}
\begin{proof} It suffices to show that any two generators of $\Z_{\m}\times \Z_{\m}$ must be standard. In fact, let $h_1,h_2\in \Z_{\m}\times \Z_{\m}$ and assume that $h_1,h_2$ generate $\Z_{\m}\times \Z_{\m}$. Therefore, $h_1^{\m}=
h_2^{\m}=1$ and thus $h_1,h_2$ can only generate elements of the form $h_1^{s}h_2^{t}$ with $0\leq s<m,\ 0\leq t<m$. On the other hand, such elements should exhaust the whole $ \Z_{\m}\times \Z_{\m}$ by assumption. This forces $\langle h_1\rangle \cap \langle h_2\rangle=\{1\}$ and so $h_1, h_2$ make a set of standard generators.
\end{proof}

\subsection{The case $\m=\n=p$ with $p$ prime.}
Furthermore, if $\m=\n=p$ is a prime number, then we will see that Table \ref{tab7.1} shrinks enormously. For a better exposition of the classification results, we present them by several separate cases, namely, by $p=2,$ $p=3$ and $p > 3.$

Before we start, we make a useful simplification of notations. Recall that $$q_{ij}=\zeta_{p^2}^{x_{ij}},\;\;\;\;1\leq i,j\leq 2,$$
so it is possible to translate the equations of $q_{ij}$ in Table \ref{tab7.1} to some simpler ones in $x_{ij}.$

\begin{proposition} \label{p7.4}
Suppose $p=2$.
\begin{itemize}
\item[(1)] Take a Nichols algebra $\mathscr{B}(V)$ in Table \ref{tab7.2}. Then $\mathscr{B}(V)^{J_{a,b,d}}$ is a Nichols algebra in ${^{\mathbbm{G}}_{\mathbbm{G}}}\mathcal{YD}^{\Phi_{a,b,d}}$ and thus $\mathscr{B}(V)^{J_{a,b,d}}\# k\mathbbm{G}$ is a connected graded rank $2$ pointed Majid algebra over $\mathbbm{G}=\Z_{2}\times \Z_2$.
\item[(2)] Any finite-dimensional connected graded rank $2$ pointed Majid algebra over $\Z_{2}\times \Z_2$ is isomorphic to $\mathscr{B}(V)^{J_{a,b,d}}\# k\mathbbm{G}$ for some $\mathscr{B}(V)$ given in Table \ref{tab7.2}.
\end{itemize}
\end{proposition}

{\setlength{\unitlength}{1mm}
\begin{table}[htd]\centering
\caption{Finite-dimensional rank $2$ Nichlos algebras over $_G^G \mathcal{Y}\mathcal{D}$ for $G=\Z_4\times \Z_4$.} \label{tab7.2}
\vspace{1mm}
\begin{tabular}{r|l|c}
\hline
  &\text{Structure constants of Dynkin diagrams} & \text{Binary tree}\\
\hline
\hline
  1.  &$ x_{12}+x_{21}=0\ \mathrm{or}\ 4,\ x_{11}\neq 0,\ x_{22}\neq 0,\ x_{21}\equiv 0 \ (\text{mod}\ 2)$ &$ T_1$\\
  \hline
  2. & $ x_{12}+x_{21}+x_{11}=4 \ \mathrm{or}\ 8,\ x_{11}=x_{22}\neq 0,\ x_{21}\equiv 0 \ (\text{mod}\ 2)$ & $T_1$\\
  \hline
  3. & $ x_{12}+x_{21}+x_{11}=4\ \mathrm{or}\ 8,\ x_{11}\neq 0, 2,\ x_{22}=2,\ x_{21}\equiv 0 \ (\text{mod}\ 2)$ &$T_2$\\
  \hline
  4. & $x_{12}+x_{21}\neq 0,2,4,6,\ x_{11}=x_{22}=2,\ x_{21}\equiv 0 \ (\text{mod}\ 2)  $& $T_2$\\
  \hline
  5. & $x_{12}+x_{21}+2x_{11}=4,8\ \mathrm{or}\ 12,\ x_{11}\neq 0, 2,\ x_{22}=2,\ x_{21}\equiv 0 \ (\text{mod}\ 2)$& $T_3$\\
  \hline
  6. & $x_{12}+x_{21}+3x_{11}\equiv 0\ (\mathrm{mod}\ 4),\ x_{22}-3x_{11}\equiv 0\ (\mathrm{mod}\ 4),\ x_{11}\neq 0,2,\ x_{21}\equiv 0 \ (\text{mod}\ 2)$& $T_8$\\
  \hline
\end{tabular}
\end{table}}

\begin{proof} Lemma \ref{l7.1} implies that all Nichols algebras $\mathscr{B}(V)$ in Table \ref{tab7.2} are of Majid type, as $x_{21}\equiv 0$ (mod $2$) in each case. Therefore $\mathscr{B}(V)^{J_{a,b,d}}$ is a Nichols algebra in ${^{\mathbbm{G}}_{\mathbbm{G}}}\mathcal{YD}^{\Phi_{a,b,d}}$ and we get a graded connected rank $2$ pointed Majid algebra $\mathscr{B}(V)^{J_{a,b,d}}\# k\mathbbm{G}$. The statement (1) is proved.

To show (2), we just need to check the Nichols algebras given in Table \ref{tab7.1} case by case if $\m=\n=p=2$ according to Theorem \ref{t7.2}.

1. For case 1 of Table \ref{tab7.1}, the condition $q_{12}q_{21}=\zeta_4^{x_{12}}\zeta_4^{x_{21}}=1$ is equivalent to $x_{12}+x_{21}=0\ \mathrm{or} \ 4$ since $0\leq x_{ij}<4.$ In addition, $q_{11}\neq 1$ and $q_{22}\neq 1$ amount to  $x_{11}\neq 0$ and $x_{22}\neq 0.$

2. For case 2 of Table \ref{tab7.1}, we have $q_{12}q_{21}=\zeta_4^{x_{12}}\zeta_4^{x_{21}}=q_{11}^{-1}=\zeta_4^{-x_{11}},$ which is the same as $x_{12}+x_{21}+x_{11}=4\ \mathrm{or} \ 8$ since $x_{11}\neq 0$ and $0\leq x_{ij}<4.$ Moreover, $q_{11}=q_{22}\neq 1$ amounts to $x_{11}=x_{22}\neq 0.$

3. The first item of case 3 of Table \ref{tab7.1}. The condition $q_{12}q_{21}=\zeta_4^{x_{12}}\zeta_4^{x_{21}}=q_{11}^{-1}=\zeta_4^{-x_{11}}$ amounts to $x_{12}+x_{21}+x_{11}=4\ \mathrm{or} \ 8,$  $q_{11}\neq \pm 1$  to $x_{11}\neq 0,2$ and $q_{22}=-1$ to $x_{22}=2.$

4. The second item of case 3 of Table \ref{tab7.1}. The condition $q_{12}q_{21}=\zeta_4^{x_{12}+x_{21}}\neq \pm 1$ is equivalent to $x_{12}+x_{21}\neq 0,2,4,6.$ Besides, $q_{11}=q_{22}=-1$ amounts to $x_{11}=x_{22}=2.$

5. Case 4 of Table \ref{tab7.1}. The equation $q_{12}q_{21}=q_{11}^{-2}$ amounts to $x_{12}+x_{21}+2x_{11}\equiv 0\ (\mathrm{mod}\ 4).$ As $0\leq x_{ij}<4$ and $x_{11}\neq 0,$ we have $x_{12}+x_{21}+2x_{11}=4,8,\ \mathrm{or }\ 12.$ In addition, $q_{11} \neq \pm 1$ amounts to $x_{11} \neq 0,2$ and thus the relation $q_{22}=q_{11}^2=-1$ is equivalent to $x_{22}=2.$

6. Case 11 of Table \ref{tab7.1}. Firstly, $q_{12}q_{21}=q_{11}^{-3}$ amounts to $x_{12}+x_{21}+3x_{11}\equiv 0\ (\mathrm{mod}\ 4).$ Secondly, $q_{22}=q_{11}^3$ amounts to $x_{22}\equiv 3x_{11}\ (\mathrm{mod} \ 4)$ and $q_{11}\neq \pm 1$ to
$x_{11}\neq 0,2.$

Finally we show that all other cases of Table \ref{tab7.1} will not occur. For case 5 of Table \ref{tab7.1}, the conditions $q_{11} \notin R_4$ and $q_{11}^4=1$ (as $p$=2) force $q_{11}=-1,$ which clearly contradicts another condition $q_{11} \ne -1$ in the same case. For the remaining cases, just notice that there are structure constants $q_{ij}$ which are not $4$-th root of unity.
\end{proof}

\begin{proposition}
Suppose $p=3$.
\begin{enumerate}
\item Take a Nichols algebra $\mathscr{B}(V)$ in Table \ref{tab7.3}. Then $\mathscr{B}(V)^{J_{a,b,d}}$ is a Nichols algebra in ${^{\mathbbm{G}}_{\mathbbm{G}}}\mathcal{YD}^{\Phi_{a,b,d}}$ and thus $\mathscr{B}(V)^{J_{a,b,d}}\# k\mathbbm{G}$ is a connected graded rank $2$ pointed Majid algebra over $\mathbbm{G}=\Z_3\times \Z_3$.

\item  Any finite-dimensional connected graded rank $2$ pointed Majid algebra over $\Z_3\times \Z_3$ is isomorphic to $\mathscr{B}(V)^{J_{a,b,d}}\# k\mathbbm{G}$ for some $\mathscr{B}(V)$ given in Table \ref{tab7.3}.
\end{enumerate}
\end{proposition}

{\setlength{\unitlength}{1mm}
\begin{table}[htd]\centering
\caption{Finite-dimensional rank $2$ Nichlos algebras over $_G^G \mathcal{Y}\mathcal{D}$ for $G=\Z_9\times \Z_9$.} \label{tab7.3}
\vspace{1mm}
\begin{tabular}{r|l|c}
\hline
  &\text{Structure constants of Dynkin diagrams} & \text{Binary tree}\\
\hline
\hline
  1.  & $x_{12}+x_{21}=0\ \mathrm{or}\ 9,x_{11}\neq 0, x_{22}\neq 0,\ x_{21}\equiv 0 \ (\text{mod}\ 3)$ & $T_1$\\
\hline
  2. & $x_{12}+x_{21}+x_{11}=9 \ \mathrm{or}\ 18, x_{11}=x_{22}\neq 0,\ x_{21}\equiv 0 \ (\text{mod}\ 3)$& $T_1$\\
  \hline
  3. & $x_{12}+x_{21}+2x_{11}\equiv 0\ (\mathrm{mod}\ 9), x_{22}-2x_{11}\equiv 0\ (\mathrm{mod}\ 9), x_{11}\neq 0,\ x_{21}\equiv 0 \ (\text{mod}\ 3)$& $T_3$\\
  \hline
  4. & $x_{12}+x_{21}+x_{22}\equiv 0\ (\mathrm{mod}\ 9), x_{11}=3\ \mathrm{or}\ 6, x_{22}\neq 0,3,6,\ x_{21}\equiv 0 \ (\text{mod}\ 3)$& $T_3$\\
  \hline
  5. & $x_{12}+x_{21}+3x_{11}\equiv 0\ (\mathrm{mod}\ 9), x_{22}-3x_{11}\equiv 0\ (\mathrm{mod}\ 9), x_{11}\neq 0,3,6,\ x_{21}\equiv 0 \ (\text{mod}\ 3)$&$T_8$\\
  \hline
\end{tabular}
\end{table}}

\begin{proof}
It is similar to the proof of Proposition \ref{p7.4}  and so we omit it.
\end{proof}

\begin{proposition}
Suppose $p>3$.
\begin{enumerate}
\item Take a Nichols algebra $\mathscr{B}(V)$ in Table \ref{tab7.4}. Then $\mathscr{B}(V)^{J_{a,b,d}}$ is a Nichols in ${^{\mathbbm{G}}_{\mathbbm{G}}}\mathcal{YD}^{\Phi_{a,b,d}}$ and thus $\mathscr{B}(V)^{J_{a,b,d}}\# k\mathbbm{G}$ is a connected graded rank $2$ pointed Majid algebra over $\mathbbm{G}=\Z_{p}\times \Z_p$.

\item  Any finite-dimensional connected graded rank $2$ pointed Majid algebra over $\Z_{p}\times \Z_p$ is isomorphic to $\mathscr{B}(V)^{J_{a,b,d}}\# k\mathbbm{G}$ for some $\mathscr{B}(V)$ given in Table \ref{tab7.4}.
\end{enumerate}
\end{proposition}

{\setlength{\unitlength}{1mm}
\begin{table}[htd]\centering
\caption{Finite-dimensional rank $2$ Nichlos algebras over $_G^G \mathcal{Y}\mathcal{D}$ for $G=\Z_{p^2}\times \Z_{p^2}$.} \label{tab7.4}
\vspace{1mm}
\begin{tabular}{r|l|c}
\hline
  &\text{Structure constants of Dynkin diagrams} & \text{Binary tree}\\
\hline
\hline
  1.  & $x_{12}+x_{21}=0\ \mathrm{or}\ p^2,x_{11}\neq 0, x_{22}\neq 0,\ x_{21}\equiv 0 \ (\text{mod}\ p)$ &$ T_1$\\
\hline
  2. & $x_{12}+x_{21}+x_{11}=p^2 \ \mathrm{or}\ 2p^2, x_{11}=x_{22}\neq 0,\ x_{21}\equiv 0 \ (\text{mod}\ p)$& $T_1$\\
\hline
  3. & $x_{12}+x_{21}+2x_{11}\equiv 0\ (\mathrm{mod}\ p^2), x_{22}-2x_{11}\equiv 0\ (\mathrm{mod}\ p^2),\ x_{21}\equiv 0 \ (\text{mod}\ p)$&$ T_3$\\
\hline
  4. & $x_{12}+x_{21}+3x_{11}\equiv 0\ (\mathrm{mod}\ p^2), x_{22}-3x_{11}\equiv 0\ (\mathrm{mod}\ p^2),\ x_{21}\equiv 0 \ (\text{mod}\ p)$&$T_8$\\
\hline
\end{tabular}
\end{table}}

\begin{proof}
Similar to the proof of Proposition \ref{p7.4}.
\end{proof}

\section{Examples: finite rank $2$ quasi-quantum groups over $\Z_\n$}

This section is devoted to a complete list of finite-dimensional connected rank $2$ pointed Majid algebras over an arbitrary cyclic group $\Z_\n.$ For the rest of the paper, $\mathbbm{G}=\Z_\n$ and $ G=\Z_{n}$ with $n=\n^2$. As before, we start with the following lemma.

\begin{lemma} \label{l8.1}
Suppose $h_1,h_2\in \Z_n$ and $h_1,h_2$ generate $\Z_n.$ Then there exists a generator $g$ of $\Z_n$ such that $h_1=g^s, \ h_2=g^t$ and $(s,t)=1.$
\end{lemma}
\begin{proof}
Let $g'$ be a generator of $\Z_n,$ then $h_1=g'^{s}, h_2=g'^{t}$ for some $0\leq s,t<n.$  Since $h_1,h_2$ generate $\Z_n,$ so we have $\langle h_1,h_2\rangle=\langle g'^{(s,t)}\rangle=\Z_n.$ This implies that $((s,t),n)=1,$ hence $g=g'^{(s,t)}$ is another generator. With this generator we have $h_1=g^{\frac{s}{(s,t)}},h_2=g^{\frac{t}{(s,t)}}$ and $(\frac{s}{(s,t)},\frac{t}{(s,t)})=1.$
\end{proof}

Again keep the notations of Sections 5 and 6. Let $\mathscr{B}(V)$ be a finite-dimensional rank $2$ Nichols algebras in ${^G_G\mathcal{YD}}$ with
$G=\Z_n=\langle g\rangle.$  By Lemma \ref{l8.1}, we can assume that
$$h_1=g^s,\;\;\;\;\;h_2=g^t  \quad \mathrm{with} \quad (s,t)=1.$$
Furthermore, if we assume $g\diamond X_1=\zeta_{n}^\alpha X_1,\ g\diamond X_2=\zeta_{n}^\beta X_2$ for some $0\leq \alpha,\beta<n,$ then we have
\begin{eqnarray}
q_{11}&=&\zeta_{n}^{s\alpha},\ \ \ \ \ q_{12}= \zeta_{n}^{s\beta},\notag \\
q_{21}&=&\zeta_{n}^{t\alpha},\ \ \ \ \ q_{22}=\zeta_{n}^{t\beta}.\notag
\end{eqnarray}
So equations \eqref{eq5.5}, or equivalently the congruence equations \eqref{eq6.8} are reduced to :
\begin{equation} \label{eq8.1}
ds\equiv \alpha \ (\mathrm{mod}\ \n),\ \ \ \ dt \equiv \beta \ (\mathrm{mod}\ \n).
\end{equation}
We remark that in this situation there is no the element
$g_1$ and   $a=b=0,\ g_2=g.$
Therefore, Proposition \ref{p5.14} implies that
\begin{lemma} The system of equations \eqref{eq8.1} is soluble if and only if
\begin{equation}\alpha t\equiv \beta s\ (\mathrm{mod}\ \n).\end{equation}
\end{lemma}
Now the $2$-cochain $J_{a,b,d}$ with $d=(t_1\alpha+t_2\beta)''$ (where $t_1, t_2$ satisfy $t_1s+t_2t=1$) is simplified as:
$$J_{a,b,d}=J_{d}:\; k\mathbb{Z}_n\otimes k\mathbb{Z}_{n}\to k,\;\;\;\;(g^i,g^j)\mapsto q^{di(j-j')}.$$
As $\mathscr{B}(V)$ is always assumed to be finite-dimensional, hence  $$s\alpha\not\equiv 0\not\equiv t\beta\ (\mathrm{mod} \ n).$$ For simplicity, we mark the following conditions by $(\ast)$:
\begin{equation*} \alpha t\equiv \beta s\ (\mathrm{mod}\ \n),\;\;\;\; s\alpha\not\equiv 0\not\equiv t\beta\ (\mathrm{mod} \ n).  \;\;\;\; (\ast)\end{equation*}

\begin{theorem} \label{t8.3}
 \begin{enumerate}
 \item Take a Nichols algebra $\mathscr{B}(V)$ in Table \ref{tab8.1}. Then $\mathscr{B}(V)^{J_d}$ is a Nichols algebra in ${^{\mathbbm{G}}_{\mathbbm{G}}}\mathcal{YD}^{\Phi_d}$ and thus $\mathscr{B}(V)^{J_d}\# k\mathbbm{G}$ is a connected graded rank $2$ pointed Majid algebra over $\mathbbm{G}=\Z_{\n}$.

\item Any finite-dimensional connected graded rank $2$ pointed Majid algebra over $\Z_{\n}$ is isomorphic to $\mathscr{B}(V)^{J_d}\# k\mathbbm{G}$ for some $\mathscr{B}(V)$ given in Table \ref{tab8.1}.
\end{enumerate}
\end{theorem}

{\setlength{\unitlength}{1mm}
\begin{table}[htd]\centering
\caption{Finite-dimensional rank $2$ Nichlos algebras over $_G^G \mathcal{Y}\mathcal{D}$ for $G=\Z_{n}.$} \label{tab8.1}
\vspace{1mm}
\renewcommand{\arraystretch}{1.2}
\begin{tabular}{r|l|l|c}
\hline
& \text{$\n$} &\text{Structure constants of Dynkin diagrams} & \text{Binary tree}\\
\hline
\hline
1. & & $s\beta+t\alpha\equiv 0\ (\mathrm{mod}\ n)$, \ ($\ast$) & $ T_1 $ \\
\hline
2. & & $s\beta+t\alpha\equiv -s\alpha \equiv -t\beta\ (\mathrm{mod}\ n) $, \ ($\ast$)& $T_1$\\
\hline
\multirow{2}{*}{3.} & \multirow{2}{*}{$2|\n$} & $s\beta+t\alpha+s\alpha \equiv 0,\ s\alpha\not\equiv \frac{n}{2},\ t\beta\equiv \frac{n}{2}\ (\mathrm{mod}\ n)$,\ ($\ast$)&$ T_2$\\
\cline{3-4}
                    &                         & $s\beta+t\alpha\not\equiv 0,\frac{n}{2}, \ s\alpha\equiv t\beta\equiv \frac{n}{2}\ (\mathrm{mod}\ n)$ , \ ($\ast$)& $T_2$\\
\hline
4. & & $s\beta+t\alpha+2s\alpha \equiv 0, 2s\alpha\equiv t\beta\ (\mathrm{mod}\ n)$ , \ ($\ast$)& $T_3$\\
\hline
5. & $2|\n$ & $s\beta+t\alpha+2s\alpha \equiv 0,s\alpha\not\equiv \frac{n}{4},\frac{n}{2},\frac{3n}{4},\ t\beta\equiv \frac{n}{2}(\mathrm{mod}\ n)$, \ ($\ast$) & $ T_3$\\
\hline
6. & $3|\n$ & $s\beta+t\alpha+2s\alpha \equiv 0, 3s\alpha \equiv 0,\ 3t\beta\not\equiv 0 \ (\mathrm{mod}\ n) $ , \ ($\ast$)&$T_3$\\
\hline
7. & $6|\n$ & $s\beta+t\alpha+2s\alpha \equiv 0, 3s\alpha \equiv 0, t\beta \equiv  \frac{n}{2}\ (\mathrm{mod}\ n)$, \ ($\ast$)& $T_3$\\
\hline
\multirow{3}{*}{8.} & \multirow{3}{*}{$6|\n$} & $s\beta+t\alpha\equiv \frac{(k+2)n}{4},s\alpha\equiv \frac{(3-k)n}{6},t\beta\equiv\frac{(3+k)n}{4}\ (\mathrm{mod}\ n), \ (k,12)=1$, \ ($\ast$) &$T_4$\\
\cline{3-4}
                    &                         & $s\beta+t\alpha\equiv \frac{kn}{12},s\alpha\equiv \frac{(3+k)n}{6},t\beta\equiv\frac{n}{2}\ (\mathrm{mod}\ n),\ (k,12)=1$, \ ($\ast$) &$T_5$\\
\cline{3-4}
                    &                         & $s\beta+t\alpha\equiv \frac{kn}{12},s\alpha\equiv \frac{(2+k)n}{4},t\beta\equiv\frac{n}{2}\ (\mathrm{mod}\ n),\ (k,12)=1$, \ ($\ast$)&$T_7$\\
\hline
\multirow{3}{*}{9.} & \multirow{3}{*}{$6|\n$} & $s\beta+t\alpha\equiv \frac{kn}{12},s\alpha\equiv \frac{(3+k)n}{6}\equiv t\beta\ (\mathrm{mod}\ n), \ (k,12)=1$, \ ($\ast$)&$T_4$\\
\cline{3-4}
                    &                         & $s\beta+t\alpha\equiv\frac{kn}{4},s\alpha\equiv \frac{(3+k)n}{6},t\beta\equiv\frac{n}{2}\ (\mathrm{mod}\ n),\ (k,12)=1$, \ ($\ast$)&$T_5$\\
\cline{3-4}
                    &                         & $s\beta+t\alpha\equiv\frac{(2+k)n}{4},s\alpha\equiv \frac{(6-k)n}{12},t\beta\equiv\frac{n}{2}\ (\mathrm{mod}\ n),\ (k,12)=1$, \ ($\ast$)&$T_7$\\
\hline
\multirow{3}{*}{10.} & \multirow{3}{*}{$6|\n$} & $s\beta+t\alpha\equiv \frac{-2kn}{9},s\alpha\equiv \frac{(9+2k)n}{18},t\beta\equiv \frac{kn}{3}\ (\mathrm{mod}\ n), \ (k,18)=1$, \ ($\ast$)&$T_6$\\
\cline{3-4}
                     &                         & $s\beta+t\alpha\equiv \frac{-kn}{9},s\alpha\equiv \frac{kn}{3},t\beta\equiv \frac{n}{2}\ (\mathrm{mod}\ n),\ (k,9)=1$, \ ($\ast$)&$T_9$\\
\cline{3-4}
                     &                         & $s\beta+t\alpha\equiv \frac{kn}{9},s\alpha\equiv \frac{(9+4k)n}{18},t\beta\equiv \frac{n}{2}\ (\mathrm{mod}\ n), \ (k,18)=1$&$T_{14}$\\
\hline
11. & & $s\beta+t\alpha+3s\alpha\equiv 0, 3s\alpha\equiv t\beta, s\alpha\not\equiv \frac{n}{2},\frac{n}{3},\frac{2n}{3}\ (\mathrm{mod}\ n) $, \ ($\ast$)& $T_8$\\
\hline
\multirow{3}{*}{12.} & \multirow{3}{*}{$4|\n$} & $s\beta+t\alpha\equiv \frac{kn}{8},s\alpha\equiv \frac{kn}{4},t\beta\equiv \frac{-kn}{8}\ (\mathrm{mod}\ n),\ (k,8)=1$, \ ($\ast$)&$T_8$\\
\cline{3-4}
                     &                         & $s\beta+t\alpha\equiv \frac{(4-k)n}{8},s\alpha\equiv \frac{kn}{4},t\beta\equiv \frac{n}{2}\ (\mathrm{mod}\ n),\ (k,8)=1$, \ ($\ast$)&$T_8$\\
\cline{3-4}
                     &                         & $s\beta+t\alpha\equiv  \frac{(4+k)n}{8},s\alpha\equiv \frac{kn}{8},t\beta\equiv \frac{n}{2}\ (\mathrm{mod}\ n),\ (k,8)=1$, \ ($\ast$)&$T_8$\\
\hline
\multirow{4}{*}{13.} & \multirow{4}{*}{$12|\n$} & $s\beta+t\alpha\equiv \frac{(12-k)n}{24},s\alpha\equiv \frac{kn}{4},t\beta\equiv \frac{-kn}{6}\ (\mathrm{mod}\ n), \ (k,24)=1$, \ ($\ast$)&$T_{10}$\\
\cline{3-4}
                     &                          & $s\beta+t\alpha\equiv \frac{kn}{24},s\alpha\equiv \frac{kn}{4},t\beta\equiv \frac{-kn}{24}\ (\mathrm{mod}\ n),\ (k,24)=1$, \ ($\ast$)&$T_{13}$\\
\cline{3-4}
                     &                          & $s\beta+t\alpha\equiv \frac{5kn}{24},s\alpha\equiv \frac{(3-k)n}{6},t\beta\equiv \frac{n}{2}\ (\mathrm{mod}\ n),\ (k,24)=1$, \ ($\ast$)&$T_{17}$\\
\cline{3-4}
                     &                          & $s\beta+t\alpha\equiv \frac{-5kn}{24},s\alpha\equiv \frac{kn}{24},t\beta\equiv \frac{n}{2}\ (\mathrm{mod}\ n),\ (k,24)=1$, \ ($\ast$)&$T_{21}$\\
\hline
\multirow{2}{*}{14.} & \multirow{2}{*}{$10|\n$} & $s\beta+t\alpha\equiv \frac{2kn}{5},s\alpha\equiv \frac{kn}{5},t\beta\equiv \frac{n}{2}\ (\mathrm{mod}\ n),\ (k,5)=1$, \ ($\ast$)&$T_{11}$\\
\cline{3-4}
                     &                          & $s\beta+t\alpha\equiv \frac{-2kn}{5},s\alpha\equiv \frac{(5-4k)n}{10},t\beta\equiv \frac{n}{2}\ (\mathrm{mod}\ n),\ (k,10)=1$, \ ($\ast$)&$T_{16}$\\
\hline
\multirow{2}{*}{15.} & \multirow{2}{*}{$10|\n$} & $s\beta+t\alpha\equiv \frac{-3kn}{20},s\alpha\equiv \frac{kn}{20},t\beta\equiv \frac{n}{2}\ (\mathrm{mod}\ n),\ (k,20)=1$, \ ($\ast$)&$T_{11}$\\
\cline{3-4}
                     &                          & $s\beta+t\alpha\equiv \frac{3kn}{20},s\alpha\equiv \frac{(5-k)n}{10},t\beta\equiv \frac{n}{2}\ (\mathrm{mod}\ n),\ (k,20)=1$, \ ($\ast$)&$T_{16}$\\
\hline
\multirow{4}{*}{16.} & \multirow{4}{*}{$30|\n$} & $s\beta+t\alpha\equiv \frac{(15+8k)n}{30},s\alpha\equiv\frac{(15+2k)n}{30},t\beta\equiv\frac{kn}{3}\ (\mathrm{mod}\ n),\ (k,30)=1$, \ ($\ast$)&$T_{12}$\\
\cline{3-4}
                     &                          & $s\beta+t\alpha\equiv \frac{(15-4k)n}{30},s\alpha\equiv \frac{kn}{5},t\beta\equiv \frac{(15-8k)n}{30}\ (\mathrm{mod}\ n),\ (k,30)=1$, \ ($\ast$)&$T_{15}$\\
\cline{3-4}
                     &                          & $s\beta+t\alpha\equiv \frac{(15-4k)n}{30},s\alpha\equiv \frac{kn}{3},t\beta\equiv \frac{n}{2}\ (\mathrm{mod}\ n),\ (k,30)=1$, \ ($\ast$)&$T_{18}$\\
\cline{3-4}
                     &                          & $s\beta+t\alpha\equiv \frac{(15+4k)n}{30},s\alpha\equiv \frac{kn}{5},t\beta\equiv \frac{n}{2}\ (\mathrm{mod}\ n),\ (k,30)=1$, \ ($\ast$)&$T_{20}$\\
\hline
\multirow{2}{*}{17.} & \multirow{2}{*}{$14|\n$} & $s\beta+t\alpha\equiv \frac{(7-6k)n}{14},s\alpha\equiv \frac{(7+2k)n}{14},t\beta\equiv \frac{n}{2}\ (\mathrm{mod}\ n),\ (k,14)=1$, \ ($\ast$)&$T_{19}$\\
\cline{3-4}
                     &                          & $s\beta+t\alpha\equiv \frac{(7+6k)n}{14},s\alpha\equiv \frac{(7-4k)n}{14},t\beta\equiv \frac{n}{2}\ (\mathrm{mod}\ n),\ (k,14)=1$, \ ($\ast$)&$T_{22}$\\
\hline
\end{tabular}
\end{table}}

\begin{proof}
The proof can be carried out in the same way as that of Proposition \ref{p7.4}. We just work on two typical examples, namely case 2 and the first item of case 8, to explain our equations about structure constants in Table \ref{tab8.1}.

\emph{Case 2.} Firstly, the equation $q_{12}q_{21}=\zeta_{n}^{s\beta}\zeta_{n}^{t\alpha}=q_{11}^{-1}=\zeta_{n}^{-s\alpha}$ is equivalent to $s\beta+t\alpha+s\alpha \equiv 0\ (\mathrm{mod} \ n).$ Secondly, the condition $q_{11}=q_{22}\neq 1$ amounts to $s\alpha\equiv t\beta\not\equiv 0\ (\mathrm{mod}\ n).$

\emph{Case 8, item 1.} According to Table \ref{tab7.1}, $q_{11}=-\zeta^{-2}$ is a $6$-th primitive root of unit, it follows that $6|n,$ and further that $6|\n.$ Let $\zeta_{12}=\zeta_{n}^{\frac{n}{12}}$ be a $12$-th primitive root of unit. Then $\zeta=\zeta_{12}^k$ for some $k$ such that $(k,12)=1.$ Again by Table \ref{tab7.1}, $q_{12}q_{21}=-\zeta^3,$ i.e. $\zeta_{n}^{s\beta}\zeta_{n}^{t\alpha}=\zeta_{12}^{6+3k}=\zeta_{n}^{\frac{(6+3k)n}{12}},$ this amounts to $ s\beta+t\alpha\equiv \frac{(k+2)n}{4} \ (\mathrm{mod}\ n).$ In addition, the equation $q_{11}=-\zeta^{-2}$ is equivalent to $s\alpha\equiv \frac{(3-k)n}{6}\ (\mathrm{mod}\ n)$ and $q_{22}=-\zeta^2$ is equivalent to $t\beta\equiv\frac{(3+k)n}{4}\ (\mathrm{mod}\ n).$
\end{proof}

\begin{remarks} We conclude the paper with several remarks.
\begin{enumerate}
\item In Theorem \ref{t8.3}, if $n$ is a prime number and $n > 2$, then one can check Table \ref{tab8.1} case by case and conclude that there are no genuine rank $2$ graded pointed Majid algebras. This fact offers another explanation to a seemingly mysterious (at least to the authors) result of Etingof and Gelaki \cite[Theorem 3.1]{EG2}, which states that elementary graded genuine quasi-Hopf algebras over a cyclic group of prime($\ne 2$) order are of rank $\le 1.$
\item If $n$ has no small prime divisors, namely 2, 3, 5, 7, then essentially the finite-dimensional grtaded pointed Majid algebras constructed in Theorem \ref{t8.3} already appeared in the recent work of Angiono \cite{A}, only in the dual version.
\item The idea of the present paper should be useful in a much broader context. In particular, the results of the rank 2 case may be extended to diagonal Nichols algebras in $_G^G\mathcal{YD}^\Phi$ of higher ranks in a similar working philosophy of the theory of pointed Hopf algebras. This shall be dealt with in our forthcoming work.
\end{enumerate}
\end{remarks}

\begin{appendix}
\section{Full Binary Trees}
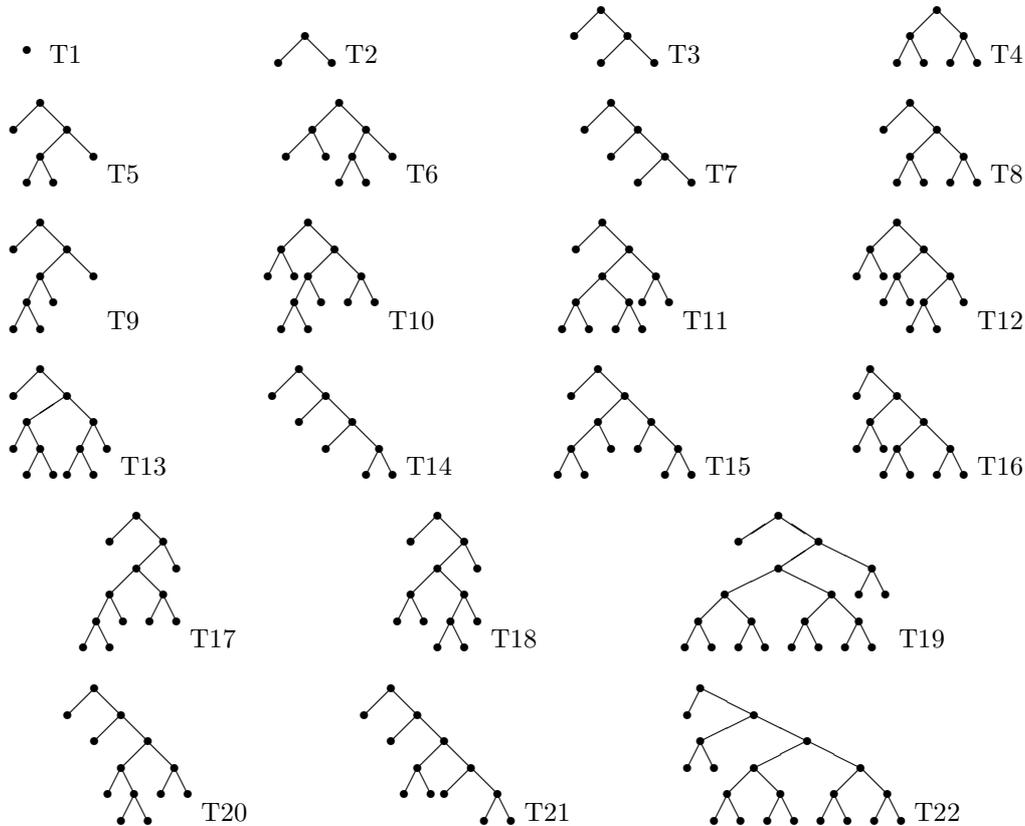
\begin{center}
\begin{minipage}{0.8\textwidth} 
\begin{center}
\setlength{\unitlength}{5pt}
\begin{picture}(2,2)
\put(1,1){\circle*{.5}}
\end{picture}
T1
\hfill \hfill
\begin{picture}(4,2)
\put(2,2){\line(-1,-1){2}}
\put(2,2){\line(1,-1){2}}
\put(2,2){\circle*{.5}}
\put(0,0){\circle*{.5}}
\put(4,0){\circle*{.5}}
\end{picture}
\,T2
\hfill \hfill
\begin{picture}(6,4)
\put(0,2){\line(1,1){2}}
\put(2,4){\line(1,-1){4}}
\put(2,0){\line(1,1){2}}
\put(2,4){\circle*{.5}}
\put(0,2){\circle*{.5}}
\put(4,2){\circle*{.5}}
\put(2,0){\circle*{.5}}
\put(6,0){\circle*{.5}}
\end{picture}
\,T3
\hfill \hfill
\begin{picture}(6,4)
\put(3,4){\line(-1,-1){2}}
\put(3,4){\line(1,-1){2}}
\put(1,2){\line(1,-2){1}}
\put(1,2){\line(-1,-2){1}}
\put(5,2){\line(-1,-2){1}}
\put(5,2){\line(1,-2){1}}
\put(3,4){\circle*{.5}}
\put(1,2){\circle*{.5}}
\put(5,2){\circle*{.5}}
\put(0,0){\circle*{.5}}
\put(2,0){\circle*{.5}}
\put(4,0){\circle*{.5}}
\put(6,0){\circle*{.5}}
\end{picture}
\,T4
\end{center}
\vspace{1em}
\begin{center}
\setlength{\unitlength}{5pt}
\begin{picture}(6,6)
\put(0,4){\line(1,1){2}}
\put(1,0){\line(1,2){1}}
\put(2,2){\line(1,1){2}}
\put(2,2){\line(1,-2){1}}
\put(2,6){\line(1,-1){4}}
\put(2,6){\circle*{.5}}
\put(0,4){\circle*{.5}}
\put(4,4){\circle*{.5}}
\put(2,2){\circle*{.5}}
\put(6,2){\circle*{.5}}
\put(1,0){\circle*{.5}}
\put(3,0){\circle*{.5}}
\end{picture}
\,T5
\hfill \hfill
\begin{picture}(8,6)
\put(0,2){\line(1,1){4}}
\put(2,4){\line(1,-2){1}}
\put(4,6){\line(1,-1){4}}
\put(6,4){\line(-1,-2){2}}
\put(5,2){\line(1,-2){1}}
\put(4,6){\circle*{.5}}
\put(2,4){\circle*{.5}}
\put(6,4){\circle*{.5}}
\put(0,2){\circle*{.5}}
\put(3,2){\circle*{.5}}
\put(5,2){\circle*{.5}}
\put(8,2){\circle*{.5}}
\put(4,0){\circle*{.5}}
\put(6,0){\circle*{.5}}
\end{picture}
\,T6
\hfill \hfill
\begin{picture}(8,6)
\put(0,4){\line(1,1){2}}
\put(2,2){\line(1,1){2}}
\put(4,0){\line(1,1){2}}
\put(2,6){\line(1,-1){6}}
\put(2,6){\circle*{.5}}
\put(0,4){\circle*{.5}}
\put(4,4){\circle*{.5}}
\put(2,2){\circle*{.5}}
\put(6,2){\circle*{.5}}
\put(4,0){\circle*{.5}}
\put(8,0){\circle*{.5}}
\end{picture}
\,T7
\hfill \hfill
\begin{picture}(7,6)
\put(0,4){\line(1,1){2}}
\put(1,0){\line(1,2){1}}
\put(2,2){\line(1,1){2}}
\put(2,2){\line(1,-2){1}}
\put(2,6){\line(1,-1){4}}
\put(6,2){\line(-1,-2){1}}
\put(6,2){\line(1,-2){1}}
\put(2,6){\circle*{.5}}
\put(0,4){\circle*{.5}}
\put(4,4){\circle*{.5}}
\put(2,2){\circle*{.5}}
\put(6,2){\circle*{.5}}
\put(1,0){\circle*{.5}}
\put(3,0){\circle*{.5}}
\put(5,0){\circle*{.5}}
\put(7,0){\circle*{.5}}
\end{picture}
\,T8
\end{center}
\vspace{1em}
\begin{center}
\setlength{\unitlength}{5pt}
\begin{picture}(6,8)
\put(0,6){\line(1,1){2}}
\put(1,2){\line(1,2){1}}
\put(2,4){\line(1,1){2}}
\put(2,4){\line(1,-2){1}}
\put(2,8){\line(1,-1){4}}
\put(1,2){\line(-1,-2){1}}
\put(1,2){\line(1,-2){1}}
\put(2,8){\circle*{.5}}
\put(0,6){\circle*{.5}}
\put(4,6){\circle*{.5}}
\put(2,4){\circle*{.5}}
\put(6,4){\circle*{.5}}
\put(1,2){\circle*{.5}}
\put(3,2){\circle*{.5}}
\put(0,0){\circle*{.5}}
\put(2,0){\circle*{.5}}
\end{picture}
\,T9
\hfill \hfill
\begin{picture}(8,8)
\put(1,6){\line(-1,-2){1}}
\put(1,6){\line(1,-2){1}}
\put(3,8){\line(-1,-1){2}}
\put(3,8){\line(1,-1){4}}
\put(5,6){\line(-1,-1){2}}
\put(3,4){\line(-1,-2){1}}
\put(3,4){\line(1,-2){1}}
\put(7,4){\line(-1,-2){1}}
\put(7,4){\line(1,-2){1}}
\put(2,2){\line(-1,-2){1}}
\put(2,2){\line(1,-2){1}}
\put(3,8){\circle*{.5}}
\put(1,6){\circle*{.5}}
\put(5,6){\circle*{.5}}
\put(0,4){\circle*{.5}}
\put(2,4){\circle*{.5}}
\put(3,4){\circle*{.5}}
\put(7,4){\circle*{.5}}
\put(2,2){\circle*{.5}}
\put(4,2){\circle*{.5}}
\put(6,2){\circle*{.5}}
\put(8,2){\circle*{.5}}
\put(1,0){\circle*{.5}}
\put(3,0){\circle*{.5}}
\end{picture}
\,T10
\hfill \hfill
\begin{picture}(8,8)
\put(0,0){\line(1,2){1}}
\put(1,2){\line(1,-2){1}}
\put(1,6){\line(1,1){2}}
\put(3,8){\line(1,-1){4}}
\put(1,2){\line(1,1){4}}
\put(3,4){\line(1,-1){2}}
\put(5,2){\line(-1,-2){1}}
\put(5,2){\line(1,-2){1}}
\put(7,4){\line(-1,-2){1}}
\put(7,4){\line(1,-2){1}}
\put(3,8){\circle*{.5}}
\put(1,6){\circle*{.5}}
\put(5,6){\circle*{.5}}
\put(3,4){\circle*{.5}}
\put(7,4){\circle*{.5}}
\put(1,2){\circle*{.5}}
\put(5,2){\circle*{.5}}
\put(6,2){\circle*{.5}}
\put(8,2){\circle*{.5}}
\put(0,0){\circle*{.5}}
\put(2,0){\circle*{.5}}
\put(4,0){\circle*{.5}}
\put(6,0){\circle*{.5}}
\end{picture}
\,T11
\hfill \hfill
\begin{picture}(8,8)
\put(1,6){\line(-1,-2){1}}
\put(1,6){\line(1,-2){1}}
\put(3,8){\line(-1,-1){2}}
\put(3,8){\line(1,-1){4}}
\put(5,6){\line(-1,-1){2}}
\put(3,4){\line(-1,-2){1}}
\put(3,4){\line(1,-2){1}}
\put(7,4){\line(-1,-1){2}}
\put(7,4){\line(1,-2){1}}
\put(5,2){\line(-1,-2){1}}
\put(5,2){\line(1,-2){1}}
\put(3,8){\circle*{.5}}
\put(1,6){\circle*{.5}}
\put(5,6){\circle*{.5}}
\put(0,4){\circle*{.5}}
\put(2,4){\circle*{.5}}
\put(3,4){\circle*{.5}}
\put(7,4){\circle*{.5}}
\put(2,2){\circle*{.5}}
\put(4,2){\circle*{.5}}
\put(5,2){\circle*{.5}}
\put(8,2){\circle*{.5}}
\put(4,0){\circle*{.5}}
\put(6,0){\circle*{.5}}
\end{picture}
\,T12
\end{center}
\vspace{1em}
\begin{center}
\setlength{\unitlength}{5pt}
\begin{picture}(7,8)
\put(2,8){\line(-1,-1){2}}
\put(2,8){\line(1,-1){4}}
\put(4,6){\line(-3,-2){3}}
\put(1,4){\line(-1,-2){1}}
\put(1,4){\line(1,-2){1}}
\put(6,4){\line(-1,-2){1}}
\put(6,4){\line(1,-2){1}}
\put(2,2){\line(-1,-2){1}}
\put(2,2){\line(1,-2){1}}
\put(5,2){\line(-1,-2){1}}
\put(5,2){\line(1,-2){1}}
\put(2,8){\circle*{.5}}
\put(0,6){\circle*{.5}}
\put(4,6){\circle*{.5}}
\put(1,4){\circle*{.5}}
\put(6,4){\circle*{.5}}
\put(0,2){\circle*{.5}}
\put(2,2){\circle*{.5}}
\put(5,2){\circle*{.5}}
\put(7,2){\circle*{.5}}
\put(1,0){\circle*{.5}}
\put(3,0){\circle*{.5}}
\put(4,0){\circle*{.5}}
\put(6,0){\circle*{.5}}
\end{picture}
\,T13
\hfill \hfill
\begin{picture}(9,8)
\put(2,8){\line(-1,-1){2}}
\put(2,8){\line(1,-1){6}}
\put(4,6){\line(-1,-1){2}}
\put(6,4){\line(-1,-1){2}}
\put(8,2){\line(-1,-2){1}}
\put(8,2){\line(1,-2){1}}
\put(2,8){\circle*{.5}}
\put(0,6){\circle*{.5}}
\put(4,6){\circle*{.5}}
\put(2,4){\circle*{.5}}
\put(6,4){\circle*{.5}}
\put(4,2){\circle*{.5}}
\put(8,2){\circle*{.5}}
\put(7,0){\circle*{.5}}
\put(9,0){\circle*{.5}}
\end{picture}
\,T14
\hfill \hfill
\begin{picture}(10,8)
\put(3,8){\line(-1,-1){2}}
\put(3,8){\line(1,-1){6}}
\put(5,6){\line(-1,-1){4}}
\put(3,4){\line(1,-2){1}}
\put(7,4){\line(-1,-2){1}}
\put(1,2){\line(-1,-2){1}}
\put(1,2){\line(1,-2){1}}
\put(9,2){\line(-1,-2){1}}
\put(9,2){\line(1,-2){1}}
\put(3,8){\circle*{.5}}
\put(1,6){\circle*{.5}}
\put(5,6){\circle*{.5}}
\put(3,4){\circle*{.5}}
\put(7,4){\circle*{.5}}
\put(1,2){\circle*{.5}}
\put(4,2){\circle*{.5}}
\put(6,2){\circle*{.5}}
\put(9,2){\circle*{.5}}
\put(0,0){\circle*{.5}}
\put(2,0){\circle*{.5}}
\put(8,0){\circle*{.5}}
\put(10,0){\circle*{.5}}
\end{picture}
\,T15
\hfill \hfill
\begin{picture}(8,8)
\put(1,8){\line(-1,-2){1}}
\put(1,8){\line(1,-1){6}}
\put(3,6){\line(-1,-1){2}}
\put(1,4){\line(-1,-2){1}}
\put(1,4){\line(1,-2){1}}
\put(5,4){\line(-1,-1){2}}
\put(3,2){\line(-1,-2){1}}
\put(3,2){\line(1,-2){1}}
\put(7,2){\line(-1,-2){1}}
\put(7,2){\line(1,-2){1}}
\put(1,8){\circle*{.5}}
\put(0,6){\circle*{.5}}
\put(3,6){\circle*{.5}}
\put(1,4){\circle*{.5}}
\put(5,4){\circle*{.5}}
\put(0,2){\circle*{.5}}
\put(2,2){\circle*{.5}}
\put(3,2){\circle*{.5}}
\put(7,2){\circle*{.5}}
\put(2,0){\circle*{.5}}
\put(4,0){\circle*{.5}}
\put(6,0){\circle*{.5}}
\put(8,0){\circle*{.5}}
\end{picture}
\,T16
\end{center}
\vspace{1em}
\begin{center}
\setlength{\unitlength}{5pt}
\hfill
\begin{picture}(7,10)
\put(4,10){\line(-1,-1){2}}
\put(4,10){\line(1,-1){2}}
\put(6,8){\line(-1,-1){4}}
\put(6,8){\line(1,-2){1}}
\put(4,6){\line(1,-1){2}}
\put(2,4){\line(-1,-2){2}}
\put(2,4){\line(1,-2){1}}
\put(6,4){\line(-1,-2){1}}
\put(6,4){\line(1,-2){1}}
\put(1,2){\line(1,-2){1}}
\put(4,10){\circle*{.5}}
\put(2,8){\circle*{.5}}
\put(6,8){\circle*{.5}}
\put(4,6){\circle*{.5}}
\put(7,6){\circle*{.5}}
\put(2,4){\circle*{.5}}
\put(6,4){\circle*{.5}}
\put(1,2){\circle*{.5}}
\put(3,2){\circle*{.5}}
\put(5,2){\circle*{.5}}
\put(7,2){\circle*{.5}}
\put(0,0){\circle*{.5}}
\put(2,0){\circle*{.5}}
\end{picture}
\,T17
\hfill \hfill
\begin{picture}(7,10)
\put(4,10){\line(-1,-1){2}}
\put(4,10){\line(1,-1){2}}
\put(6,8){\line(-1,-1){4}}
\put(6,8){\line(1,-2){1}}
\put(4,6){\line(1,-1){2}}
\put(2,4){\line(-1,-2){1}}
\put(2,4){\line(1,-2){1}}
\put(6,4){\line(-1,-2){2}}
\put(6,4){\line(1,-2){1}}
\put(5,2){\line(1,-2){1}}
\put(4,10){\circle*{.5}}
\put(2,8){\circle*{.5}}
\put(6,8){\circle*{.5}}
\put(4,6){\circle*{.5}}
\put(7,6){\circle*{.5}}
\put(2,4){\circle*{.5}}
\put(6,4){\circle*{.5}}
\put(1,2){\circle*{.5}}
\put(3,2){\circle*{.5}}
\put(5,2){\circle*{.5}}
\put(7,2){\circle*{.5}}
\put(4,0){\circle*{.5}}
\put(6,0){\circle*{.5}}
\end{picture}
\,T18
\hfill \hfill
\begin{picture}(15,10)
\put(1,2){\line(1,-2){1}}
\put(1,2){\line(-1,-2){1}}
\put(5,2){\line(1,-2){1}}
\put(5,2){\line(-1,-2){1}}
\put(9,2){\line(1,-2){1}}
\put(9,2){\line(-1,-2){1}}
\put(13,2){\line(1,-2){1}}
\put(13,2){\line(-1,-2){1}}
\put(3,4){\line(1,-1){2}}
\put(3,4){\line(-1,-1){2}}
\put(11,4){\line(1,-1){2}}
\put(11,4){\line(-1,-1){2}}
\put(7,6){\line(2,-1){4}}
\put(7,6){\line(-2,-1){4}}
\put(14,6){\line(1,-2){1}}
\put(14,6){\line(-1,-2){1}}
\put(10,8){\line(-3,-2){3}}
\put(10,8){\line(2,-1){4}}
\put(7,10){\line(-3,-2){3}}
\put(7,10){\line(3,-2){3}}
\put(7,10){\circle*{.5}}
\put(4,8){\circle*{.5}}
\put(10,8){\circle*{.5}}
\put(7,6){\circle*{.5}}
\put(14,6){\circle*{.5}}
\put(3,4){\circle*{.5}}
\put(11,4){\circle*{.5}}
\put(13,4){\circle*{.5}}
\put(15,4){\circle*{.5}}
\put(1,2){\circle*{.5}}
\put(5,2){\circle*{.5}}
\put(9,2){\circle*{.5}}
\put(13,2){\circle*{.5}}
\put(0,0){\circle*{.5}}
\put(2,0){\circle*{.5}}
\put(4,0){\circle*{.5}}
\put(6,0){\circle*{.5}}
\put(8,0){\circle*{.5}}
\put(10,0){\circle*{.5}}
\put(12,0){\circle*{.5}}
\put(14,0){\circle*{.5}}
\end{picture}
\,T19
\hfill \makebox{}
\end{center}
\vspace{1em}
\begin{center}
\setlength{\unitlength}{5pt}
\hfill
\begin{picture}(9,10)
\put(2,10){\line(-1,-1){2}}
\put(2,10){\line(1,-1){6}}
\put(4,8){\line(-1,-1){2}}
\put(6,6){\line(-1,-1){2}}
\put(4,4){\line(-1,-2){1}}
\put(4,4){\line(1,-2){2}}
\put(8,4){\line(-1,-2){1}}
\put(8,4){\line(1,-2){1}}
\put(5,2){\line(-1,-2){1}}
\put(2,10){\circle*{.5}}
\put(0,8){\circle*{.5}}
\put(4,8){\circle*{.5}}
\put(2,6){\circle*{.5}}
\put(6,6){\circle*{.5}}
\put(4,4){\circle*{.5}}
\put(8,4){\circle*{.5}}
\put(3,2){\circle*{.5}}
\put(5,2){\circle*{.5}}
\put(7,2){\circle*{.5}}
\put(9,2){\circle*{.5}}
\put(4,0){\circle*{.5}}
\put(6,0){\circle*{.5}}
\end{picture}
\,T20
\hfill \hfill
\begin{picture}(11,10)
\put(2,10){\line(-1,-1){2}}
\put(2,10){\line(1,-1){8}}
\put(4,8){\line(-1,-1){2}}
\put(6,6){\line(-1,-1){2}}
\put(8,4){\line(-1,-1){2}}
\put(4,4){\line(-1,-2){1}}
\put(4,4){\line(1,-2){1}}
\put(10,2){\line(-1,-2){1}}
\put(10,2){\line(1,-2){1}}
\put(2,10){\circle*{.5}}
\put(0,8){\circle*{.5}}
\put(4,8){\circle*{.5}}
\put(2,6){\circle*{.5}}
\put(6,6){\circle*{.5}}
\put(4,4){\circle*{.5}}
\put(8,4){\circle*{.5}}
\put(3,2){\circle*{.5}}
\put(5,2){\circle*{.5}}
\put(6,2){\circle*{.5}}
\put(10,2){\circle*{.5}}
\put(9,0){\circle*{.5}}
\put(11,0){\circle*{.5}}
\end{picture}
\,T21
\hfill \hfill
\begin{picture}(16,10)
\put(1,10){\line(-1,-2){1}}
\put(1,10){\line(2,-1){12}}
\put(5,8){\line(-2,-1){4}}
\put(1,6){\line(-1,-2){1}}
\put(1,6){\line(1,-2){1}}
\put(9,6){\line(-2,-1){4}}
\put(5,4){\line(-1,-1){2}}
\put(5,4){\line(1,-1){2}}
\put(13,4){\line(-1,-1){2}}
\put(13,4){\line(1,-1){2}}
\put(3,2){\line(-1,-2){1}}
\put(3,2){\line(1,-2){1}}
\put(7,2){\line(-1,-2){1}}
\put(7,2){\line(1,-2){1}}
\put(11,2){\line(-1,-2){1}}
\put(11,2){\line(1,-2){1}}
\put(15,2){\line(-1,-2){1}}
\put(15,2){\line(1,-2){1}}
\put(1,10){\circle*{.5}}
\put(0,8){\circle*{.5}}
\put(5,8){\circle*{.5}}
\put(1,6){\circle*{.5}}
\put(9,6){\circle*{.5}}
\put(0,4){\circle*{.5}}
\put(2,4){\circle*{.5}}
\put(5,4){\circle*{.5}}
\put(13,4){\circle*{.5}}
\put(3,2){\circle*{.5}}
\put(7,2){\circle*{.5}}
\put(11,2){\circle*{.5}}
\put(15,2){\circle*{.5}}
\put(2,0){\circle*{.5}}
\put(4,0){\circle*{.5}}
\put(6,0){\circle*{.5}}
\put(8,0){\circle*{.5}}
\put(10,0){\circle*{.5}}
\put(12,0){\circle*{.5}}
\put(14,0){\circle*{.5}}
\put(16,0){\circle*{.5}}
\end{picture}
\,T22
\hfill \makebox{}
\end{center}
\end{minipage}
\end{center}
\end{appendix}

\vskip 15pt

\noindent{\bf Acknowledgement:} Part of the work was done while Huang and Liu were visiting the University of Stuttgart financially supported by the DAAD. They would like to thank their host Professor Steffen K\"onig for the kind hospitality.


\begin{thebibliography}{99}
\bibitem{an}
Andruskiewitsch, Nicol\'as: \emph{On finite-dimensional Hopf algebras}. Proceedings of the International Congress of Mathematicians, Vol. \textbf{2} (Seoul, 2014), 117-142, KYUNG MOON Sa Co. Ltd., Korea.

\bibitem{as1}
Andruskiewitsch, Nicol\'as; Schneider, Hans-J\"urgen: \emph{Lifting of quantum linear spaces and pointed Hopf algebras of order $p^3$}. J. Algebra \textbf{209} (1998), no. 2, 658-691.

\bibitem{as}
Andruskiewitsch, Nicol\'as; Schneider, Hans-J\"urgen: \emph{Finite quantum groups and Cartan matrices}. Adv. Math \textbf{154} (2000), 1-45.

\bibitem{as2}
Andruskiewitsch, Nicol\'as; Schneider, Hans-J\"urgen: \emph{Pointed Hopf algebras}. New directions in Hopf algebras, 1-68, Math. Sci. Res. Inst. Publ., \textbf{43}, Cambridge Univ. Press, Cambridge, 2002.

\bibitem{AS1}
Andruskiewitsch, Nicol\'as; Schneider, Hans-J\"urgen: \emph{On the classification of finite-dimensional pointed Hopf algebras}. Ann. of Math. (2) \textbf{171} (2010), no. 1, 375-417.

\bibitem{A}
Angiono, Iv\'an Ezequiel: \emph{Basic quasi-Hopf algebras over cyclic groups}. Adv. Math. \textbf{225} (2010), no. 6, 3545-3575.

\bibitem{Ang}
Angiono, Iv\'an Ezequiel: \emph{On Nichols algebras of diagonal type}. J. Reine Angew. Math. \textbf{683} (2013), 189-251.


\bibitem{DPR}
Dijkgraaf, Robbert; Pasquier, Vincent; Roche, Philippe: \emph{Quasi Hopf algebras, group cohomology and orbifold models}. \emph{Nuclear Phys. B Proc. Suppl.} \textbf{18}B (1990), 60-72.

\bibitem{D}
Drinfeld, Vladimir G.: \emph{Quantum groups}. Proceedings of the International Congress of Mathematicians, Vol. \textbf{1}, (Berkeley, Calif., 1986), 798-820, Amer. Math. Soc., Providence, RI, 1987.

\bibitem{Dr}
Drinfeld, Vladimir G.: \emph{Quasi-Hopf algebras}. (Russian) Algebra i Analiz \textbf{1} (1989), no. 6, 114-148; translation in Leningrad Math. J. \textbf{1} (1990), no. 6, 1419-1457.

\bibitem{EM}
Eilenberg, Samuel; MacLane, Saunders: \emph{Cohomology theory of Abelian groups and homotopy theory}. I. Proc. Nat. Acad. Sci. U. S. A. \textbf{36}, (1950). 443-447.

\bibitem{EG1}
Etingof, Pavel; Gelaki, Shlomo: \emph{Finite-dimensional quasi-Hopf algebras with radical of codimension 2}. Math. Res. Lett. \textbf{11} (2004), no. 5-6, 685-696.

\bibitem{EG2}
Etingof, Pavel; Gelaki, Shlomo: \emph{On radically graded finite-dimensional quasi-Hopf algebras}. Mosc. Math. J. \textbf{5} (2005), no. 2, 371-378.

\bibitem{EG3}
Etingof, Pavel; Gelaki, Shlomo: \emph{Liftings of graded quasi-Hopf algebras with radical of prime codimension}. J. Pure Appl. Algebra \textbf{205} (2006), no. 2, 310-322.

\bibitem{ENO}
Etingof, Pavel; Nikshych, Dimitri; Ostrik, Victor: \emph{On fusion categories}. Ann. of Math. (2) \textbf{162} (2005), no. 2, 581-642.

\bibitem{EO}
Etingof, Pavel; Ostrik, Victor: \emph{Finite tensor categories.} Mosc. Math. J. \textbf{4} (2004), no. 3, 627-654.

\bibitem{G}
Gelaki, Shlomo: \emph{Basic quasi-Hopf algebras of dimension $n^3$}.  J. Pure Appl. Algebra \textbf{198} (2005), no. 1-3, 165-174.

\bibitem{H0}
Heckenberger, Istvan: \emph{The Weyl groupoid of a Nichols algebra of diagonal type}. Invent. Math. \textbf{164} (2006), no. 1, 175-188.

\bibitem{h}
Heckenberger, Istvan: \emph{Examples of finite-dimensional rank 2 nichols algebras of diagonal type}. Compositio Math. \textbf{143} (2007) 165-190.

\bibitem{h2}
Heckenberger, Istvan: \emph{Rank 2 Nichols algebras with finite arithmetic root system}. Algebr. Represent. Theory \textbf{11} (2008), no. 2, 115-132.

\bibitem{H3}
Heckenberger, Istvan: \emph{Classification of arithmetic root systems}. Adv. Math. \textbf{220} (2009), no. 1, 59-124.

\bibitem{qha1}
Huang, Hua-Lin: \emph{Quiver approaches to quasi-Hopf algebras}. J. Math. Phys. \textbf{50} (4) (2009) 043501, 9pp.

\bibitem{qha2}
Huang, Hua-Lin: \emph{From projective representations to quasi-quantum groups}. Sci. China Math. \textbf{55} (2012), no. 10, 2067-2080.

\bibitem{qha3}
Huang, Hua-Lin; Liu, Gongxiang; Ye, Yu: \emph{Quivers, quasi-quantum groups and finite tensor categories}. Comm. Math. Phys. \textbf{303} (2011), no. 3, 595-612.

\bibitem{bgrc1}
Huang, Hua-Lin; Liu, Gongxiang; Ye, Yu: \emph{ The braided monoidal structures on a class of linear Gr-categories}. Algebr. Represent. Theory \textbf{17} (2014), no. 4, 1249-1265.

\bibitem{bgrc2}
Huang, Hua-Lin; Liu, Gongxiang; Ye, Yu: \emph{On Braided Linear Gr-categories}. arXiv:1310.1529.

\bibitem{qha5}
Huang, Hua-Lin; Liu, Gongxiang; Ye, Yu: \emph{Graded elementary quasi-Hopf algebras of tame representation type}. Israel Math. J., in press. arXiv:1401.6727.

\bibitem{js}
Joyal, Andr\'e; Street, Ross: \emph{Braided tensor categories}. Adv. Math. \textbf{102} (1993), no. 1, 20-78.

\bibitem{majid}
Majid, Shahn: \emph{Algebras and Hopf algebras in braided categories}. Advances in Hopf algebras (Chicago, IL, 1992), 55-105, Lecture Notes in Pure and Appl. Math., \textbf{158}, Dekker, New York, 1994.

\bibitem{MN}
Mason, Geoffrey; Ng, Siu-Hung: \emph{Group cohomology and gauge equivalence of some twisted quantum doubles}. Trans. Amer. Math. Soc. \textbf{353} (2001), no. 9, 3465-3509

\bibitem{Mon}
Montgomery, Susan: {\tt Hopf algebras and their actions on rings}. CBMS Lecture Notes \textbf{82}, Amer. Math. Soc., 1993.


\bibitem{O}
Ostrik, Victor: \emph{Multi-fusion categories of Harish-Chandra bimodules}. Proceedings of the International Congress of Mathematicians, Vol. \textbf{3} (Seoul, 2014), 121-142, KYUNG MOON Sa Co. Ltd., Korea.

%

\end{thebibliography}
\end{document}